\documentclass[12pt,oneside]{book}
\usepackage{amssymb}  
\usepackage{amsthm}   
\usepackage{amsmath}  
\usepackage{dsfont}   
\usepackage{enumerate}
\usepackage{centernot}
\usepackage{mathtools}
\usepackage{color}
\usepackage[*/.style={solid,draw=black,fill=white},
o/.style={solid,draw=black,fill=black},
affine mark=o,
edge length=1.5cm,
ordering=Carter,
root radius=0.1cm,
text style/.style={scale=1.2}
]{dynkin-diagrams}
\usepackage[backend=biber,style=alphabetic,sorting=anyt]{biblatex}
\theoremstyle{definition}
\newtheorem{theorem}{Theorem}[chapter]
\newtheorem{definition}[theorem]{Definition}
\newtheorem{proposition}[theorem]{Proposition}

\newtheorem{lemma}[theorem]{Lemma}
\newtheorem{corollary}{Corollary}[theorem]

\newcommand{\Z}{\ensuremath{\mathbb{Z}}}

\newcommand{\R}{\ensuremath{\mathbb{R}}}
\newcommand{\C}{\ensuremath{\mathbb{C}}}

\newcommand{\laa}{\mathfrak{a}}
\newcommand{\lab}{\mathfrak{b}}

\newcommand{\lag}{\mathfrak{g}}
\newcommand{\lah}{\mathfrak{h}}

\newcommand{\lak}{\mathfrak{k}}

\newcommand{\lan}{\mathfrak{n}}

\newcommand{\lap}{\mathfrak{p}}

\newcommand{\gl}{\mathfrak{gl}}

\newcommand\restrict[1]{\raisebox{-.38ex}{$\big\lvert$}_{#1}}

\DeclareMathOperator{\en}{End}
\DeclareMathOperator{\tr}{tr}

\newcommand{\hs}{_{\textbf{hs}}}
\DeclareMathOperator{\gr}{gr}
\newcommand{\coo}{c^{\Xi_0}_{\Xi_0}}
\newcommand{\cto}{c^{\Xi_2}_{\Xi_0}}
\renewcommand{\cot}{c^{\Xi_0}_{\Xi_2}}
\newcommand{\ctt}{c^{\Xi_2}_{\Xi_2}}
\newcommand{\doo}{c^{\Omega_0}_{\Omega_0}}
\newcommand{\dfo}{c^{\Omega_4}_{\Omega_0}}
\newcommand{\dof}{c^{\Omega_0}_{\Omega_4}}
\newcommand{\dff}{c^{\Omega_4}_{\Omega_4}}

\newcommand{\THFO}{e^{-\Omega_0}(\Theta_{\Omega_0})_{\hs}}
\newcommand{\THFF}{e^{-\Omega_0}(\Theta_{\Omega_4})_{\hs}}
\newcommand{\THGO}{e^{-\Xi_0}(\Theta_{\Xi_0})_{\hs}}
\newcommand{\THGT}{e^{-\Xi_0}(\Theta_{\Xi_2})_{\hs}}
\newcommand{\dcart}[1]{\mathfrak{h}^*_{#1}}
\DeclareMathOperator{\odd}{Odd}
\DeclareMathOperator{\even}{Even}
\DeclareMathOperator{\ch}{ch}
\DeclareMathOperator{\im}{Im}

\DeclareMathOperator{\Span}{Span}
\DeclareMathOperator{\ad}{ad}
\DeclareMathOperator{\rk}{rk}
\DeclareMathOperator{\sgn}{sgn}
\DeclareMathOperator{\mult}{mult}
\DeclareMathOperator{\vir}{Vir}

\pdfoutput=1

\usepackage[letterpaper,bindingoffset=0.0in,%
left=1.5in,right=1in,top=1in,bottom=1in,%
footskip=.25in]{geometry}

\usepackage{setspace}
\usepackage{blindtext}

\makeatletter
\renewcommand*\@makechapterhead[1]{%
	\vspace*{0.75in}%
	{\parindent \z@ \raggedright \normalfont
		\ifnum \c@secnumdepth >\m@ne
		\huge\bfseries \@chapapp\space \thechapter
		\par\nobreak
		\vskip 20\p@
		\fi
		\interlinepenalty\@M
		\Huge \bfseries #1\par\nobreak
		\vskip 40\p@
}}
\renewcommand*\@makeschapterhead[1]{%
	\vspace*{0.75in}%
	{\parindent \z@ \raggedright
		\normalfont
		\interlinepenalty\@M
		\Huge \bfseries  #1\par\nobreak
		\vskip 40\p@
}}
\makeatother

\begin{document}

\frontmatter

\thispagestyle{empty}

\vspace*{1in}

\centerline{BRANCHING RULE DECOMPOSITION OF THE LEVEL-1 $E_8^{(1)}$-MODULE} 
\centerline{WITH RESPECT TO THE IRREGULAR SUBALGEBRA $F_4^{(1)}\oplus G_2^{(1)}$}

\vfill

\centerline{BY}

\vskip 12pt

\centerline{JOSHUA D. CAREY}

\vskip 12pt

\centerline{BS, Marywood University, 2015}
\centerline{MA, Binghamton University State University of New York, 2017}

\vfill

\centerline{DISSERTATION}

\vskip 12pt

\centerline{Submitted in partial fulfillment of the requirements for}
\centerline{the degree of Doctor of Philosophy in Mathematical Sciences}
\centerline{in the Graduate School of}
\centerline{Binghamton University}
\centerline{State University of New York}
\centerline{2022}

\newpage

\thispagestyle{empty}

\vspace*{0in}
\vfill

\centerline{\copyright\ Copyright by Joshua Drew Carey 2022}

\vskip 12pt

\centerline{All Rights Reserved}

\newpage

\vspace*{0in}
\vfill

\centerline{Accepted in partial fulfillment of the requirements for}
\centerline{the degree of Doctor of Philosophy in Mathematical Sciences}
\centerline{in the Graduate School of}
\centerline{Binghamton University}
\centerline{State University of New York}
\centerline{2022}
    \vskip 12pt
\centerline{May 15, 2022}
    \vskip 12pt
\centerline{Alex J. Feingold, Chair and Faculty Advisor}
\centerline{Department of Mathematics and Statistics, Binghamton University}
    \vskip 12pt
\centerline{Fernando Guzman, Member}
\centerline{Department of Mathematics and Statistics, Binghamton University}
    \vskip 12pt
\centerline{Hung Tong-Viet, Member}
\centerline{Department of Mathematics and Statistics, Binghamton University}
    \vskip 12pt
\centerline{Leslie Lander, Outside Examiner}
\centerline{Department of Computer Science, Binghamton University}

\newpage

\chapter*{Abstract}
\begin{doublespacing}
    
Given a Lie algebra of type $E_8$, one can use Dynkin diagram automorphisms of the $E_6$ and $D_4$ Dynkin diagrams to locate a subalgebra of type $F_4\oplus G_2$. These automorphisms can be lifted to the affine Kac-Moody counterparts of these algebras and give a subalgebra of type $F_4^{(1)}\oplus G_2^{(1)}$ within a type $E_8^{(1)}$ Kac-Moody Lie algebra. We will consider the level-1 irreducible $E_8^{(1)}$-module $V^{\Lambda_0}$ and investigate its branching rule, that is how it decomposes as a direct sum of irreducible $F_4^{(1)}\oplus G_2^{(1)}$-modules.

We calculate these branching rules using a character formula of Kac-Peterson which uses theta functions and the so-called ``string functions." We will make use of Jacobi's, Ramanujan's and the Borweins' theta functions (and their respective properties and identities) in our calculation, including some identities involving the Rogers-Ramanujan series. Virasoro character theory is used to verify string functions stated by Kac and Peterson. We also investigate dissections of some interesting $\eta$-quotients.
    
\end{doublespacing}

\newpage

\chapter*{Dedication}
\bigskip
\begin{doublespacing}
\centering
Dedicated to my wife Andrea and our children and my parents Frank and Linda, without whose love and support I would not be here.
\vskip 12pt
Dedicated also in memory of Michael Romanowski who first inspired my love of mathematics.

\end{doublespacing}

\newpage

\chapter*{Acknowledgements}
\begin{doublespacing}
I would first like to thank Bernard Julia for posing a question that led to this project. I would also like to thank Denis Bernard and Jean Thierry-Mieg for helpful comments regarding their paper. I am forever grateful to my advisor, Alex Feingold, for his knowledge, guidance and patience during my years as his student. I am especially grateful for all of the time he has dedicated to helping and encouraging me with this project and to teaching me about Kac-Moody algebras. I am grateful as well to my committee members, Fernando Guzman, Hung Tong-Viet and Leslie Lander for so generously giving of their time and expertise. Finally, I would like to thank all of the professors and fellow graduate students at Binghamton with whom I have had the pleasure of working and learning.

\end{doublespacing}

\newpage

\tableofcontents

\mainmatter


\begin{doublespacing}

\addcontentsline{toc}{chapter}{Introduction}
\chapter*{Introduction}
In this dissertation, we will solve a problem in the representation theory of infinite dimensional affine Kac-Moody Lie algebras. Given an irreducible module, $V^\Lambda$, of an affine algebra, $\widehat{\lag}$, and a subalgebra, $\widehat{\lap}$, branching rules give the decomposition of $V^\Lambda$ as a direct sum of irreducible $\widehat{\lap}$-modules. This sum is often infinite, but not always. When such a decomposition has only a finite number of summands, $\widehat{\lag}$ and $\widehat{\lap}$ are called a conformal pair and we say that $\widehat{\lap}$ is conformal in $\widehat{\lag}$. In the case of a regular conformal pair, i.e. $\rk(\lag) = \rk(\lap)$, graded dimensions of all modules involved can be calculated quickly using Lepowsky's principal specialization formula \cite{Lepowsky}. When the pair is irregular, i.e. $\rk(\lag)\neq \rk(\lap)$, another approach must be taken.

Our main result is a complete proof of the branching decomposition of the level-1 irreducible $E_8^{(1)}$-module 
$$V^{\Lambda_0} = (V^{\Omega_0}\otimes V^{\Xi_0})\oplus(V^{\Omega_4}\otimes V^{\Xi_2})$$
as a direct sum of just two tensor products. The summands are a level-1 $F_4^{(1)}$-module tensored with a level-1 $G_2^{(1)}$-module. 
This result is stated in several places (\cite{KacSan}, \cite{KacWak}, \cite{BTM}) without proof. 
We prove it here by comparing the graded dimensions of the left- and right-hand sides of the previous equation. 
We use a character formula of Kac-Peterson \cite{KacPeterson} to calculate graded dimensions of the level-1 $E_8^{(1)}$-module $V^{\Lambda_0}$, 
the level-1 $F_4^{(1)}$-modules $V^{\Omega_0}$ and $V^{\Omega_4}$ and the level-1 $G_2^{(1)}$ modules $V^{\Xi_0}$ and $V^{\Xi_2}$. 
This formula makes use of theta functions and ``string functions." Utilizing the horizontal specialization allows us to use the theory of theta series of lattices 
to calculate our theta functions. Kac and Peterson state the string functions required, but without proof in the case of all modules but $V^{\Lambda_0}$. 
Bernard and Thierry-Mieg are able to verify the $G_2^{(1)}$ string functions \cite{BTM} and we verify those for $F_4^{(1)}$ using the related language of 
branching functions.

Our calculation will make heavy use of theta functions and theta function identities. Most notably, we make use of 2 of the 40 identities for the Rogers-Ramanujan series from \cite{berndt} to 
simplify products of string functions into sums of  $\eta$-quotients with fractional multiples of $\tau$. These sums, however, contain no fractional powers of 
$q$ due to hidden cancellations that can be revealed by dissections of the individual quotients.

It should be noted that when investigating a branching rule where $\widehat{\lag}$ and $\widehat{\lap}$ are not a conformal pair, the ``coset Virasoro" 
Goddard-Kent-Olive construction \cite{GKO} plays a vital role because Virasoro modules count the multiplicities of $\widehat{\lap}$ summands. 
In the case of a conformal pair, the Virasoro modules are not helpful because the central charge of the coset Viraroro is $0$.  
Even though we could not make use of the coset Virasoro theory in our main calculation, it was necessary to verify the $F_4^{(1)}$ string functions. 
This verification will utilize 4 more of the 40 identities from \cite{berndt}.

\chapter{Background}
In this first chapter, we give an introduction to information on finite dimensional Lie algebras and infinite dimensional affine Kac-Moody Lie algebras. We conclude by constructing projections which identify the subalgebra with which we are concerned.


\section{Finite Dimensional Lie Algebras}
In this expository section we present the necessary background information on finite dimensional Lie algebras. This theory is classical and what is presented here can be found in any introductory text on Lie algebras such as \cite{humphreys} or \cite{carter}.
\begin{definition}
A \textbf{Lie algebra} is a vector space, $\lag$, over a field, $F$, together with a bilinear operation $[-,-]\colon\lag\times\lag\rightarrow\lag$ called the \textbf{Lie bracket} such that
\begin{enumerate}[(L1)]
    \item $[x,x] = 0$ for all $x\in\lag$,
    \item $[x,[y,z]]+[y,[z,x]]+[z,[x,y]] = 0$ for all $x,y,z\in\lag.$
\end{enumerate}
\end{definition}
Condition (L2) is referred to as the \textbf{Jacobi Identity}. Note that (L1) and bilinearity together imply that the Lie bracket is anticommutative, i.e. $[x,y] = -[y,x]$. Anticommuntativity and bilinearity will only imply (L1) in the case that we are working over a field of characteristic other than 2. Most of what follows in this section will hold for Lie algebras over any field of characteristic 0, but we will only consider complex Lie algebras.
\begin{definition}
For $\laa,\lab$ subsets of a Lie algebra $\lag$, define $[\laa,\lab]\coloneqq\Span\{[a,b]\mid a\in\laa, b\in\lab\}$. Then any subspace $\lah$ of $\lag$ such that $[\lah,\lah]\subseteq\lah$ is called a \textbf{Lie subalgebra} of $\lag$. Further, if $[\lag,\lah]\subseteq\lah$, then $\lah$ is called an \textbf{ideal} of $\lag$. If $\lah\subsetneq\lag$ is an ideal, we call it a \textbf{proper ideal}.
\end{definition}
Every Lie algebra has two ideals, namely the trivial ideal $\{0\}$ and the entire Lie algebra itself.
\begin{definition}
A Lie algebra, $\lag$, is called \textbf{simple} when it has no nontrivial, proper ideals and $[\lag,\lag] = \lag$. A Lie algbera is called \textbf{semisimple} if it can be expressed as a direct sum of simple Lie algebras.
\end{definition}
\begin{definition}
A Lie algebra, $\lag$, is called \textbf{abelian} if $[x,y] = 0$ for all $x,y\in\lag$.
\end{definition}
This leads us to the simplest example of a Lie algebra: Any vector space, $V$, with bracket defined by $[w,v] = 0$ for all $w,v\in V$ is trivially an abelian Lie algebra. We can also define a Lie algebra given any associative algebra, $(A,\circ)$, by giving it the bracket $[a,b]\coloneqq a\circ b-b\circ a$ for all $a,b\in A$. Condition (L1) is clearly satisfied; The associativity of $\circ$ will then give the Jacobi identity. One such example of this is the following:
\begin{definition}
Given a vector space, $V$, define $\gl(V)$ as the associative algebra $\en(V)$ with Lie bracket $[f,g] = f\circ g-g\circ f$ for any $f,g\in\en(V)$. Similarly, let $\gl(n)$ be the Lie algebra of $n\times n$ complex matrices with bracket defined by $[A,B] = AB-BA$.
\end{definition}
These two Lie algebras are important to the area of representation theory. First we will need the following:
\begin{definition}
Given two Lie algebras, $\lag$ and $\lah$, we define a \textbf{homomorphism} as a linear map $\phi\colon\lag\rightarrow\lah$ such that $\phi([x,y]_\lag) = [\phi(x),\phi(y)]_\lah$ where the subscript on the bracket denotes which algebra's bracket is being taken.
\end{definition}
Such a map, $\phi$, is called a monomorphism, epimorphism, endomorphism, isomorphism or automorphism according to its properties.
\begin{definition}
Given a Lie algebra $\lag$ and a vector space $V$, a \textbf{representation} of $\lag$ on $V$ is a homomorphism $\phi\colon\lag\rightarrow\gl(V)$. If $\phi$ is injective, we say that the representation is \textbf{faithful}.
\end{definition}
\begin{definition}
Given a Lie algebra $\lag$ and a vector space $V$, we call $V$ a \textbf{$\lag$-module} when there exists a bilinear action $\cdot\colon\lag\times V\rightarrow V$ such that $[x,y]\cdot v = x\cdot(y\cdot v)-y\cdot(x\cdot v)$ for all $x,y\in\lag$ and all $v\in V$. A subspace $W$ of $V$ is called a \textbf{submodule} when $x\cdot w\in W$ for all $x\in\lag$ and all $w\in W$. If the only submodules of $V$ are $\{0\}$ and $V$, it is called \textbf{irreducible}. A $\lag$-module that can be expressed as a direct sum of irreducible $\lag$-modules is called \textbf{completely reducible}.
\end{definition}
Note that given a representation $\phi\colon\lag\rightarrow\gl(V)$, we can view $V$ as a $\lag$-module under the action $x\cdot v = \phi(x)(v)$. Similarly, given a $\lag$-module, $V$, we can construct a representation $\phi\colon\lag\rightarrow\gl(V)$ via $x\mapsto \phi_x$ where $\phi_x(v) = x\cdot v$.
We now give an important example of a Lie algebra representation. Given a Lie algebra $\lag$ define $\ad_x$ by $\ad_x(y) = [x,y]$ for all $x,y\in\lag$. The \textbf{adjoint representation} is given by $\ad\colon\lag\rightarrow\gl(\lag)$ with $\ad(x) = \ad_x$. 
\begin{definition}
Let $\lag$ be a Lie algebra and $\kappa(-,-)\colon\lag\times\lag\rightarrow\C$ a bilinear form defined by $\kappa(x,y) = \tr(\ad(x)\circ\ad(y))$ for all $x,y\in\lag$. We call this form the \textbf{Killing form}.
\end{definition}
The Killing form is invariant (i.e. $\kappa([x,y],z) = \kappa(x,[y,z])$) and symmetric. Further, if a Lie algebra is finite dimensional and simple, then its Killing form is the unique (up to scalar multiple) symmetric, invariant bilinear form. 
\begin{theorem}[Cartan's Criterion]
A finite dimensional Lie algebra, $\lag$, is semisimple if and only if its Killing form is non-degenerate.
\end{theorem}
The non-degeneracy of $\kappa$ on a semisimple Lie algebra $\lag$ gives a natural isomorphism between the underlying vector space $\lag$ and its dual vector space $\lag^*$. For the remainder of this section, we assume that our Lie algebras are semisimple unless otherwise noted.
\begin{definition}
A \textbf{Cartan subalgebra} (CSA) of a Lie algebra $\lag$ is a maximal abelian subalgebra $\lah$ of $\lag$ such that $\{\ad_h\mid h\in\lah\}$ is simultaneously diagonalizable in $\en(\lag)$.
\end{definition}
If a Cartan subalgebra exists, it is not unique. Finite dimensional simple Lie algebras always have CSAs and they are all conjugate under the group of inner automorphisms. This result immediately extends to semisimple Lie algebras by taking direct sums of the corresponding CSAs. This also gives us that the dimension of the CSA will be the same regardless of choice of CSA. This leads us to define the \textbf{rank} of a Lie algebra $\lag$ by $\rk(\lag) = \dim(\lah)$, where $\lah$ is a CSA of $\lag$ (note that we will use $\ell$ to represent the rank of the semisimple Lie algebra with which we are working throughout this work).

Let $\lag$ be a finite dimensional Lie algebra of rank $\ell$. Fix a CSA, $\lah$, of $\lag$ and choose a basis $h_1,\ldots,h_\ell$. Because $\{\ad_{h_i}\mid1\leq i\leq\ell\}$ is simultaneously diagonalizable in $\en(\lag)$, we can decompose $\lag$ into a direct sum of simultaneous eigenspaces for the $\ad_{h_i}$'s.
\begin{definition}
Let $0\neq\alpha\in\lah^*$ and define $\lag_\alpha = \{x\in\lag\mid[h,x] = \alpha(h)x\text{ for all }h\in\lah\}$. If $\lag_\alpha\neq \{0\}$, we call $\lag_\alpha$ a \textbf{root space} of $\lag$ and say $\alpha$ is a \textbf{root} of $\lag$. The set of roots is denoted by $\Phi$. The decomposition $$\lag = \lah\oplus\bigoplus_{\alpha\in\Phi}\lag_\alpha$$ is called a \textbf{root space decomposition} of $\lag$ with respect to $\lah$.
\end{definition}
Although $0\not\in\Phi$ by definition, $\lah = \lag_0$. We list some properties of the root spaces and roots.
\begin{proposition}
Let $\alpha,\beta\in\Phi\cup\{0\}$. Then:
\begin{enumerate}[(1)]
    \item $[\lag_\alpha,\lag_\beta]\subseteq\lag_{\alpha+\beta}$ with equality holding when $\alpha+\beta\in\Phi$, so $[\lag_\alpha,\lag_{-\alpha}]\subseteq\lah$,
    \item If $x\in\lag_\alpha$ and $y\in\lag_{\beta}$ with $\beta\neq -\alpha$ then $\kappa(x,y) = 0$,
    \item $\kappa\restrict{\lah\times\lah}$ is non-degenerate,
    \item If $\alpha\in\Phi$ then $-\alpha\in\Phi$.
\end{enumerate}
\end{proposition}
Because $\kappa$ is nondegenerate on $\lah$, it induces an isomorphism between $\lah$ and $\lah^*$ which allows us to identify $\lah$ with $\lah^*$. For any $\alpha\in\Phi$, the corresponding vector in $\lah$ (which we denote by $h_\alpha$) is uniquely determined by the condition $\kappa(h_\alpha,h) = \alpha(h)$ for all $h\in\lah$. This will induce a bilinear form on $\lah^*$ defined by $(\alpha,\beta)\coloneqq\kappa(h_\alpha,h_\beta)$ for all $\alpha,\beta\in\lah^*$. We have further properties:
\begin{proposition}
Let $\alpha,\beta\in\Phi$.
\begin{enumerate}[(1)]
    \item If $x\in\lag_\alpha$ and $y\in\lag_{-\alpha}$ then $[x,y] = \kappa(x,y)h_\alpha$,
    \item $\dim(\lag_\alpha) = 1$,
    \item $[\lag_\alpha,\lag_{-\alpha}] = \C h_\alpha$,
    \item $2\alpha\not\in\Phi$.
\end{enumerate}
\end{proposition}
\begin{definition}
Let $\R^\ell$ be the Euclidean space with inner product denoted by $(-,-)$ and let $\Phi\subseteq\R^\ell$. Then $\Phi$ is called a \textbf{root system} if the following are satisfied:
\begin{enumerate}[(R1)]
    \item $\Phi$ is finite, spans $\R^\ell$ and does not contain 0,
    \item if $\alpha\in\Phi$ then $k\alpha\in\Phi$ if and only if $k = \pm 1$,
    \item if $\alpha\in\Phi$ then $s_\alpha(\Phi) = \Phi$ where $s_\alpha(\beta) = \beta-\dfrac{2(\beta,\alpha)}{(\alpha,\alpha)}\alpha$ (i.e. reflection through the hyperplane perpendicular to $\alpha$),
    \item if $\alpha,\beta\in\Phi$ then $\langle\beta,\alpha\rangle\coloneqq\dfrac{2(\beta,\alpha)}{(\alpha,\alpha)}\in\Z$.
\end{enumerate}
We call the elements of $\Phi$ \textbf{roots}.
\end{definition}
We have that the set of roots of a finite dimensional, simple Lie algebra is a root system.
\begin{definition}
A subset $\Delta$ of a root system $\Phi$ is called a \textbf{base} when the following hold:
\begin{enumerate}[(B1)]
    \item $\Delta$ is a basis of $\R^\ell$,
    \item each root $\beta\in\Phi$ can be written in the form $\beta = \sum_{\alpha\in\Delta}k_\alpha\alpha$ with the $k_\alpha$'s all nonnegative or all nonpositive integers.
\end{enumerate}
We call the elements of $\Delta$ \textbf{simple roots} and note that choice of $\Delta$ is not unique for a given root system.
\end{definition}
We denote the set of \textbf{positive roots} (relative to a choice of $\Delta = \{\alpha_1,\ldots,\alpha_\ell\}$) by $\Phi^+ = \{\alpha\in\Phi\mid\alpha = k_1\alpha_1+\cdots+k_\ell\alpha_\ell,0\leq k_i\in\Z\}$. We can then define the set of \textbf{negative roots} as $\Phi^- = -\Phi^+$ and notice that $\Phi = \Phi^+\cup\Phi^-$.
\begin{definition}
The \textbf{root lattice} of a root system $\Phi$ is the integral span of the simple roots and is denoted by $Q$. We also define $Q^+ = \{\sum_{i = 1}^\ell k_i\alpha_i\mid0\leq k_i\in\Z\}$. We define a partial order on $\R^\ell$ by $\alpha\leq\beta$ when $\beta-\alpha\in Q^+$.
\end{definition}
\begin{definition}
Given a root system $\Phi$ with base $\Delta$, the \textbf{Weyl Group}, $\mathcal{W}$, is the group generated by the $s_{\alpha_i}$, that is $\mathcal{W} = \langle s_{\alpha}\mid\alpha\in\Delta\rangle$.
\end{definition}
Root systems aid in the complete classification of the finite dimensional simple (and by extension, semisimple) Lie algebras. By the Cartan-Killing Theorem, all finite dimensional simple Lie algebras can be classified as a member of one of four infinite families or one of five exceptional types. We can give these classifications in two ways.
\begin{definition}
Given a Lie algebra $\lag$ with root system $\Phi$ and base $\Delta$, we define the \textbf{Cartan Matrix} to be the $\ell\times\ell$ matrix, denoted $A$, with $A = [\langle\alpha_i,\alpha_j\rangle]$ where $\Delta = \{\alpha_1,\ldots,\alpha_\ell\}$.
\end{definition}
Since the roots of finite dimensional simple Lie algebras can have at most two different squared lengths, the Cartan matrix will give the corresponding Lie algebra once a choice of length for either the long or short roots is made. A different but equivalent way to classify these Lie algebras is by using \textbf{Dynkin diagrams}. These diagrams are constructed by drawing $\ell$ nodes representing $\alpha_1,\ldots,\alpha_\ell$ in $\Delta$. The number of edges connecting nodes $i$ and $j$ is equal to $A_{ij}A_{ji}$ (where $A_{ij}$ is the $ij^{\text{th}}$-entry of $A$). If two nodes are connected by more than one edge, an arrow is added pointing to the root whose squared length is shorter. We list these diagrams for the 4 infinite families and 5 exceptional types in tables \ref{DynkinABCD} and \ref{DynkinEFG} below.

\begin{table}[h]
\centering
\begin{tabular}{rcl}
$A_\ell$  &  ($\ell\geq 1$)  &  $\dynkin[text style/.style={scale=1.2}, labels = {1,2,\ell-1,\ell}, label]A{}$\\ \\
$B_\ell$  &  $(\ell\geq 2)$  &  $\dynkin[text style/.style={scale=1.2}, labels = {1,2,\ell-2,\ell-1,\ell}, label]B{}$\\ \\
$C_\ell$  &  $(\ell\geq 3)$  &  $\dynkin[text style/.style={scale=1.2}, labels = {1,2,\ell-2,\ell-1,\ell}, label]C{}$\\ \\
$D_\ell$  &  $(\ell\geq 4)$  &  $\dynkin[text style/.style={scale=1.2}, labels = {1,2,\ell-3,\ell-2,\ell-1,\ell}, label directions={,,,right,,}, label]D{}$\\ \\
\end{tabular}   
\caption{Dynkin diagrams for the infinite families $A_\ell,B_\ell,C_\ell$ and $D_\ell$}
\label{DynkinABCD}
\end{table}

\begin{table}[h]
\centering
\begin{tabular}{rcl}
$E_6$  &  &  \dynkin[text style/.style={scale=1.2}, label]E6\\ \\
$E_7$  &  &  \dynkin[backwards=true, text style/.style={scale=1.2}, label]E7\\ \\
$E_8$  &  &  \dynkin[backwards=true, text style/.style={scale=1.2}, label]E8\\ \\
$F_4$  &  &  \dynkin[text style/.style={scale=1.2}, label]F4\\ \\
$G_2$  &  &  \dynkin[text style/.style={scale=1.2}, labels = {2,1}, label]G2
\end{tabular}   
\caption{Dynkin diagrams for the exception types: $E_6, E_7, E_8, F_4$ and $G_2$}
\label{DynkinEFG}
\end{table}

\begin{definition}
Let $\lag$ be a Lie algebra with CSA $\lah$ and let $V$ be a finite dimensional $\lag$-module such that $\lah$ acts on $V$ simultaneously diagonalizably. We denote the simultaneous eigenspaces by $V_\mu = \{v\in V\mid h\cdot v = \mu(h)v\text{ for all }h\in\lah\}$ where $\mu\in\lah^*$. If $V_\mu\neq\{0\}$ we call $V_\mu$ a \textbf{weight space} and $\mu$ a \textbf{weight} of $V$. We denote the set of weights of $V$ by $\Pi(V)$. A nonzero vector $v\in V$ is called a \textbf{highest weight vector} (HWV) when $v\in V_\lambda$ for some $\lambda\in\Pi(V)$ and $x\cdot v = 0$ for all $x\in\lag_\alpha$ with $\alpha\in\Phi^+$.
\end{definition}
If $V$ is irreducible then it contains a unique HWV up to scalar multiple and $V$ is uniquely determined by the weight of its HWV. Such a module is called a \textbf{highest weight module} (HWM) and is denoted by $V^\lambda$ where $\lambda$ is the weight of its HWV. We will also use $\Pi^\lambda$ to denote $\Pi(V^\lambda)$ and note that if $\mu\in\Pi^\lambda$ then $\mu\leq\lambda$ in the partial order defined above.
\begin{definition}
The \textbf{weight space decomposition} of a HWM, $V^\lambda$ is the decomposition $$ V^\lambda = \bigoplus_{\mu\in\Pi^\lambda}V^\lambda_\mu.$$
\end{definition}
Note that if $\lag$ is a finite dimensional simple Lie algebra, $\lag$ is an irreducible $\lag$-module under the action of $\ad$. Then the weight space decomposition is simply the root space decomposition and its highest weight is called the \textbf{highest root} and denoted by $\theta$. Then $\Pi^\theta = \Phi\cup\{0\}$.
\begin{definition}
For a root system $\Phi$ with base $\Delta$, we define the associated \textbf{fundamental weights} $\lambda_1,\ldots,\lambda_\ell$ by $\langle\lambda_i,\alpha_j\rangle = \delta_{i,j}$ for $1\leq i,j\leq\ell$. The \textbf{weight lattice} associated with $\Phi$ is $P = \{\lambda\in\lah^*\mid\langle\lambda,\alpha_i\rangle\in\Z,\alpha_i\in\Delta\} = \Z\lambda_1+\ldots+\Z\lambda_\ell$. The set of \textbf{dominant integral weights} associated with $\Phi$ is $P^+ = \{\sum_{i = 1}^\ell k_i\lambda_i\mid 0\leq k_i\in\Z\}$. 
\end{definition}
For every $\lambda\in P^+$, there is a finite dimensional irreducible module $V^\lambda$ of $\lag$ with highest weight $\lambda$ and $\Pi^\lambda = \{w(\mu)\mid\mu\in P^+, \mu\leq\lambda,w\in\mathcal{W}\}\subseteq P$. Further, for any finite dimension irreducible $\lag$-module, $V$, there exists $\lambda\in P^+$ such that $V = V^\lambda$ and $w(\Pi^\lambda) = \Pi^\lambda$ for all $w\in\mathcal{W}$.
\begin{definition}
For a highest weight $\lag$-module, $V^\lambda$ we define the \textbf{character} of $V^\lambda$ by $$\ch(V^\lambda) = \sum_{\mu\in\Pi^\lambda}\dim(V^\lambda_\mu)e^\mu.$$
\end{definition}
The character is an element of the group ring $\Z[P]$. The character of a direct sum of $g$-modules is the sum of the characters of the individual modules and likewise the character of a tensor product of modules is the product of the characters of the individual modules. 
\begin{theorem}[Weyl Character Formula]
Let $\lambda\in P^+$. Then $$\ch(V^\lambda) = \dfrac{S_{\lambda+\rho}}{S_\rho}$$ where $$S_\mu = \displaystyle\sum_{w\in\mathcal{W}}\sgn(w)e^{w(\mu)-\rho}\text{ and }\rho = \displaystyle\sum_{i = 1}^\ell\lambda_i.$$
\end{theorem}
\begin{theorem}[Weyl Denominator Formula]
$$S_{\rho} = \prod_{\alpha\in\Phi^+}(1-e^{-\alpha}).$$
\end{theorem}


\section{Infinite Dimensional Affine Lie Algebras}
\begin{definition}
\label{AffineDef}
Let $\lag$ be a finite dimensional simple Lie algebra with non-degenerate, symmetric, invariant bilinear form $(-,-)$. Let $\C[t,t^{-1}]$ be the associative algebra of Laurent polynomials in the variable $t$. The (untwisted) \textbf{affine Kac-Moody Lie algebra} $\widehat{\lag}$ associated with $\lag$ is $$\widehat{\lag} = \lag\otimes_\C\C[t,t^{-1}]\oplus\C c\oplus\C d$$ with brackets given by $$[x(m),y(n)] = [x,y](m+n)+m\delta_{m,-n}(x,y)c,$$ $$[c,\widehat{g}] = 0\text{ and }[d,x(m)]= mx(m)$$
where $x(m)$ denotes $x\otimes t^m$ for any $x\in\lag$ and $m\in\Z$.
\end{definition}
We will often refer to these algebras simply as ``affine (Lie) algebras." \begin{definition}
For $\lah$ the CSA of $\lag$, the \textbf{Heisenberg algebra} is $$\widehat{\lah} = \lah\otimes_\C\C[t,t^{-1}]\oplus\C c$$ with brackets
$$[h_i(m),h_j(n)] = m\delta_{m,-n}(h_i,h_j)c$$ and $$[\widehat{h},c] = 0.$$
\end{definition}
The Heisenberg algebra plays an important role in the representation theory of $\widehat{g}$. Note that we can identify $\lag$ with $\lag\otimes t^0$ in $\widehat{\lag}$. Just as in the finite dimensional simple case, $\widehat{\lag}$ will have a CSA, $H = \lah\oplus\C c\oplus\C d$. We therefore also get a root space decomposition, $$\widehat{g} = H\oplus\bigoplus_{\alpha\in\widehat{\Phi}}\widehat{\lag}_\alpha$$ where $\widehat{\Phi}$ denotes the \textbf{affine root system}. We will define the multiplicity of a root by $\mult(\alpha) = \dim(\widehat{\lag}_\alpha)$. By this definition, all roots of a finite dimensional simple Lie algebra have multiplicity 1. We will construct $\widehat{\Phi}$ by using our identification to extend the simple roots $\alpha_i$ of $\lag$ to be elements of $H^*$. We then define $\Lambda_0,\delta\in H^*$ such that $\Lambda_0(\lah) = \delta(\lah) = 0$, $\Lambda_0(d) = \delta(c) = 0$ and $\delta(d) = \Lambda_0(c) = 1$. We then get that $\widehat{\Phi} = \{\alpha+n\delta\mid\alpha\in\Phi,n\in\Z\}\cup\{n\delta\mid0\neq n\in\Z\}$. The two sets in this union are called the \textbf{real} and \textbf{imaginary} roots, respectively. The simple roots, $\widehat{\Delta}$, of $\widehat{\lag}$ are $\widehat{\Delta} = \{\alpha_0,\alpha_1,\ldots,\alpha_\ell\}$ where $\alpha_0\coloneqq \delta-\theta$ ($\theta$ is the highest root of $\Phi$ as above). We can extend the inner product on $\lah$ to $H$ so that $(\lah,c) = (\lah,d) = 0$, $(c,c) = (d,d) = 0$ and $(c,d) = 1$. Similarly we will extend the inner product on $\lah^*$ to $H^*$ so that $(\lah^*,\Lambda_0) = (\lah^*,\delta) = 0$, $(\Lambda_0,\Lambda_0) = (\delta,\delta) = 0$ and $(\Lambda_0,\delta) = 1$. We can also extend our partial order on $\lah^*$ to $H^*$ by $\mu\leq\lambda$ when $\lambda-\mu = \sum_{i = 0}^\ell k_i\alpha_i$ with $0\leq k_i\in\Z$.
\begin{definition}
The \textbf{fundamental weights} of $\widehat{\lag}$ are defined by analogy with the finite dimensional simple case. They are either of the form
$\Lambda_0$ or $\Lambda_i = n_i\Lambda_0+\lambda_i$ for $i = 1,\ldots,\ell$ where the $\lambda_i$ are fundamental weights of $\lag$. For $i = 1,\ldots,\ell$, we call $n_i$ the \textbf{level} of $\Lambda_i$ and in the case of $\Lambda_0$, we say that its level is 1.
\end{definition}
As in the finite dimensional case, we have a \textbf{root lattice}, \textbf{weight lattice} and set of \textbf{dominant integral weights}, defined respectively as:
$$\widehat{Q} = \Z\alpha_0+\Z\alpha_1+\ldots+\Z\alpha_\ell$$
$$\widehat{P} = \Z\Lambda_0+\Z\Lambda_1+\ldots+\Z\Lambda_\ell$$ and
$$\widehat{P}^+ = \left\{\sum_{i = 0}^\ell k_i\Lambda_i\mid 0\leq k_i\in\Z\right\}.$$
\begin{definition}
The \textbf{affine Weyl group} $\widehat{\mathcal{W}}$ is the group generated by simple reflections $s_{\alpha_i}$ for $\alpha_i\in\widehat{\Delta}$.
\end{definition}
\begin{definition}
Let $\widehat{\lag}$ be an affine algebra with underlying finite dimensional simple Lie algebra $\lag$ and finite root system $\Phi$. Fix a base $\Delta$ of $\Phi$. The \textbf{dual root lattice}, $M^*$ is defined as $M^* = \Z\{\alpha\mid \alpha\in\Delta\text{ and }\alpha\text{ is long}\}+k\Z\{\alpha\mid \alpha\in\Delta\text{ and }\alpha\text{ is short}\}$ where $k$ is the ratio of squared lengths of a long root to squared length of a short root.
\end{definition}
For each $\alpha\in M^*$, let $$t_\alpha(\lambda)\coloneqq\lambda+\lambda(c)\alpha - \left((\lambda,\alpha)+\dfrac{1}{2}(\alpha,\alpha)\lambda(c)\right)\delta$$ for all $\lambda\in\mathfrak{h}^*$ and define $t(M^*) = \{t_\alpha\mid\alpha\in M^*\}$. Note first that $t_\alpha\circ t_\beta = t_{\alpha+\beta} = t_{\beta+\alpha} = t_{\beta}\circ t_{\alpha}$, $t_0(\lambda) = \lambda$ and $t_\alpha\circ t_{-\alpha} = t_0$ for all $\alpha,\beta\in M^*$ and $\lambda\in H^*$. Thus, $t(M^*)$ is an abelian group. In fact, as abelian groups, $t(M^*)\cong M^*$. Since $\mathcal{W}(M^*) = M^*$ we can define an action of $\mathcal{W}$ on $t(M^*)$ by $wt_\alpha w^{-1} = t_{w(\alpha)}$ for $w\in\mathcal{W}$ and $\alpha\in M^*$. We can then express the affine Weyl group as: $\widehat{\cal W} = \mathcal{W}\ltimes t(M^*)$.
\begin{definition}
For each $\Lambda\in\widehat{P}^+$, there exists an irreducible highest weight $\widehat{\lag}$-module, $V^\Lambda$. The \textbf{level} of the representation is $\Lambda(c)$ and is the scalar by which $c$ acts on $V^\Lambda$. The module $V^\Lambda$ has a weight space decomposition $$V^\Lambda = \bigoplus_{\mu\in\Pi^\Lambda}V_{\mu}^\Lambda$$ where, as before, $V_\mu^\Lambda = \{v\in V^\Lambda\mid h\cdot v = \mu(h)v\text{ for all }h\in H\}$ for $\mu\in H^*$ and $\Pi^\Lambda = \{\mu\in\widehat{P}\mid V_\mu^\Lambda\neq 0\} = \{w(\mu)\in\widehat{P}\mid\mu\in\widehat{P}^+,\mu\leq\Lambda,w\in\widehat{\mathcal{W}}\}$.
\begin{definition}
The \textbf{character} of $V^\Lambda$ is defined by $$\ch(V^\Lambda) = \sum_{\mu\in\Pi^\Lambda}\dim(V_\mu^\Lambda)e^\mu.$$ In order to convert this infinite sum of terms from $\Z[\widehat{P}]$ into a power series in $\Z[e^{-\alpha_0},\ldots,e^{-\alpha_\ell}]$, we define the \textbf{graded dimension} of $V^\Lambda$ as $\gr(V^\Lambda) = e^{-\Lambda}\ch(V^\Lambda)$.
\end{definition}
\begin{theorem}[Weyl-Kac Character Formula]
Let $\Lambda\in\widehat{P}^+$. Then $$\ch(V^\Lambda) = \frac{S_{\Lambda+\widehat{\rho}}}{S_{\widehat{\rho}}}$$ where $S_\mu = \displaystyle\sum_{w\in\widehat{\mathcal{W}}}\sgn(w)e^{w\mu-\widehat{\rho}}$ and $\widehat{\rho} = \displaystyle\sum_{i = 0}^\ell\lambda_i$.
\end{theorem}
\begin{theorem}[Weyl-Kac Denominator Formula]
$$S_{\widehat{\rho}} = \prod_{\alpha\in\widehat{\Phi}^+}(1-e^{-\alpha})^{\mult(\alpha)}.$$
\end{theorem}
Each affine algebra has a generalized Cartan matrix and an extended Dynkin diagram. Below we list the extended Dynkin diagrams for five affine algebras we will need to perform our calculations. Note that a black node corresponds to $\alpha_0$ and that the labeling $X_n^{(1)}$ denotes an affine algebra whose associated finite dimensional simple Lie algebra is of type $X_n$.
\end{definition}

\begin{table}[h]
\centering
\begin{tabular}{rc}
$D_4^{(1)}$  &  \dynkin[text style/.style={scale=1.2}, label]D[1]4\\ \\
$E_6^{(1)}$  &  \dynkin[text style/.style={scale=1.2}, label]E[1]6\\ \\
$E_8^{(1)}$  &  \dynkin[backwards=true, text style/.style={scale=1.2}, label]E[1]8\\ \\
$F_4^{(1)}$  &  \dynkin[text style/.style={scale=1.2}, label]F[1]4\\ \\
$G_2^{(1)}$  &  \dynkin[text style/.style={scale=1.2}, labels = {0,2,1}, label]G[1]2
\end{tabular}   
\caption{Dynkin diagrams of some affine algebras}
\label{DynkinAff}
\end{table}


\section{Root Systems and Projections}
We now show that a Lie algebra of type $E_8$ will contain a subalgebra of type $F_4+G_2$. The theory of finite dimensional semisimple Lie algebras discussed in the first section of this chapter allows us to obtain this result by showing that the dual of the CSA of a type $E_8$ algebra will contain a subspace which is the dual of the CSA of a type $F_4+G_2$ algebra. This is accomplished by constructing explicit orthogonal projections on the real span of the simple roots. Although our algebras are all taken over $\C$, we momentarily work over $\R$ in order to preserve the positive definite bilinear form on the roots and thus preserve the geometry of the root systems. These results can be lifted to the full dual spaces by taking complex linear combinations. Our results will then be extended to the corresponding affine algebras. 

We will begin this section by giving some more information about the roots and weights of the algebras with which we are concerned. We first list the Cartan matrices for Lie algebras of types $E_8$, $E_6$, $D_4$, $F_4$ and $G_2$ (note that these can quickly be recovered from the Dynkin diagrams). We then use the standard convention that long roots have squared length 2 and recover the inner product matrix, that is the symmetric matrix $B$ with entries $B_{ij} = (\alpha_i,\alpha_j)$. We also give the fundamental weights in terms of the simple roots and give diagram automorphisms of some of the Dynkin diagrams.

\subsection{Finite Dimensional Simple Lie Algebras}
\subsubsection{Type $E_8$}
We will denote our simple roots by $\alpha_1,\ldots,\alpha_8$.
\subsubsection*{Cartan Matrix}
$$A_{E_8} = 
\begin{pmatrix*}[r]
2  & -1 & 0  & 0  & 0  & 0  & 0  & 0  \\
-1 & 2  & -1 & 0  & 0  & 0  & 0  & 0  \\
0  & -1 & 2  & -1 & 0  & 0  & 0  & 0  \\
0  & 0  & -1 & 2  & -1 & 0  & 0  & 0  \\
0  & 0  & 0  & -1 & 2  & -1 & -1 & 0  \\
0  & 0  & 0  & 0  & -1 & 2  & 0  & 0  \\
0  & 0  & 0  & 0  & -1 & 0  & 2  & -1 \\
0  & 0  & 0  & 0  & 0  & 0  & -1 & 2 
\end{pmatrix*}
$$
\subsubsection*{Inner Product Matrix}
Since there is only one root length, the $ij^{\text{th}}$-entry of $A_{E_8}$ is $$2(\alpha_i,\alpha_j)/|\alpha_j|^2 = 2(\alpha_i,\alpha_j)/2 = (\alpha_i,\alpha_j)$$ and so $B_{E_8} = A_{E_8}$.

\subsubsection*{Fundamental Weights}
We use the notation $(a_1,\ldots,a_8)\coloneqq a_1\alpha_1+\cdots+a_8\alpha_8$. Our fundamental weights are:
\begin{align*}
\lambda_1 &= (2,3,4,5,6,3,4,2),&\lambda_2 &= (3,6,8,10,12,6,8,4),\\
\lambda_3 &= (4,8,12,15,18,9,12,6),&\lambda_4 &= (5,10,15,20,24,12,16,8),\\
\lambda_5 &= (6,12,18,24,30,15,20,10),&\lambda_6 &= (3,6,9,12,15,8,10,5),\\
\lambda_7 &= (4,8,12,16,20,10,14,7),&\lambda_8 &= (2,4,6,8,10,5,7,4).
\end{align*}
Note that the weight lattice and root lattice of type $E_8$ are the same, $P_{E_8} = Q_{E_8}$.
\subsubsection{Type $E_6$}
We will denote our simple roots by $\alpha'_1,\ldots,\alpha'_6$.
\subsubsection*{Cartan Matrix}
$$A_{E_6} = 
\begin{pmatrix*}[r]
2  & -1 & 0  & 0  & 0  & 0  \\
-1 & 2  & -1 & 0  & 0  & 0  \\
0  & -1 & 2  & -1 & -1 & 0  \\
0  & 0  & -1 & 2  & 0  & 0  \\
0  & 0  & -1 & 0  & 2  & -1 \\
0  & 0  & 0  & 0  & -1 & 2 
\end{pmatrix*}
$$

\subsubsection*{Inner Product Matrix}
As above, $B_{E_6} = A_{E_6}$.

\subsubsection*{Fundamental Weights}
We use notation similar to the notation used above, $(a_1,\ldots,a_6)\coloneqq a_1\alpha_1'+\cdots+a_6\alpha_6'$:
\begin{align*}
\lambda'_1 &= \frac{1}{3}(4,5,6,3,4,2),&\lambda'_2 &= \frac{1}{3}(5,10,12,6,8,4),&\lambda'_3 &= (2,4,6,3,4,2)\\
\lambda'_4 &= (1,2,3,2,2,1),&\lambda'_5 &= \frac{1}{3}(4,8,12,6,10,5),&\lambda'_6 &= \frac{1}{3}(2,4,6,3,5,4).
\end{align*}

\subsubsection*{Diagram Automorphism}
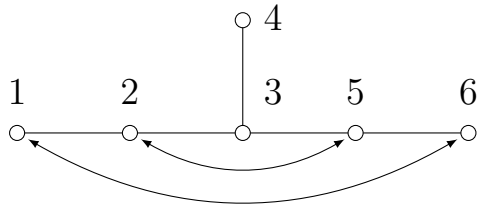
\begin{center}
    \dynkin[text style/.style={scale=1.2}, involutions={16;251}, label directions={above,above,above right,,above,above}, label, ]E6
\end{center}
Looking at the $E_6$ Dynkin diagram, there is a clear order 2 symmetry. This gives a diagram automorphism, $\tau,$ which corresponds to an outer automorphism of the Lie algebra. It is defined as follows:
\begin{align*}
\tau(\alpha'_1) &= \alpha'_6,&\tau(\alpha'_2) &= \alpha'_5,&\tau(\alpha'_3) &= \alpha'_3,\\
\tau(\alpha'_4) &= \alpha'_4,&\tau(\alpha'_5) &= \alpha'_2,&\tau(\alpha'_6) &= \alpha'_1.
\end{align*}

\subsubsection{Type $F_4$}
We denote our simple roots by $\beta_1,\ldots,\beta_4.$
\subsubsection*{Cartan Matrix}
$$A_{F_4} = 
\begin{pmatrix*}[r]
2  & -1 & 0  & 0  \\
-1 & 2  & -1 & 0  \\
0  & -2 & 2  & -1 \\
0  & 0  & -1 & 2 
\end{pmatrix*}
$$

\subsubsection*{Inner Product Matrix}
Since $|\beta_1|^2 = |\beta_2|^2 = 2$ and $|\beta_3|^2 = |\beta_4|^2 = 1$, we will multiply rows 1 and 2 of $A_{F_4}$ by 1 and rows 3 and 4 by $\dfrac{1}{2}$ to obtain our inner product matrix:
$$B_{F_4} = 
\begin{pmatrix*}[r]
2  & -1 & 0    & 0    \\
-1 & 2  & -1   & 0    \\
0  & -1 & 1    & -\frac{1}{2} \\
0  & 0  & -\frac{1}{2} & 1 
\end{pmatrix*}
$$

\subsubsection*{Fundamental Weights}
We use the notation $(b_1,\ldots,b_4)\coloneqq b_1\beta_1+\cdots+b_4\beta_4$:
\begin{align*}
\omega_1 &= (2,3,4,2),&\omega_2 &= (3,6,8,4),&\omega_3 &= (2,4,6,3),&\omega_4 = (1,2,3,2).
\end{align*}
Note that $P_{F_4} = Q_{F_4}$.

\subsubsection{Type $D_4$}
We denote our simple roots by $\gamma_1,\ldots,\gamma_4.$
\subsubsection*{Cartan Matrix}
$$A_{D_4} = 
\begin{pmatrix*}[r]
2  & -1 & 0  & 0  \\
-1 & 2  & -1 & -1 \\
0  & -1 & 2  & 0  \\
0  & -1 & 0  & 2
\end{pmatrix*}
$$

\subsubsection*{Inner Product Matrix}
As with $E_8$ and $E_6$, $B_{D_4} = A_{D_4}.$

\subsubsection*{Fundamental Weights}
We use the notation $(b_1,\ldots,b_4)\coloneqq b_1\gamma_1+\cdots+b_4\gamma_4$:
\begin{align*}
\psi_1 &= \frac{1}{2}(2,2,1,1),&\psi_2 &= (1,2,1,1),&\psi_3 &= \frac{1}{2}(1,2,2,1),&\psi_4 &= \frac{1}{2}(1,2,1,2).
\end{align*}

\subsubsection*{Diagram Automorphism}
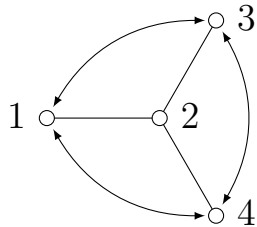
\begin{center}
    \dynkin[text style/.style={scale=1.2}, involutions={14;43;31}, label directions={left,right,,}, label, ]D4
\end{center}
Looking at the $D_4$ Dynkin diagram, there are clearly two order 3 symmetries. There are also three order 2 symmetries, each coming from swapping two of nodes 1, 3 or 4. We will only need one order 3 diagram automorphism, $\sigma$, and its corresponding outer automorphism of the Lie algebra. It is defined as follows:
\begin{align*}
\sigma(\gamma_1) &= \gamma_3,&\sigma(\gamma_2) &= \gamma_2,&\sigma(\gamma_3) &= \gamma_4,&\sigma(\gamma_4) &= \gamma_1.
\end{align*}

\subsubsection{Type $G_2$}
We denote our simple roots by $\zeta_1,\zeta_2.$
\subsubsection*{Cartan Matrix}
$$A_{G_2} = 
\begin{pmatrix*}[r]
2  & -1 \\
-3 & 2
\end{pmatrix*}
$$

\subsubsection*{Inner Product Matrix}
Since $|\zeta_1|^2  = 2$ and $|\zeta_2|^2 = 2/3$, we will multiply row 1 of $A_{G_2}$ by 1 and row 2 by $\dfrac{1}{3}$ to obtain our inner product matrix:
$$B_{G_2} = 
\begin{pmatrix*}[r]
2  & -1  \\
-1 & \frac{2}{3}
\end{pmatrix*}
$$

\subsubsection*{Fundamental Weights}
We use the notation $(c_1,c_2)\coloneqq c_1\zeta_1+c_2\zeta_2$:
\begin{align*}
\xi_1 &= (2,3),&\xi_2 &= (1,2).
\end{align*}
Note that $P_{G_2} = Q_{G_2}$.

\subsubsection{Projections of $\mathfrak{h}^*$}
Now suppose we have five Lie algebras: $\mathfrak{g}$ of type $E_8$, $\mathfrak{g'}$ of type $E_6$, $\mathfrak{k}$ of type $D_4$, $\mathfrak{a}$ of type $F_4$ and $\mathfrak{b}$ of type $G_2$. Recall that the dual space of a CSA is the complex span of the respective algebra's simple roots. As mentioned earlier, we will work with the real spans of the simple roots, but our results immediately extend to the full dual spaces via complex linear combinations. We will denote these dual spaces by $\mathfrak{h}^*$ with an appropriate subscript.

\subsubsection*{$\dcart{E_6}\subset\dcart{E_8}$}
Consider the subspace $\mathfrak{h}^*_1\coloneqq\R\{\alpha_3,\ldots,\alpha_8\}<\dcart{\frak g}$. Since the invariant bilinear form on $\mathfrak{h}^*_1$ is positive-definite, we can use Gram-Schmidt to get an orthogonal basis which may give rational linear combinations of roots. In this case, we will multiply each vector by the least common multiple of the denominators in the rational coefficients of the simple roots so that our basis remains in the root lattice. From our orthogonal basis we then get an orthogonal projection $P_1\colon \dcart{\frak g}\rightarrow\dcart{\frak g}$ given by the matrix
$$\begin{pmatrix}
 0 & 0 & 0 & 0 & 0 & 0 & 0 & 0 \\
 0 & 0 & 0 & 0 & 0 & 0 & 0 & 0 \\
 0 & -\frac{4}{3} & 1 & 0 & 0 & 0 & 0 & 0 \\
 0 & -\frac{5}{3} & 0 & 1 & 0 & 0 & 0 & 0 \\
 0 & -2 & 0 & 0 & 1 & 0 & 0 & 0 \\
 0 & -1 & 0 & 0 & 0 & 1 & 0 & 0 \\
 0 & -\frac{4}{3} & 0 & 0 & 0 & 0 & 1 & 0 \\
 0 & -\frac{2}{3} & 0 & 0 & 0 & 0 & 0 & 1 
\end{pmatrix}.$$
Under the projection $P_1$, we have: $\alpha_1\mapsto 0$, $\alpha_2\mapsto -\frac{1}{3}(4\alpha_3+5\alpha_4+6\alpha_5+3\alpha_6+4\alpha_7+2\alpha_8)$ and $\alpha_i\mapsto\alpha_i$ for $i = 3,\ldots, 8$ and it follows that $\im(P_1) = \lah^*_1$.
Notice that the map $p_1(\alpha_i) = \alpha'_{i-2}$ for $i = 3,\ldots 8$ induces a vector space isomorphism between $\mathfrak{h}^*_1$ and $\dcart{\frak g'}$. Further, the invariant bilinear form is preserved by $p_1$ and so we will identify $\mathfrak{h}^*_1$ with $\dcart{\frak g'}$.

\subsubsection*{$\dcart{F_4}\subset\dcart{E_6}$}
Recall the outer automorphism $\tau$ from above and consider the subspace $$\mathfrak{h}^*_T\coloneqq\R\left\{\alpha'_4, \alpha'_3, \frac{1}{2}(\alpha'_2+\alpha'_5), \frac{1}{2}(\alpha'_1+\alpha'_6)\right\}$$ of $\dcart{\frak g'}$. We can define an orthogonal projection $T\colon \dcart{\frak g'}\rightarrow\dcart{\frak g'}$ by $T(\alpha) = \frac{1}{2}(\tau(\alpha)+\tau^2(\alpha))$. Under $T$ we have: $\alpha'_1,\alpha'_6\mapsto\frac{1}{2}(\alpha'_1+\alpha'_6)$, $\alpha'_2,\alpha'_5\mapsto\frac{1}{2}(\alpha'_2+\alpha'_5)$, $\alpha'_3\mapsto\alpha'_3$ and $\alpha'_4\mapsto\alpha'_4$ and it follows that $\im(T) = \lah^*_T$. Notice that the map $t$ defined by $t(\frac{1}{2}(\alpha'_1+\alpha'_6)) = \beta_4$, $t(\frac{1}{2}(\alpha'_2+\alpha'_5)) = \beta_3$, $t(\alpha'_4) = \beta_2$ and $t(\alpha'_3) = \beta_1$ induces a vector space isomorphism between $\mathfrak{h}^*_T$ and $\dcart{\frak a}$. Further, the invariant bilinear form is preserved by $t$ and so we will identify $\mathfrak{h}^*_T$ with $\dcart{\frak a}$.

\subsubsection*{$\dcart{D_4}\subset\dcart{E_8}$}
Consider the subspace $\mathfrak{h}^*_2\coloneqq\R\{\alpha_2+2\alpha_3+3\alpha_4+4\alpha_5+2\alpha_6+3\alpha_7+2\alpha_8,\alpha_1,\alpha_2,\alpha_2+2\alpha_3+2\alpha_4+2\alpha_5+\alpha_6+\alpha_7\}<\dcart{\frak g}$. Since the invariant bilinear form on $\mathfrak{h}^*_2$ is positive-definite, we can use Gram-Schmidt as we did above to get an orthogonal basis and projection $P_2\colon \dcart{\frak g}\rightarrow\dcart{\frak g}$ given by the matrix
$$\begin{pmatrix}
 1 & 0 & 0 & 0 & 0 & 0 & 0 & 0 \\
 0 & 1 & 0 & 0 & 0 & 0 & 0 & 0 \\
 0 & 0 & 1 & 0 & 0 & 0 & 0 & 0 \\
 0 & 0 & 1 & 0 & 0 & 0 & 0 & 0.5 \\
 0 & 0 & 1 & 0 & 0 & 0 & 0 & 1 \\
 0 & 0 & 0.5 & 0 & 0 & 0 & 0 & 0.5 \\
 0 & 0 & 0.5 & 0 & 0 & 0 & 0 & 1 \\
 0 & 0 & 0 & 0 & 0 & 0 & 0 & 1 
\end{pmatrix}$$
Under the projection $P_2$, we have: $\alpha_1\mapsto \alpha_1$, $\alpha_2\mapsto\alpha_2$, $\alpha_3\mapsto\frac{1}{2}(2\alpha_3+2\alpha_4+2\alpha_5+\alpha_6+\alpha_7)$, $\alpha_8\mapsto\frac{1}{2}(\alpha_4+2\alpha_5+\alpha_6+2\alpha_7+2\alpha_8)$ and $\alpha_i\mapsto 0$ for $i = 4,5,6,7$. It will follow that $\im(P_2) = \lah^*_2.$
Notice that the map $p_2$ given by $p_2(\alpha_2+2\alpha_3+3\alpha_4+4\alpha_5+2\alpha_6+3\alpha_7+2\alpha_8) = \gamma_1$, $p_2(\alpha_1) = \gamma_2$, $p_2(\alpha_2) = \gamma_3$ and $p_2(\alpha_2+2\alpha_3+2\alpha_4+2\alpha_5+\alpha_6+\alpha_7) = \gamma_4$ induces a vector space isomorphism between $\mathfrak{h}^*_2$ and $\dcart{\frak k}$. Further, the invariant bilinear form is preserved by $p_2$ and so we will identify $\mathfrak{h}^*_2$ with $\dcart{\frak k}$.

\subsubsection*{$\dcart{G_2}\subset\dcart{D_4}$}

Recall the outer automorphism $\sigma$ from above and consider the subspace $$\mathfrak{h}^*_S\coloneqq\R\left\{\gamma_2, \frac{1}{3}(\gamma_1+\gamma_3+\gamma_4)\right\}$$ of $\dcart{\frak k}$. We can define an orthogonal projection $S\colon \dcart{\frak k}\rightarrow\dcart{\frak k}$ by $S(\gamma) = \frac{1}{3}(\sigma(\gamma)+\sigma^2(\gamma)+\sigma^3(\gamma))$. Under $S$ we have: $\gamma_1,\gamma_3,\gamma_4\mapsto\frac{1}{3}(\gamma_1+\gamma_3+\gamma_4)$ and $\gamma_2\mapsto\gamma_2$. It follows that $\im(S) = \lah^*_S$. Notice that the map $s$ defined by $s(\gamma_2) = \zeta_1$ and $s(\frac{1}{3}(\gamma_1+\gamma_3+\gamma_4)) = \zeta_2$ induces a vector space isomorphism between $\mathfrak{h}^*_S$ and $\dcart{\frak b}$. Further, the invariant bilinear form is preserved by $s$ and so we will identify $\mathfrak{h}^*_S$ with $\dcart{\frak b}$.

\subsubsection*{$\dcart{F_4}\oplus\dcart{G_2}\subset\dcart{E_8}$}
We make a final note that although $\lah^*_S,\lah^*_T\subset\lah^*$ (under the above identifications), we must make sure that $\lah^*_S\cap\lah^*_T = \{0\}$ in order to get a direct sum. Under the identification of our various isomorphisms, we identify $\zeta_1$ with $\alpha_1$ and $\zeta_2$ with $\lambda_1-2\alpha_1$. Immediately, $P_1(\alpha_1) = 0$ and $P_1(\lambda_1-2\alpha_1) = 0$, which gives us what we wanted.

\subsection{Affine Kac-Moody Lie Algebras}
Our results can be extended to all of the corresponding affine algebras. The dual Cartan of each will now have a dimension 2 greater than that of the corresponding simple Lie algebra. This is due to the addition of the basic imaginary root, $\delta,$ and the basic  fundamental weight, $\Lambda_0$. We will always represent the basic fundamental weight with a subscript 0, but we will change the Greek letter depending on the algebra with which we are working (i.e. $\Lambda_0$ for $E_8^{(1)}$, $\Lambda_0'$ for $E_6^{(1)}$, $\Psi_0$ for $D_4^{(1)}$, $\Omega_0$ for $F_4^{(1)}$ and $\Xi_0$ for $G_2^{(1)}$). We will also use a subscript on $\delta$ to denote the algebra with which we are working. Finally, we will denote affine algebras by adding a hat to their corresponding finite dimensional simple algebras.

\subsubsection{$\widehat{P_1}$ and $\widehat{P_2}$}
Let $\widehat{\dcart{\frak g}}\coloneqq\mathbb{C}\{\Lambda_0,\delta,\alpha_1,\ldots,\alpha_8\}$ and $\widehat{\mathfrak{h}^*_1}\coloneqq\mathbb{C}\{\Lambda_0,\delta,\alpha_3,\ldots,\alpha_8\}$. Then $P_1$ can be extended to a projection $\widehat{P_1}\colon\widehat{\dcart{\frak g}}\rightarrow\widehat{\dcart{1}}$ in the following way: $\widehat{P_1}\restrict{\dcart{\frak g}} = P_1$, $\widehat{P_1}(\Lambda_0) = \Lambda_0$ and $\widehat{P_1}(\delta) = \delta$. Similarly, $\widehat{p_1}$ defined such that $\widehat{p_1}\restrict{\dcart{1}} = p_1$, $\widehat{p_1}(\Lambda_0) = \Lambda'_0$ and $\widehat{p_1}(\delta) = \delta_{\frak g'}$ gives an isomorphism between $\widehat{\dcart{1}}$ and $\widehat{\dcart{\frak g'}}$ where the latter is defined as $\C\{\Lambda'_0,\delta_{\frak g'},\alpha'_1,\ldots,\alpha'_6\}$.

Similarly, we will extend $P_2$ to the projection $\widehat{P_2}\colon\widehat{\dcart{g}}\rightarrow\widehat{\dcart{2}}$ by defining it to fix $\Lambda_0$ and $\delta$ and restricting to $P_2$ on $\dcart{\frak g}$. We will also extend $p_2$ to an isomorphism $\widehat{p_2}\colon\widehat{\dcart{2}}\rightarrow\widehat{\dcart{\frak k}}$ by restricting to $p_2$ on $\dcart{2}$ and demanding that $\widehat{p_2}(\delta) = \delta_{\frak k}$ and $\widehat{p_2}(\Lambda_0) = \Psi_0$.

\subsubsection{$\widehat{T}$ and $\widehat{S}$}
Notice that we can extend $\tau$ and $\sigma$ to diagram automorphisms of the $E_6^{(1)}$ and $D_4^{(1)}$ Dynkin diagrams by respectively fixing $\alpha_0'$ and $\gamma_0$. Define: $$\widehat{\dcart{T}}\coloneqq\C\{\Lambda'_0, \delta_{\frak g'}, \alpha'_4, \alpha'_3, \frac{1}{2}(\alpha'_2+\alpha'_5), \frac{1}{2}(\alpha'_1+\alpha'_6)\}$$ and $$\widehat{\dcart{S}}\coloneqq\{\Psi_0,\delta_{\frak k},\gamma_2, \frac{1}{3}(\gamma_1+\gamma_3+\gamma_4)\}.$$ We can then extend our projections to get:
$\widehat{T}\colon\widehat{\dcart{\frak g'}}\rightarrow\widehat{\dcart{T}}$ with $\widehat{T}\restrict{\dcart{\frak g'}} = T$, $\widehat{T}(\Lambda'_0) = \Lambda'_0$ and $\widehat{T}(\delta_{\frak g'}) = \delta_{\frak g'}$ and $\widehat{S}\colon\widehat{\dcart{\frak k}}\rightarrow\widehat{\dcart{S}}$ with $\widehat{S}\restrict{\dcart{\frak k}} = S$, $\widehat{S}(\Psi_0) = \Psi_0$ and $\widehat{S}(\delta_{\frak k}) = \delta_{\frak k}$.
We can further extend the isomorphisms $t$ and $s$ to $\widehat{t}\colon\widehat{\dcart{T}}\rightarrow\widehat{\dcart{\frak a}}$ and $\widehat{s}\colon\widehat{\dcart{S}}\rightarrow\widehat{\dcart{\frak b}}$ by restricting to $t$ and $s$ on $\dcart{T}$ and $\dcart{S}$ respectively and by sending basic fundamental weight to basic fundamental weight ($\Lambda'_0\mapsto\Omega_0$ and $\Psi_0\mapsto\Xi_0$ respectively) and basic imaginary root to basic imaginary root ($\delta_{\frak g'}\mapsto\delta_{\frak a}$ and $\delta_{\frak k}\mapsto\delta_{\frak b}$ respectively).

\subsubsection{The Basic Fundamental Weight and Imaginary Root}
We must note that hidden in our projections of $\Lambda_0$, $\Lambda'_0$ and $\Psi_0$ is that although each appears to be projected to itself as a vector, we must restrict its domain as a linear functional to the corresponding Cartan subalgebra. For example, $\widehat{P_1}(\Lambda_0) = \Lambda_0\restrict{\widehat{\mathfrak{h}}_{\frak g'}}$ and $\widehat{P_2}(\Lambda_0) = \Lambda_0\restrict{\widehat{\mathfrak{h}}_{\frak k}}$. This also holds analogously for $\delta$.

\chapter{Character Theory}
In this chapter we dive more deeply into the techniques of character calculations for affine Kac-Moody Lie algebras. We will also use some numerical data to begin to get a handle on our decomposition.

\section{Specializations}
In this dissertation, we use specialization as a tool for making character calculations simpler. A $(s_0,s_1,\ldots,s_\ell)$ \textbf{-specialization} of $\ch(V^\Lambda)$ where $V^\Lambda$ is a highest weight module of some affine algebra with simple roots $\alpha_i$, for $i = 0,\ldots,\ell$, is given by the substitution $e^{-\alpha_i} = q^{s_i}$. One important specialization is the \textbf{principal specialization}, that is the $(1,1,\ldots,1)$-specialization. This specialization is often used because a theorem of Lepowsky \cite{Lepowsky} allows it to be calculated as a quotient of two infinite products.

In Section 1.3 we identified a direct sum of isomorphic copies of dual CSAs of type $F_4$ and $G_2$ inside of the $E_8$ dual CSA. This allowed us to identify the roots and weights of $F_4$ and $G_2$ with elements of the $E_8$ dual CSA. When applying a specialization in a branching rule, we need to be mindful of these identifications and their implications on how choice of specialization for our $E_8^{(1)}$ highest weight module will affect the specialization of any $F_4^{(1)}$ and $G_2^{(1)}$ submodules.

For example, we will notice that under our isomorphisms, we identify the short simple root $\zeta_2$ of $G_2$ with $\frac{1}{3}(\lambda_1-2\alpha_1)$. Under the principal specialization, $e^{-\lambda_1} = q^{29}$ and $e^{-\alpha_1} = q$ and so $e^{-\zeta_2} = e^{-\frac{1}{3}(\lambda_1-2\alpha_1)} = q^{9}$. So we would be forced to use a different specialization on our $G_2^{(1)}$ modules. In fact, applying the principal specialization in our case would give $e^{-\alpha_i} = q$ for $i = 0,\ldots,8$, $e^{-\beta_i} = q$ for $i = 0,1,2,3,4$ and $e^{-\zeta_i} = q$ for $i = 0,1$. And so we only have issues with this one simple root.

Rather than trying to mix specializations, we choose instead to use the \textbf{horizontal specialization}, which is the $(1,0,\ldots,0)$ specialization, and which we will denote by the subscript \textbf{hs}. Under this specialization, $e^{-\alpha} = 1$ for any $\alpha$ from the root lattice of the underlying finite dimensional, simple algebra. Further, since $\delta = \alpha_0+\theta$ and $\theta$ is in the root lattice of the underlying simple algebra, $e^{-\delta} = q$ under this specialization. The graded dimension of many affine highest weight modules have been calculated up to a finite number of terms in \cite{KMPS} using this specialization. We will use these approximations to get an idea of how our level 1 $F_4^{(1)}$ and $G_2^{(1)}$ modules sit inside of the basic module $V^{\Lambda_0}$. And although we cannot make use of Lepowsky's beautiful theorem, we will use an equally beautiful formula of Kac and Peterson \cite{KacPeterson} which will make use of theta functions.

\section{The Kac-Peterson Theta Function Formula}
Throughout this section, $\Lambda\in\widehat{P}^+$ will be of level $k$. We will consider the character formula $$\ch(V^\Lambda) = \sum_{\mu\in\Pi^\Lambda}\dim(V_\mu^\Lambda)e^\mu.$$ Define the set of maximal weights of $V^\Lambda$ by $\max(\Lambda)\coloneqq\{\lambda\in\Pi^\Lambda\mid\lambda+\delta\not\in\Pi^\Lambda\}$. We now state some properties of weights of affine modules.

\begin{proposition}
For $\Lambda\in\widehat{P}^+$, we have:
\begin{enumerate}[(1)]
    \item $\widehat{\cal W}(\max(\Lambda)) = \max(\Lambda)$,
    \item for $\lambda\in\Pi^\Lambda$, we have $\widehat{\cal W}_\lambda\cap t(M^*) = 1$ where $\widehat{\cal W}_\lambda$ is the stabilizer of $\lambda$ in $\widehat{\cal W}$,
    \item for each $\mu\in\Pi^\Lambda$, there exists a unique $\lambda\in\max(\Lambda)$ and $0\leq n\in\Z$ such that $\mu = \lambda-n\delta$,
    \item for $\lambda\in\Pi^\Lambda$, the set of all $k\in\Z$ such that $\lambda-k\delta\in\Pi^\Lambda$, is of the form $\{k\in\Z\mid k\geq -p\}$ with $p\geq 0$, and the map $k\mapsto\dim(V^\Lambda_{\lambda-k\delta})$ is nondecreasing on this set.
\end{enumerate}
For proofs and more information on these, see Chapter 12 of \cite{kac}. One immediate consequence is that $\Pi^\Lambda = \displaystyle\bigsqcup_{\lambda\in\max(\Lambda)}\{\lambda-n\delta\mid 0\leq n\in\Z\}$.  Define the generating series $\displaystyle a_\lambda^\Lambda = \sum_{n = 0}^\infty\dim(V^\Lambda_{\lambda-n\delta})e^{-n\delta}$. Then by the properties we just stated, we get: $$\ch(V^\Lambda) = \sum_{\lambda\in\max(\Lambda)}e^\lambda a^\Lambda_\lambda = \sum_{\substack{\lambda\in\max(\Lambda) \\ \lambda\mod t(M^*)}}\left(\sum_{t_\alpha\in t(M^*)} e^{t_\alpha(\lambda)}\right)a_\lambda^\Lambda.$$
For $\lambda\in\max(\Lambda)$ of level $k$, set $\displaystyle\Theta_\lambda\coloneqq e^{-\frac{|\lambda|^2}{2k}\delta}\sum_{t_\alpha\in t(M^*)}e^{t_\alpha(\lambda)}.$ Then for any $t_\alpha\in t(M^*)$, we have $$t_\alpha(\lambda) = t_\alpha(k\Lambda_0+\overline{\lambda}) = k\Lambda_0+\overline{\lambda}+k\alpha-\left((\overline{\lambda},\alpha)+\frac{1}{2}|\alpha|^2k\right)\delta$$ where $\overline{\lambda} = \lambda-k\Lambda_0$ is an element of the finite dimensional weight lattice, $P$. Let $\gamma = \alpha+\frac{1}{k}\overline{\lambda}$. Notice that $$-\frac{1}{2}k|\gamma|^2\delta = -\frac{1}{2}k\left(|\alpha|^2+\frac{2}{k}(\overline{\lambda},\alpha)+\frac{1}{k^2}|\overline{\lambda}|^2\right)\delta = -\left(\frac{1}{2}k|\alpha|^2+(\overline{\lambda},\alpha)+\frac{1}{2k}|\overline{\lambda}|^2\right)\delta.$$ So $$t_\alpha(\lambda) = k\Lambda_0+k\gamma-\frac{1}{2}k|\gamma|^2\delta+\frac{|\overline{\lambda}|^2}{2k}\delta.$$ This gives us that $$\Theta_{\lambda} = e^{k\Lambda_0}\sum_{\gamma\in M^*+k^{-1}\overline{\lambda}}e^{-\frac{1}{2}k|\gamma|^2\delta+k\gamma}.$$ This shows that $\Theta_\lambda$ only depends on $\lambda\mod (kM^*+\C\delta).$ For $\Lambda,\lambda\in H^*$ we define two rational quantities: $s_\Lambda\coloneqq\dfrac{|\Lambda+\widehat{\rho}|^2}{2(k+h^\vee)}$ (where $h^\vee = \widehat{\rho}(c)$ is the \textbf{dual Coxeter number}) and $s_\Lambda(\lambda)\coloneqq s_\Lambda-\dfrac{|\lambda|^2}{2k}$. Now for any $\lambda\in\widehat{P}$, we define the series $$\displaystyle c_\lambda^\Lambda\coloneqq e^{-s_\Lambda(\lambda)\delta}\sum_{n\in\Z}\dim(V^\Lambda_{\lambda-n\delta}) e^{-n\delta}.$$ Notice that if $\lambda\in\max(\Lambda)$ then $c^\Lambda_\lambda = e^{-s_\Lambda(\lambda)\delta}a^\Lambda_\lambda$. Kac and Peterson showed that if $\mu\in\widehat{P}$ and $\lambda\in\max(\Lambda)$ such that $\mu = \lambda-n\delta$ for some $n\in\Z$ then $c_\mu^\Lambda = c_\lambda^\Lambda$. They also showed that $c^\Lambda_{w(\lambda)} = c^\Lambda_\lambda$ for all $w\in\widehat{\cal W}$. Because $\widehat{\cal W} = \mathcal{W}\ltimes t(M^*)$, this means that $c^\Lambda_{w'(\lambda)+k\gamma+n\delta} = c^\Lambda_\lambda$ for all $w'\in\mathcal{W}$, $\gamma\in M^*$ and $n\in\C$ (see page 169 of \cite{KacPeterson}). We can now rewrite our character formula as follows:
$$\ch(V^\Lambda) = q^{-s_\Lambda}\sum_{\substack{\lambda\in\widehat{P}\mod (kM^*+\C\delta) \\ \lambda(c) = k}}c^\Lambda_\lambda\Theta_\lambda$$ where $\widehat{P}\mod (kM^*+\C\delta)$ is the set of coset representatives of $\widehat{P}/(kM^*+\C\delta)$. This formula will allow us to perform our character calculation using the theta functions which will be introduced in the next chapter.
\end{proposition}

\section{Numerical Data}
There are two level one modules for the $F_4^{(1)}$ affine algebra and two for the $G_2^{(1)}$ affine algebra. In order to get a handle on how the $E_8^{(1)}$ basic module decomposes with respect to these, it is helpful to look at the first few ``slices," that is the first few graded pieces with respect to our chosen specialization, of each module. As mentioned earlier, we will use the horizontal specialization. In \cite{KMPS}, we are able to get graded dimensions up to slice 7. These dimensions are listed in table \ref{slices} below. Also listed are the pages of \cite{KMPS} on which each module's graded dimensions are given.

\begin{table}[h]
\centering
\begin{tabular}{|c|ccccc|}
\hline
      & \multicolumn{5}{c|}{\textbf{Dimension}}                                                                                               \\ \hline
\textbf{Slice} & \multicolumn{1}{c|}{$V^{\Lambda_0}$}        & \multicolumn{1}{c|}{$V^{\Omega_0}$}      & \multicolumn{1}{c|}{$V^{\Omega_4}$}      & \multicolumn{1}{c|}{$V^{\Xi_0}$}    & $V^{\Xi_2}$    \\ \hline
0     & \multicolumn{1}{c|}{1}        & \multicolumn{1}{c|}{1}      & \multicolumn{1}{c|}{26}     & \multicolumn{1}{c|}{1}    & 7    \\ \hline
1     & \multicolumn{1}{c|}{248}      & \multicolumn{1}{c|}{52}     & \multicolumn{1}{c|}{299}    & \multicolumn{1}{c|}{14}   & 34   \\ \hline
2     & \multicolumn{1}{c|}{4124}     & \multicolumn{1}{c|}{377}    & \multicolumn{1}{c|}{1702}   & \multicolumn{1}{c|}{42}   & 119  \\ \hline
3     & \multicolumn{1}{c|}{34752}    & \multicolumn{1}{c|}{1976}   & \multicolumn{1}{c|}{7475}   & \multicolumn{1}{c|}{140}  & 322  \\ \hline
4     & \multicolumn{1}{c|}{213126}   & \multicolumn{1}{c|}{7852}   & \multicolumn{1}{c|}{27300}  & \multicolumn{1}{c|}{350}  & 819  \\ \hline
5     & \multicolumn{1}{c|}{1057504}  & \multicolumn{1}{c|}{27404}  & \multicolumn{1}{c|}{88452}  & \multicolumn{1}{c|}{840}  & 1862 \\ \hline
6     & \multicolumn{1}{c|}{4530744}  & \multicolumn{1}{c|}{84981}  & \multicolumn{1}{c|}{260650} & \multicolumn{1}{c|}{1827} & 4025 \\ \hline
7     & \multicolumn{1}{c|}{17333248} & \multicolumn{1}{c|}{243230} & \multicolumn{1}{c|}{714727} & \multicolumn{1}{c|}{3858} & 8218 \\ \hline
\textbf{Page}  & \multicolumn{1}{c|}{856}      & \multicolumn{2}{c|}{868}                                  & \multicolumn{2}{c|}{879}         \\ \hline
\end{tabular}   
\caption{Graded dimensions for level-one $F_4^{(1)}$ and $G_2^{(1)}$ modules up to slice 7}
\label{slices}
\end{table}

So we have the following truncated graded dimensions:
\begin{align}
\notag
\gr(V^{\Lambda_0})_{\hs} = 1 &+248q+4124q^2+34752q^3+213126q^4+1057504q^5\\
&+4530744q^6+17333248q^7+\cdots\\
\notag
\gr(V^{\Omega_0})_{\hs} = 1 &+52q+377q^2+1976q^3+7852q^4+27404q^5\\
&+84981q^6+243230q^7+\cdots\\
\notag
\gr(V^{\Omega_4})_{\hs} = 26 &+299q+1702q^2+7475q^3+27300q^4+88452q^5\\
&+260650q^6+714727q^7+\cdots\\
\notag
\gr(V^{\Xi_0})_{\hs} = 1 &+14q+42q^2+140q^3+350q^4+840q^5\\
&+1827q^6+3858q^7+\cdots\\
\notag
\gr(V^{\Xi_2})_{\hs} = 7 &+34q+119q^2+322q^3+819q^4+1862q^5\\
&+4025q^6+8218q^7+\cdots
\end{align}

Note that the only way to get a dimension of 1 on slice 0 is to tensor $V^{\Omega_0}\otimes V^{\Xi_0}$. If we tensor $V^{\Omega_4}\otimes V^{\Xi_2}$ we will get a dimension of $182$ on the initial slice, and so this module must live deeper down. We will check the graded dimensions to see how all of this fits together. Note that when we begin multiplying, adding or shifting these truncated graded dimensions, we always have to truncate our answer at $q^7$ as any additional terms will be inaccurate.

\begin{align}
\notag
    \gr(V^{\Omega_0})_{\hs}\gr(V^{\Xi_0})_{\hs} = 1 &+ 66 q + 1147 q^2 + 9578 q^3 + 58980 q^4 + 292144 q^5\\
    &+ 1252518 q^6 + 4790354 q^7+\cdots
\end{align}

\begin{align}
\notag
    \gr(V^{\Omega_4})_{\hs}\gr(V^{\Xi_2})_{\hs} = 182 &+ 2977 q + 25174 q^2 + 154146 q^3 + 765360 q^4\\
    &+ 3278226 q^5+  12542894 q^6 + 43889869 q^7+\cdots
\end{align} 

Shifting the second tensor product down one slice (i.e. multiplying equation 7 by $q$) will give us 

\begin{align}
\notag
    q\gr(V^{\Omega_4})_{\hs}\gr(V^{\Xi_2})_{\hs} = 182q &+ 2977 q^2 + 25174 q^3 + 154146 q^4 + 765360 q^5\\
    &+ 3278226 q^6+  12542894 q^7+\cdots
\end{align} 

and so adding equations 6 and 8 will give:

\begin{align}
    \notag
    \gr(V^{\Omega_4})_{\hs}\gr(V^{\Xi_2})_{\hs}&+q\gr(V^{\Omega_4})_{\hs}\gr(V^{\Xi_2})_{\hs}\\ 
    \notag
    &= 1 +248q+4124q^2+34752q^3+213126q^4+1057504q^5\\
    &+4530744q^6+17333248q^7+\cdots
\end{align}
but this is precisely $\gr(V^{\Lambda_0})_{\hs}$ up to slice 7. This is our first indication that $V^{\Lambda_0} = (V^{\Omega_0}\otimes V^{\Xi_0})\oplus (V^{\Omega_4}\otimes V^{\Xi_2})$ however we will need to prove this result for the entire graded dimensions of the modules in order to get our result.

\chapter{Theta Functions and Theta Function Identities}
In this chapter, we define the various theta functions that appear in our character calculation. We will also give some useful identities relating these theta functions. We will conclude with proofs of some nice identities needed to complete our character calculation. A good source of information about these functions, their history and many identities can be found in \cite{cooper}. In what follows, $\tau\in\C$ with $\im(\tau)>0$ and $q\coloneqq e^{2\pi i\tau}$. Also, let $\nu_k\coloneqq e^{2\pi i/k}$. The two most important functions we will consider are the \textbf{Euler function} $\varphi(q) = \prod_{n = 1}^\infty(1-q^n)$ and the \textbf{Dedekind $\eta$-function} $\eta(\tau) = q^{1/24}\varphi(q)$. We will often write $\varphi_n = \varphi(q^n)$ and $\eta_n = \eta(n\tau)$.

\section{Theta Functions and Classical Identities}
In this section we present the definitions of the Jacobi, Ramanujan and Borweins' cubic theta functions and give some classical identities from the literature. We begin with an important result of Jacobi, given in \cite{jacobi}. For a more modern proof, see Theorem 0.1 in \cite{cooper}.
\begin{proposition}[Jacobi Triple Product Identity]
$$\sum_{n = -\infty}^\infty q^{n^2}x^n = \prod_{n = 1}^\infty(1+xq^{2n-1})(1+x^{-1}q^{2n-1})(1-q^{2n}) = (-xq;q^2)_\infty(-q/x;q^2)_\infty(q^2;q^2)_\infty$$ where $(a;q)_\infty = \prod_{n = 0}^\infty(1-aq^n)$ is the $q$-Pochammer symbol.
\end{proposition} In terms of the $q$-Pochammer symbol, $\varphi_n = (q^n;q^n)_\infty$ and $\eta_n = q^{n/24}(q^n;q^n)_\infty$.

\subsection{Jacobi Theta Functions}
The four functions we will define are specializations of the function $$\vartheta(z;t) = \sum_{n = -\infty}^\infty\exp(\pi i n^2 t+2\pi i n z)$$ where $z,t\in\C$ with $\im(t)>0$.

\begin{definition}[Jacobi Theta Functions]
\leavevmode
\begin{enumerate}[(1)]
    \item $\displaystyle\theta_2(q)\coloneqq e^{\frac{\pi i\tau}{2}}\vartheta(\tau;2\tau) = \sum_{n = -\infty}^\infty q^{(n+1/2)^2}$
    \item $\displaystyle\theta_3(q)\coloneqq \vartheta(0;2\tau) = \sum_{n = -\infty}^\infty q^{n^2}$
    \item $\displaystyle\theta_4(q)\coloneqq \vartheta\left(\frac{1}{2};2\tau\right) = \sum_{n = -\infty}^\infty(-1)^nq^{n^2}$
    \item $\displaystyle\psi_k(q)\coloneqq e^{\frac{2\pi i\tau}{k^2}}\vartheta(2\tau/k;2\tau) = \sum_{n = -\infty}^\infty q^{(n+1/k)^2}$.
\end{enumerate}
\end{definition}

Note that $\psi_2(q) = \theta_2(q)$. The following identities are specializations of transformation properties of $\vartheta$ proven by Jacobi in \cite{jacobi}. A good resource for information on these properties and their proofs can found in pages 462-490 of \cite{whittaker}. Another good source of proofs for these identities is \cite{cooper}. It should be noted, however, that Cooper presents these identities in the equivalent language of the Ramanujan theta functions $\phi$ and $\psi$. See the remark after Definition \ref{RamThetaDef} below.

\begin{proposition}
\label{jacobiprops}
\leavevmode
\begin{enumerate}[(1)]
    \item $\theta_3(q)^4 = \theta_2(q)^4+\theta_4(q)^4$
    \item $\theta_3(q)^2 = \theta_3(q^2)^2+\theta_2(q^2)^2$
    \item $\theta_4(q)^2 = \theta_3(q^2)^2-\theta_2(q^2)^2$
    \item $\psi_3(q) = \dfrac{1}{2}(\theta_3(q^{1/9})-\theta_3(q))$
    \item $\psi_6(q) = \dfrac{1}{2}(\theta_2(q^{1/9})-\theta_2(q))$
\end{enumerate}
\end{proposition}
\begin{proof}
Identity (1) is Theorem 3.7 in \cite{cooper}, (2) and (3) are the second identity in Corollary 3.4 in \cite{cooper} where (3) comes from replacing $q$ with $-q$. For (4), notice that:
\begin{align}
    \theta_3(q^{1/9}) &= \sum_{n = -\infty}^\infty q^{\frac{1}{9}n^2}\\
    &= \sum_{n\equiv 0\mod 3}q^{\frac{1}{9}n^2}+\sum_{n\equiv 1\mod 3}q^{\frac{1}{9}n^2}+\sum_{n\equiv -1\mod 3}q^{\frac{1}{9}n^2}\\
    &= \sum_{k = -\infty}^\infty q^{\frac{1}{9}(3k)^2}+\sum_{k = -\infty}^\infty q^{\frac{1}{9}(3k+1)^2}+\sum_{k = -\infty}^\infty q^{\frac{1}{9}(-3k-1)^2}\\
    &= \sum_{k = -\infty}^\infty q^{k^2}+\sum_{k = -\infty}^\infty q^{k^2+\frac{2}{3}k+\frac{1}{9}}+\sum_{k = -\infty}^\infty q^{k^2+\frac{2}{3}k+\frac{1}{9}}\\
    &= \sum_{k = -\infty}^\infty q^{k^2}+2\sum_{k = -\infty}^\infty q^{(k+1/3)^2}\\
    &= \theta_3(q)+2\psi_3(q)
\end{align}
which will immediately give the desired result. Similarly,
\begin{align}
    \theta_2(q^{1/9}) &= \sum_{n = -\infty}^\infty q^{\frac{1}{9}(n+1/2)^2} = \sum_{n = -\infty}^\infty q^{\frac{n^2+n+1/4}{9}}\\
    &= \sum_{n\equiv 1\mod 3}q^{\frac{n^2+n+1/4}{9}}+\sum_{n\equiv 0\mod 3}q^{\frac{n^2+n+1/4}{9}}+\sum_{n\equiv -1\mod 3}q^{\frac{n^2+n+1/4}{9}}\\
    &= \sum_{k = -\infty}^\infty q^{\frac{9k^2+9k+9/4}{9}}+\sum_{k = -\infty}^\infty q^{\frac{9k^2+3k+1/4}{9}}+\sum_{k = -\infty}^\infty q^{\frac{9k^2-3k+1/4}{9}}\\
    &= \sum_{k = -\infty}^\infty q^{k^2+k+1/4}+2\sum_{k = -\infty}^\infty q^{k^2+\frac{1}{3}k+\frac{1}{36}}\\
    &= \sum_{k = -\infty}^\infty q^{(k+1/2)^2}+2\sum_{k = -\infty}^\infty q^{(k+1/6)^2}\\
    &= \theta_2(q)+2\psi_6(q)
\end{align}
where rightmost two summands in line (9) combine to give line (10) by changing the index of summation in the third summand to $-k$. This immediately gives (5).
\end{proof}

\subsection{Ramanujan's Theta Functions}
Ramanujan defined a generalization of the theta functions above. Good information on these can be found in Cooper \cite{cooper} as well as in Berndt et al. \cite{berndt}.

\begin{definition}
A \textbf{general theta function}, $f(a,b)$ is defined by $$f(a,b)\coloneqq\sum_{n = -\infty}^\infty a^{n(n+1)/2}b^{n(n-1)/2}$$ where $|ab|<1$.
\end{definition}
In his notebook (see \cite{ramanujan}, page 34, entry 18), Ramanujan showed the following properties of $f(a,b)$:

\begin{proposition}
\leavevmode
\begin{enumerate}[(1)]
    \item $f(a,b) = f(b,a)$
    \item $f(1,a) = 2f(a,a^3)$
    \item $f(-1,a) = 0$
    \item $f(a,b) = (-a;ab)_\infty(-b;ab)_\infty(ab;ab)_\infty$.
\end{enumerate}
\end{proposition}
Note that 4 follows by applying the Jacobi Triple Product Identity (JTP) with $x = \sqrt{a/b}$ and $q = \sqrt{ab}$. We now give three important theta functions of Ramanujan:

\begin{definition}[Ramanujan Theta Functions]
\label{RamThetaDef}
\leavevmode
\begin{enumerate}[(1)]
    \item $\displaystyle\phi(q)\coloneqq f(q,q) = \sum_{n = -\infty}^\infty q^{n^2} = (-q;q^2)^2_\infty(q^2;q^2)_\infty$
    \item $\displaystyle\psi(q)\coloneqq f(q,q^3) = \sum_{n = 0}^\infty q^{n(n+1)/2} = \dfrac{(q^2;q^2)_\infty}{(q;q^2)_\infty}$
    \item $\displaystyle f(-q)\coloneqq f(-q,-q^2) = \sum_{n = -\infty}^\infty(-1)^nq^{n(3n-1)/2} = (q;q)_\infty$
\end{enumerate}
\end{definition}
Notice that $\phi(q) = \theta_3(q)$, $\phi(-q) = \theta_4(q)$ and $\psi(q) = \dfrac{1}{2}q^{-1/8}\theta_2(q^{1/2})$. We also notice that $f(-q) = \varphi(q)$ and, in doing so, recover Euler's Pentagonal Number Theorem. We now define one of Ramanujan's Eisenstein Series that will be useful to us later on:

\begin{definition}
$$P(q)\coloneqq q\dfrac{d}{dq}\log(\eta_1^{24}) = 1-24\sum_{j = 1}^\infty\frac{jq^j}{1-q^j}.$$
\end{definition} The following was originally proven by Jacobi in \cite{jacobi}. For a more modern treatment, see Theorem 3.26 in \cite{cooper}.

\begin{proposition}
\label{PhiEis}
\leavevmode
\begin{enumerate}[(1)]
    \item $\phi(q)^4 = \dfrac{1}{3}(4P(q^4)-P(q))$
    \item $\phi(-q)^4 = \dfrac{1}{3}(P(q)-6P(q^2)+8P(q^4))$
\end{enumerate}
\end{proposition}
Finally, we define the Rogers-Ramanujan series $G(q)$ and $H(q)$.

\begin{definition}[Rogers-Ramanujan Series]
\leavevmode
\begin{enumerate}[(1)]
    \item $\displaystyle G(q)\coloneqq\sum_{n = 0}^\infty\frac{q^{n^2}}{(q;q)_n}$
    \item $\displaystyle H(q)\coloneqq\sum_{n = 0}^\infty\frac{q^{n(n+1)}}{(q;q)_n}$
\end{enumerate}
where $(a;q)_n\coloneqq\prod_{k = 0}^{n-1}(1-aq^k)$.
\end{definition}
These functions satisfy the following identities (see \cite{bern22}, \cite{bern17} and pages 214-215 of \cite{bern19}):

\begin{proposition}
\label{GH}
\leavevmode
\begin{enumerate}[(1)]
    \item $G(q) = \dfrac{1}{(q;q^5)_\infty(q^4;q^5)_\infty}$
    \item $H(q) = \dfrac{1}{(q^2;q^5)_\infty(q^3;q^5)_\infty}$.
\end{enumerate}
\end{proposition}
Two other identities involving $G$ and $H$ are:
\begin{proposition}
\label{RRR}
\leavevmode
\begin{enumerate}[(1)]
    \item $G(q^2)G(q^3)+qH(q^2)H(q^3) = \dfrac{\chi(-q^3)}{\chi(-q)}$
    \item $G(q^6)H(q)-qG(q)H(q^6) = \dfrac{\chi(-q)}{\chi(-q^3)}$
\end{enumerate}
where $\chi(q)\coloneqq (-q;q^2)_\infty$.
\end{proposition}
These identities were proven by Rogers in \cite{Rogers} (see equations 10.1 and 10.2; see also entries 3.7 and 3.8 in \cite{berndt}).

\subsection{The Borweins' Cubic Theta Functions}
In \cite{borwein1}, the Borweins' introduced the so-called ``cubic theta functions'' $a(q), b(q)$ and $c(q)$. They are so named because they can be related by a formula similar to proposition \ref{jacobiprops} (1) that involves third powers.

\begin{definition}[The Borweins' Cubic Theta Functions]
\leavevmode
\begin{enumerate}[(1)]
    \item $\displaystyle a(q)\coloneqq\sum_{m,n\in\Z}q^{n^2+nm+m^2}$
    \item $\displaystyle b(q)\coloneqq\sum_{m,n\in\Z}\nu_3^{n-m}q^{n^2+nm+m^2}$
    \item $\displaystyle c(q)\coloneqq\sum_{m,n\in\Z}q^{(n+\frac{1}{3})^2+(n+\frac{1}{3})(m+\frac{1}{3})+(m+\frac{1}{3})^2}$
\end{enumerate}
\end{definition}
We will now state some identities relating the cubic theta functions to each other and to some of the other theta functions defined above.

\begin{proposition}
\label{CubicThetaProps}
\leavevmode
\begin{enumerate}[(1)]
    \item $a(q) = \theta_3(q)\theta_3(q^3)+\theta_2(q)\theta_2(q^3)$
    \item $b(q) = \dfrac{\eta_1^3}{\eta_3} = \dfrac{\varphi_1^3}{\varphi_3}$
    \item $c(q) = \dfrac{3\eta_3^3}{\eta_1} = \dfrac{3q^{1/3}\varphi_3^3}{\varphi_1}$
    \item $a(q^4) = \dfrac{1}{2}(\theta_3(q)\theta_3(q^3)+\theta_4(q)\theta_4(q^3))$
    \item $a(q)^3 = b(q)^3+c(q)^3$
    \item $b(q) = \dfrac{3}{2}a(q^3)-\dfrac{1}{2}a(q)$
    \item $c(q) = \dfrac{1}{2}a(q^{1/3})-\dfrac{1}{2}a(q)$
    \item $a(q)+a(-q) = 2a(q^4)$
    \item $b(q)+b(-q) = 2b(q^4)$
    \item $c(q)+c(-q) = 2c(q^4)$
    \item $b(q) = a(q^3)-c(q^3)$
    \item $a(q)-a(q^4) = 6q\psi(q^2)\psi(q^6)$
    \item $a(q)a(q^2) = \dfrac{1}{4}(6P(q^6)-3P(q^3)+2P(q^2)-P(q))$
    \item $b(q)b(q^2) = \dfrac{1}{8}(18P(q^6)-9P(q^3)-2P(q^2)+P(q))$
    \item $c(q)c(q^2) = \dfrac{3}{8}(-2P(q^6)+P(q^3)+2P(q^2)-P(q))$
\end{enumerate}
\end{proposition}

\begin{proof}
Identities 1, 4, 6 and 7 are lemma 2.1 (i)(a), (i)(b), (ii) and (iii) in \cite{borwein2} respectively. Identities 2 and 3 are proposition 2.2(i) and (ii) in \cite{borwein2} respectively. Identity 5 is theorem 2.3 in \cite{borwein2}. Identities 8-10 are given in the proof of theorem 2.6 in \cite{borwein2}. Identity 11 follows from 6 and 7 since $$a(q^3)-c(q)^3 = a(q^3)-\dfrac{1}{2}a(q)+\dfrac{1}{2}a(q^3) = \dfrac{3}{2}a(q^3)-\dfrac{1}{2}a(q) = b(q).$$ For identity 12, notice that $$a(q)-a(q^4) = a(q)-\dfrac{1}{2}(a(q)+a(-q))$$ (by 8) which equals $$\dfrac{1}{2}(a(q)-a(-q)).$$ At the top of page 41 in \cite{borwein2}, it is shown that $$a(q)-a(-q) = 3\theta_2(q)\theta_2(q^3),$$ so $$\dfrac{1}{2}(a(q)-a(-q)) = \dfrac{3}{2}\theta_2(q)\theta_2(q^3) = 6q\psi(q^2)\psi(q^6).$$ A result equivalent to identity 13 was first stated by Liouville in \cite{Liouville} without proof. A proof is given by Alaca, Alaca and Williams in \cite{Alaca} (see theorem 13). Finally, 14 and 15 were proven by Fine (see equations 33.1 and 33.124 on page 86 of \cite{fine}).
\end{proof}

\begin{corollary}
\label{oops}
$\phi(-q)\phi(-q^3)+2q\psi(q^2)\psi(q^6) = a(q^4)$.
\end{corollary}

\begin{proof}
Recall that $\phi(-q) = \theta_4(q)$ and $\psi(q) = \dfrac{1}{2}q^{-1/8}\theta_2(q^{1/2})$. Our desired identity is then equivalent to $\theta_4(q)\theta_4(q^3)+\dfrac{1}{2}\theta_2(q)\theta_2(q^3) = a(q^4)$. By proposition \ref{jacobiprops} (1) and (4), $$a(q)-a(q^4) = \theta_3(q)\theta_3(q^3)+\theta_2(q)\theta_2(q^3)-\dfrac{1}{2}\theta_3(q)\theta_3(q^3)-\dfrac{1}{2}\theta_4(q)\theta_4(q^3),$$ but by proposition \ref{jacobiprops} (12), $$a(q)-a(q^4) = 6q\psi(q^2)\psi(q^6) = \dfrac{3}{2}\theta_2(q)\theta_2(q^3).$$ So $$\dfrac{1}{2}\theta_3(q)\theta_3(q^3)+\theta_2(q)\theta_2(q^3)-\dfrac{1}{2}\theta_4(q)\theta_4(q^3) = \dfrac{3}{2}\theta_2(q)\theta_2(q^3)$$ or, equivalently $$\dfrac{1}{2}\theta_3(q)\theta_3(q^3) = \dfrac{1}{2}\theta_2(q)\theta_2(q^3)-\dfrac{1}{2}\theta_4(q)\theta_4(q^3).$$ Thus, 
\begin{align*}
    \theta_4(q)\theta_4(q^3)+\dfrac{1}{2}\theta_2(q)\theta_2(q^3) &= \dfrac{1}{2}\theta_4(q)\theta_4(q^3)+\dfrac{1}{2}\theta_4(q)\theta_4(q^3)+\dfrac{1}{2}\theta_2(q)\theta_2(q^3)\\ 
    &= \dfrac{1}{2}\theta_3(q)\theta_3(q^3)+\dfrac{1}{2}\theta_4(q)\theta_4(q^3).
\end{align*}
But by proposition \ref{jacobiprops} (4), this equals $a(q^4)$ as desired.
\end{proof}

\section{More Classical Identities}
We prove two more sets of identities. We start by showing summation formulas for four $\eta$-quotients, which will be useful to us. Although these formulas are stated in several sources, their proofs are not given as they are straightforward applications of the JTP.

\begin{proposition}
\label{JTPids}
\leavevmode
\begin{enumerate}[(1)]
    \item $\displaystyle\frac{\eta_2\eta_3^2}{\eta_1\eta_6} = \sum_{n = -\infty}^\infty q^{(6n+1)^2/24}$
    \item $\displaystyle\frac{\eta_1\eta_6^2}{\eta_2\eta_3} = \sum_{n = -\infty}^\infty(-1)^nq^{(3n+1)^2/3}$
    \item $\displaystyle\frac{\eta_1^2\eta_6}{\eta_2\eta_3} = \sum_{n = -\infty}^\infty\nu_6^{2n+1}q^{(2n+1)^2/8}$
    \item $\displaystyle\frac{\eta_2^2\eta_3}{\eta_1\eta_6} = \sum_{n = -\infty}^\infty\nu_6^nq^{n^2}$.
\end{enumerate}
\end{proposition}

\begin{proof}
We begin by mentioning two elementary ideas: First, $$\displaystyle(q^a;q^b)_\infty = \prod_{\substack{k>0 \\ k\equiv a\mod b}}(1-q^k).$$ Second, $(q;q)_\infty(\nu_kq;q)_\infty(\nu_k^2q;q)\cdots(\nu_k^{k-1}q;q)_\infty = (q^k;q^k)_\infty$. We also note that $$\displaystyle(-q^a;q^b)_\infty = \prod_{k>0}(1+q^{kb+a}) = \prod_{k>0}\frac{1-q^{2kb+2a}}{1-q^{kb+a}} = \frac{(q^{2a};q^{2b})_\infty}{(q^a;q^b)_\infty}.$$

For (1), we rewrite the sum as $$q^{1/24}\sum_{n = -\infty}^\infty(q^{3/2})^{n^2}(q^{1/2})^n.$$ Applying the JTP, we get that this equals
\begin{align}
 q^{1/24}(-q^2;q^3)_\infty(-q;q^3)_\infty(q^3;q^3)_\infty &= q^{1/24}\frac{(q^4;q^6)_\infty}{(q^2;q^3)_\infty}\frac{(q^2;q^6)_\infty}{(q;q^3)_\infty}(q^3;q^3)_\infty\\
 &= q^{1/24}\frac{\varphi_2/\varphi_6}{\varphi_1/\varphi_3}\varphi_3\\
 &= q^{1/24}\frac{\varphi_2\varphi_3^2}{\varphi_1\varphi_6} = \frac{\eta_2\eta_3^2}{\eta_1\eta_6}.
\end{align}

For (2), we will again rewrite the sum and apply the JTP:
\begin{align}
\sum_{n = -\infty}^\infty(-1)^nq^{(3n+1)^2/3} &= q^{1/3}\sum_{n = -\infty}^\infty(q^3)^{n^2}(-q^2)^n\\
&= q^{1/3}(q^5;q^6)_\infty(q;q^6)_\infty(q^6;q^6)_\infty\\
&= q^{1/3}\frac{(q;q)_\infty}{(q^2;q^6)_\infty(q^4;q^6)_\infty(q^3;q^6)_\infty}\\
&= q^{1/3}\frac{\varphi_1}{(\varphi_2/\varphi_6)(\varphi_3/\varphi_6)}\\
&= q^{1/3}\frac{\varphi_1\varphi_6^2}{\varphi_2\varphi_3} = \frac{\eta_1\eta_6^2}{\eta_2\eta_3}.
\end{align}

We will follow the same process for (3) and (4) to get:
\begin{align}
\sum_{-\infty}^\infty\nu_6^{2n+1}q^{(2n+1)^2/8} &= \nu_6 q^{1/8}\sum_{-\infty}^\infty(q^{1/2})^{n^2}(\nu_6^2q^{1/2})^n\\
&= \nu_6q^{1/8}(\nu_6^5q;q)_\infty(\nu_6;q)_\infty(q;q)_\infty\\
&= q^{1/8}\nu_6(1-\nu_6)(\nu_6^5q;q)_\infty(\nu_6q;q)_\infty(q;q)_\infty\\
&= q^{1/8}\frac{(q^6;q^6)_\infty}{(\nu_6^2q;q)_\infty(\nu_6^3q;q)_\infty(\nu_6^4q;q)_\infty}\\
&= q^{1/8}\frac{(q^6;q^6)_\infty}{(\nu_3q;q)_\infty(-q;q)_\infty(\nu_3^2q;q)_\infty}\\
&= q^{1/8}\frac{\varphi_6}{(\varphi_3/\varphi_1)(\varphi_2/\varphi_1)}\\
&= q^{1/8}\frac{\varphi_1^2\varphi_6}{\varphi_2\varphi_3} = \frac{\eta_1^2\eta_6}{\eta_2\eta_3}
\end{align}
(where in line (23) we use the fact that $\nu_6(1-\nu_6) = 1$) and
\begin{align}
\sum_{n = -\infty}^\infty\nu_6^nq^{n^2} & = (\nu_6^4q;q^2)_\infty(\nu_6^2q;q^2)_\infty(q^2;q^2)_\infty\\
&= (\nu_3^2q;q^2)_\infty(\nu_3q;q^2)_\infty(q^2;q^2)_\infty\\
&= \frac{(q^3;q^6)_\infty}{(q;q^2)_\infty}(q^2;q^2)_\infty\\
&= \frac{\varphi_3/\varphi_6}{\varphi_1/\varphi_2}\varphi_2\\
&= \frac{\varphi_2^2\varphi_3}{\varphi_1\varphi_6} = \frac{\eta_2^2\eta_3}{\eta_1\eta_6}
\end{align}
which completes our proof.
\end{proof}
Note that the third of our observations at the beginning of the previous proof will give us $\phi(q) = \dfrac{\varphi_2^5}{\varphi_1^2\varphi_4^2}$ and $\psi(q) = \dfrac{\varphi_2^2}{\varphi_1}$. One last identity that will be helpful is the following:
\begin{proposition}
\label{a12c12}
$$a(q)a(q^2)+2c(q)c(q^2) = \frac{1}{2}(\theta_3(q^{1/2})^4+\theta_4(q^{1/2})^4).$$
\end{proposition}
\begin{proof}
By  Proposition \ref{CubicThetaProps} (13) and (15), we have $a(q)a(q^2)+2c(q)c(q^2) = 2P(q^2)-P(q)$. By Proposition \ref{PhiEis}, $\phi(q)^4+\phi(-q)^4 = \theta_3(q)^4+\theta_4(q)^4 = 4P(q^4)-2P(q^2)$. The result follows by multiplying this last expression by $\frac{1}{2}$ and substituting $q^{1/2}$ for $q$.
\end{proof}

\section{Some Quotient Identities}
In this section we present four $\eta$-quotients and prove some nice identities relating them to Ramanujan's and  the Borweins' cubic theta functions. These identities will be needed later in our character calculation. 
\begin{definition}
Define the following $q$-series: $$w(q)\coloneqq \dfrac{\varphi_2^2\varphi_3^3}{\varphi_1\varphi_6^2}\text{, }x(q)\coloneqq \dfrac{\varphi_1^3\varphi_6^2}{\varphi_2^2\varphi_3}\text{, }y(q)\coloneqq \dfrac{\varphi_1^2\varphi_6^3}{\varphi_2\varphi_3^2}\text{ and }z(q)\coloneqq \dfrac{\varphi_2^3\varphi_3^2}{\varphi_1^2\varphi_6}.$$
\end{definition}


\begin{theorem}
\label{wDissection}
$w(q) = \phi(-q^3)\phi(-q^9)+qx(q^3)$.
\end{theorem}

\begin{proof}
We begin by noting that $w(q) = \phi(-q^3)\dfrac{\eta_2^2\eta_3}{\eta_1\eta_6}$. Now, by Proposition \ref{JTPids} (4) we get that $\displaystyle\frac{\eta_2^2\eta_3}{\eta_1\eta_6} = \sum_{n = -\infty}^\infty\nu_6^nq^{n^2}$. We will split this sum up based on the congruence of the power of $q$ modulo 3. So
\begin{align}
\sum_{n = -\infty}^\infty\nu_6^nq^{n^2} &= \sum_{n\equiv 0\mod 3}\nu_6^nq^{n^2}+\sum_{n\equiv 1\mod 3}\nu_6^nq^{n^2}+\sum_{n\equiv -1\mod 3}\nu_6^nq^{n^2}\\
&= \sum_{k = -\infty}^\infty\nu_6^{3k}q^{9k^2}+\sum_{k = -\infty}^\infty\nu_6^{3k+1}q^{(3k+1)^2}+\sum_{k = -\infty}^\infty\nu_6^{-3k-1}q^{(-3k-1)^2}\\
&=\sum_{k = -\infty}^\infty(-1)^k(q^9)^{k^2}+\sum_{k = -\infty}^\infty\nu_6(-1)^kq^{(3k+1)^2}+\sum_{k = -\infty}^\infty\nu_6^{-1}(-1)^kq^{(3k+1)^2}.
\end{align}
Since $k$ is even if and only if $k^2$ is even, we have $(-1)^k = (-1)^{k^2}$ and so $$\sum_{k = -\infty}^\infty(-1)^k(q^9)^{k^2} = \sum_{k = -\infty}^\infty(-q^9)^{k^2} = \phi(-q^9).$$ Also, note that $\nu_6+\nu_6^{-1} = 1$ and so $$\sum_{k = -\infty}^\infty\nu_6(-1)^kq^{(3k+1)^2}+\sum_{k = -\infty}^\infty\nu_6^{-1}(-1)^kq^{(3k+1)^2} = \sum_{k = -\infty}^\infty(-1)^kq^{(3k+1)^2} = \frac{\eta_3\eta_{18}^2}{\eta_6\eta_9}$$ where the final equality follows from Proposition \ref{JTPids} (2). Putting these results together, we get
\begin{align}
w(q) &= \phi(-q^3)\frac{\eta_2^2\eta_3}{\eta_1\eta_6}\\
&= \phi(-q^3)\left(\phi(-q^9)+\frac{\eta_3\eta_{18}^2}{\eta_6\eta_9}\right)\\
&= \phi(-q^3)\phi(-q^9)+\frac{\varphi_3^2}{\varphi_6}q\frac{\varphi_3\varphi_{18}^2}{\varphi_6\varphi_9}\\
&=\phi(-q^3)\phi(-q^9)+qx(q^3).
\end{align}
\end{proof}

\begin{theorem}
\label{yDissection}
$y(q) = z(q^3)-2q\psi(q^3)\psi(q^9)$.
\end{theorem}

\begin{proof}
We will again begin by splitting our quotient: $y(q) = \psi(q^3)q^{-3/24}\dfrac{\eta_1^2\eta_6}{\eta_2\eta_3}$. By Proposition \ref{JTPids} (3), $\displaystyle q^{-3/24}\frac{\eta_1^2\eta_6}{\eta_2\eta_3} = q^{-1/8}\sum_{n = -\infty}^\infty\nu_6^{2n+1}q^{(2n+1)^2/8} = \sum_{n = -\infty}^\infty\nu_6^{2n+1}q^{(n^2+n)/2}$. As before, we will break this sum up by congruence class of $n$ modulo 3:
\begin{align}
\sum_{n = -\infty}^\infty\nu_6^{2n+1}q^{(n^2+n)/2} =& \sum_{n\equiv 0\mod 3}\nu_6^{2n+1}q^{(n^2+n)/2}+\sum_{n\equiv 1\mod 3}\nu_6^{2n+1}q^{(n^2+n)/2}\nonumber\\
&+\sum_{n\equiv -1\mod 3}\nu_6^{2n+1}q^{(n^2+n)/2}\\
=& \sum_{k = -\infty}^\infty\nu_6^{2(3k)+1}q^{((3k)^2+3k)/2}+\sum_{k = -\infty}^\infty\nu_6^{2(3k+1)+1}q^{((3k+1)^2+3k+1)/2}\nonumber\\
&+\sum_{k = -\infty}^\infty\nu_6^{2(-3k-1)+1}q^{((-3k-1)^2-3k-1)/2}\\
=& \sum_{k = -\infty}^\infty\nu_6^{6k+1}q^{(9k^2+3k)/2}+\sum_{k = -\infty}^\infty\nu_6^{6k+3}q^{(9k^2+9k+2)/2}\nonumber\\
&+\sum_{k = -\infty}^\infty\nu_6^{-6k-1}q^{(9k^2+3k)/2}.
\end{align}
Now since $\nu_6^{6k+3} = -1$ for all $k\in\Z$, we have $$\sum_{k = -\infty}^\infty\nu_6^{6k+3}q^{(9k^2+9k+2)/2} = -q\sum_{k = -\infty}^\infty(q^9)^{k(k+1)/2} = -2q\psi(q^9).$$ Further, $\nu_6^{6k+1} = \nu_6$ and $\nu_6^{-6k-1} = \nu_6^{-1}$ for all $k\in\Z$, so
\begin{align}
\sum_{k = -\infty}^\infty\nu_6^{6k+1}q^{(9k^2+3k)/2}&+\sum_{k = -\infty}^\infty\nu_6^{6k+3}q^{(9k^2+9k+2)/2}+\sum_{k = -\infty}^\infty\nu_6^{-6k-1}q^{(9k^2+3k)/2}\\
&= (\nu_6+\nu_6^{-1})\sum_{k = -\infty}^\infty q^{(9k^2+3k)/2}\\
&= \sum_{k = -\infty}^\infty q^{\frac{9k^2+3k}{2}\cdot3\cdot\frac{4}{12}}\\
&= \sum_{k = -\infty}^\infty q^{3(36k^2+12k)/24}\\
&= q^{-3/24}\sum_{k = -\infty}^\infty q^{3(36k^2+12k+1)/24}\\
&= q^{-3/24}\sum_{k = -\infty}^\infty q^{3(6k+1)^2/24}\\
&= q^{-3/24}\frac{\eta_6\eta_9^2}{\eta_3\eta_{18}}\\
&= \frac{\varphi_6\varphi_9^2}{\varphi_3\varphi_{18}}
\end{align}
where line 49 follows from 48 by Proposition \ref{JTPids} (1). Putting everything together, we get:
\begin{equation}
    y(q) = \psi(q^3)\left(\frac{\varphi_6\varphi_9^2}{\varphi_3\varphi_{18}}-2q\psi(q^9)\right) = \frac{\varphi_6^2}{\varphi_3}\frac{\varphi_6\varphi_9^2}{\varphi_3\varphi_{18}}-2q\psi(q^3)\psi(q^9) = z(q^3)-2q\psi(q^3)\psi(q^9)
\end{equation}
\end{proof}

\begin{lemma}
\label{ellipse}
Let $f(x,y) = 2x^2+2xy+2y^2+x+y$ and $i, n\in\Z$. Then there exists $j\in\Z$ such that $f(i,j) = n$ if and only if there exists $k\in\Z$ such that $f(i,k+\frac{1}{2}) = n.$
\end{lemma}

\begin{proof}
We will rearrange $f(x,y)=n$ to obtain $2y^2+(2x+1)y+(2x^2+x-n)=0$ and use the quadratic formula to find a formula for $y$ in terms of $x$. This gives us:
$$y = \frac{-(2x+1)\pm\sqrt{(2x+1)^2-4(2)(2x^2+x-n)}}{4}.$$
Let $a(x) = -(2x+1)$ and $b(x) = \sqrt{(2x+1)^2-4(2)(2x^2+x-n)}$ and set $y_1(x) = \frac{1}{4}(a(x)+b(x))$ and $y_2(x) = \frac{1}{4}(a(x)-b(x))$. Now for $i\in\Z$, we have $a(i)$ is odd and thus either $a(i)\equiv 1\mod 4$ or $a(i)\equiv 3\mod 4$. 

Then $y_1(i)\in\Z$ if and only if $a(i)+b(i)\in 4\Z$. This is true if and only if either $a(i)\equiv 1\mod 4$ and $b(i)\equiv 3\mod 4$ or $a(i)\equiv 3\mod 4$ and $b(i)\equiv 1\mod 4$. Both of these cases are equivalent to saying that $a(i)-b(i)\equiv 2\mod 4$, which is true if and only if $y_2(i)\in\Z+\frac{1}{2}.$

Similarly, $y_2(i)\in\Z$ if and only if $a(i)-b(i)\in 4\Z$. This is true if and only if either $a(i)\equiv 1\mod 4$ and $b(i)\equiv 1\mod 4$ or $a(i)\equiv 3\mod 4$ and $b(i)\equiv 3\mod 4$. Both of these cases are equivalent to saying that $a(i)+b(i)\equiv 2\mod 4$, which is true if and only if $y_1(i)\in\Z+\frac{1}{2}.$ Our claim now follows immediately.
\end{proof}

\begin{lemma}
\label{EllipticSeries}
Let $f(x,y)$ be as above and let $g(x,y) = 2x^2+2xy+2y^2+2x+3y+1$. Then $\sum_{i,j\in\Z}q^{f(a,b)} = \sum_{i,j\in\Z}q^{g(a,b)}$.
\end{lemma}
\begin{proof}
First, note that for any integer $z\in\Z$, we have $z^2+z\geq 0$. Then $f(x,y) = (x+y)^2+(x^2+x)+(y^2+y)\geq 0$ for all $x,y\in\Z$. Similarly, $g(x,y) = [(x+y)^2+(x+y)]+(x^2+x)+(y+1)^2\geq 0$. Now let $A(n) = \{(i,j)\in\Z^2\mid f(i,j) = n\}$ and $B(n) = \{(i,j)\in\Z^2\mid g(i,j) = n\}$. Then $\displaystyle\sum_{i,j\in\Z}q^{f(a,b)} = \sum_{n = 0}^\infty|A(n)|q^n$ and $\displaystyle\sum_{i,j\in\Z}q^{g(a,b)} = \sum_{n = 0}^\infty|B(n)|q^n$. We will be done if we can show $|A(n)| = |B(n)|$ for all nonnegative $n\in\Z$. By lemma \ref{ellipse}, we have that if $(i,j)\in A(n)$ then there exists $k\in\Z$ such that $f(i,k+\frac{1}{2}) = n$. But since $g(x,y) = f(x,y+\frac{1}{2})$, we have that $f(i,k+\frac{1}{2}) = g(i,k) = n$. This implies that for any $(i,j)\in A(n)$, there exists $k\in\Z$ such that $(i,k)\in B(n)$. Thus, $|A(n)|\leq|B(n)|$. Similarly, if $(i,k)\in B(n)$ then $g(i,k) = f(i,k+\frac{1}{2}) = n$. By lemma \ref{ellipse}, there exists $j\in\Z$ such that $f(i,j) = n$ and so $|B(n)|\leq|A(n)|$. Therefore, $|A(n)| = |B(n)|$.
\end{proof}

\begin{lemma}
\label{zSum}
$$\sum_{i,j\in\Z}q^{f(x,y)} = z(q).$$
\end{lemma}

\begin{proof}
First, note that
\begin{align}
z(q) &= \frac{\varphi_2^2}{\varphi_1}\cdot\frac{q^{-1/24}\eta_2\eta_3^2}{\eta_1\eta_6}\\
&= \frac{1}{2}\left(\sum_{m\in\Z}q^{m(m+1)/2}\right)q^{-1/24}\left(\sum_{n\in\Z}q^{(6n+1)^2/24}\right)\\
&= \frac{1}{2}q^{-1/24}\sum_{m,n\in\Z}q^{\frac{1}{2}m^2+\frac{1}{2}m+\frac{36}{24}n^2+\frac{12}{24}n+\frac{1}{24}}\\
&= \frac{1}{2}\sum_{m,n\in\Z}q^{\frac{1}{2}m^2+\frac{1}{2}m+\frac{3}{2}n^2+\frac{1}{2}n}.
\end{align}
Now we will make the substitution $x = \frac{1}{2}(m+n)$ and $y =\frac{1}{2}(m-n)$. Note that this is an invertible system of linear equations and so we get $m = x+y$ and $n = x-y$ with $x$ and $y$ ranging over $\frac{1}{2}\Z$ such that $x\pm y\in\Z$. So,
\begin{align}
\frac{1}{2}\sum_{m,n\in\Z}q^{\frac{1}{2}m^2+\frac{1}{2}m+\frac{3}{2}n^2+\frac{1}{2}n} &= \frac{1}{2}\sum_{\substack{x,y\in\frac{1}{2}\Z\\ x\pm y\in\Z}}q^{2x^2-2xy+2y^2+x}\\
&= \frac{1}{2}\left(\sum_{x,y\in\Z}q^{2x^2-2xy+2y^2+x}+\sum_{x,y\in\Z+\frac{1}{2}}q^{2x^2-2xy+2y^2+x}\right)\\
&= \frac{1}{2}\left(\sum_{x,y\in\Z}q^{2x^2-2xy+2y^2+x}+\sum_{x,y\in\Z}q^{2x^2-2xy+2y^2+2x+y+1}\right).
\end{align}
Now, notice that
\begin{align}
\sum_{\substack{i,j\in\Z\\i\text{ even}\\j\text{ odd}}}q^{(i^2+ij+j^2+i+j)/2} &= \sum_{m,n\in\Z}q^{2m^2+2mn+2n^2+2m+3n+1}\\ 
&= \sum_{m,n\in\Z}q^{g(m,n)}
\end{align}
where we made the substitutions $i = 2m$ and $j = 2n+1$. But we also have that
$$\sum_{\substack{i,j\in\Z\\i\text{ even}\\j\text{ odd}}}q^{(i^2+ij+j^2+i+j)/2} = \sum_{m,n\in\Z}q^{2m^2-2mn+2n^2+m}$$
where we made the substitutions $i = -2n$ and $j = 2m-1$. So we must have that $$\sum_{x,y\in\Z}q^{2x^2-2xy+2y^2+x} = \sum_{x,y\in\Z}q^{g(x,y)} = \sum_{x,y\in\Z}q^{f(x,y)}$$ by line 60 and lemma \ref{EllipticSeries}. Additionally,
\begin{align}
\sum_{i,j\in\Z}q^{g(i,j)} &= \sum_{i,j\in\Z}q^{2i^2+2ij+2j^2+2i+3j+1}\\
&= \sum_{k,j\in\Z}q^{2k^2-2kj+2j^2-2k+3j+1}\\
&= \sum_{l,j\in\Z}q^{2l^2-2lj+2j^2+2l+j+1}
\end{align}
where we made the substitutions $i = -k$ and $k = l+1$. So then $$\sum_{x,y\in\Z}q^{2x^2-2xy+2y^2+x} = \sum_{x,y\in\Z}q^{f(x,y)}$$ by lemma \ref{EllipticSeries}. Putting all of these results together, we have:
\begin{align}
\frac{1}{2}\left(\sum_{x,y\in\Z}q^{2x^2-2xy+2y^2+2x+y+1}+\sum_{x,y\in\Z}q^{2x^2-2xy+2y^2+2x+y+1}\right) &= \frac{1}{2}\left(\sum_{x,y\in\Z}q^{f(x,y)}+\sum_{x,y\in\Z}q^{f(x,y)}\right)\\
&= \sum_{x,y\in\Z}q^{f(x,y)}
\end{align}
which was what we wanted.
\end{proof}

\begin{theorem}
\label{xDissection}
$x(q) = z(q^2)-q^{-\frac{1}{3}}c(q^4)$
\end{theorem}

\begin{proof}
We will proceed similarly to the proofs of our claims about $w(q)$ and $y(q)$. First we note that $$x(q) = \frac{\varphi_6^2}{\varphi_2^2}b(q)$$ (which follows from Proposition \ref{CubicThetaProps} (2)) and proceed by splitting $b(q)$ by odd and even powers of $q$. To isolate even powers of $q$ in $b(q)$ we will calculate $\frac{1}{2}(b(q)+b(-q))$, but this equals $b(q^4)$ by Proposition \ref{CubicThetaProps} (9). To isolate odd powers of $q$ in $b(q)$, we first use Proposition \ref{CubicThetaProps} (11) to write $b(q) = a(q^3)-c(q^3)$. To extract terms with odd powers of $q$ in $b(q)$, we will extract the terms with odd powers of $q$ from both $a(q^3)$ and $c(q^3)$. For $a(q^3)$, we note that $\frac{1}{2}(a(q)-a(-q)) = 6q\psi(q^2)\psi(q^6)$ by Proposition \ref{CubicThetaProps} (12) and so $\frac{1}{2}(a(q^3)-a(-q^3)) = 6q^3\psi(q^6)\psi(q^{18})$.
For ease of notation, we will use the operators $\odd()$ and $\even()$ to denote the isolation of odd and even terms of a given series respectively. Let $d(q) = q^{-1/3}c(q)$. Then $d(q^3) = q^{-1}c(q^3)$ and thus $\odd(c(q^3)) = q\even(d(q^3))$. But $\even(d(q^3)) = [\even(d(x))]_{x = q^3}$, which reduces our problem to isolating the even terms of $d(q)$. By definition, $\displaystyle d(q) = \sum_{i,j\in \Z}q^{i^2+ij+j^2+i+j}$ and $i^2+ij+j^2+i+j$ will be even if and only if at most one of $i$ or $j$ is odd. That is, $$\even(d(q)) = \sum_{\substack{i,j\in\Z \\ i,j\text{ even}}}q^{i^2+ij+j^2+i+j}+\sum_{\substack{i,j\in\Z \\ i\text{ even, }j\text{ odd}}}q^{i^2+ij+j^2+i+j}+\sum_{\substack{i,j\in\Z \\ i\text{ odd, }j\text{ even}}}q^{i^2+ij+j^2+i+j}.$$ Noting symmetry in the last two summations and substituting $i = 2m$, $j = 2n$ and $i = 2m$, $j = 2n+1$ in the first and second summation respectively, we have 
\begin{align}
\even(d(q)) &= \sum_{m,n\in\Z}q^{4m^2+4mn+4n^2+2m+2n}+2\sum_{m,n\in\Z}q^{4m^2+4mn+4n^2+2m+6n+2}\\ 
&= \left[\sum_{m,n\in\Z}x^{f(m,n)}+2\sum_{m,n\in\Z}x^{g(m,n)}\right]_{x = q^2}.
\end{align}
But by lemma \ref{EllipticSeries} and lemma \ref{zSum}, line 67 equals: $$\left[3\sum_{m,n\in\Z}x^{f(m,n)}\right]_{x = q^2} = [3z(x)]_{x = q^2} = 3z(q^2).$$ Therefore,
\begin{align}
\odd(b(q)) &= \odd(a(q^3))-\odd(c(q^3))\\
&= 6q^3\psi(q^6)\psi(q^{18})-q\even(d(q^3))\\
&= 6q^3\psi(q^6)\psi(q^{18})-q[\even(d(x))]_{x = q^3}\\
&= 6q^3\psi(q^6)\psi(q^{18})-q[3z(x^2)]_{x = q^3}\\
&= 6q^3\psi(q^6)\psi(q^{18})-3qz(q^6)\\
&= -3qy(q^2)
\end{align}
where line 73 follows from 72 by Theorem \ref{yDissection}. Putting our results together, we have $$x(q) = \frac{\varphi_6^2}{\varphi_2^2}b(q) = \frac{\varphi_6^2}{\varphi_2^2}(b(q^4)-3qy(q^2)) = \frac{\varphi_6^2}{\varphi_2^2}\cdot\frac{\varphi_4^3}{\varphi_{12}}-\frac{\varphi_6^2}{\varphi_2^2}\cdot\frac{3q\varphi_2^2\varphi_{12}^3}{\varphi_4\varphi_6^2} = z(q^2)-3q\frac{\varphi_{12}^3}{\varphi_4}$$
$$=  z(q^2)-3q^{-1/3}c(q^4)$$ which was what we wanted.
\end{proof}

\begin{corollary}
\label{zDissection}
$$q^{1/6}z(q) = \frac{1}{3}(c(q^{1/2})-c(q^2)).$$
\end{corollary}

\begin{proof}
In the proof of Theorem \ref{xDissection}, we showed that $\odd(c(q^3)) = 3qz(q^6)$. We can also calculate $\odd(c(q^3))$ as follows:
\begin{align}
\odd(c(q^3)) &= \frac{1}{2}(c(q^3)-c(-q^3))\\
&= c(q^3) - \frac{1}{2}(c(q^3)+c(-q^3))\\
&= c(q^3)-c(q^{12})
\end{align}
where line 76 follows from 75 by \ref{CubicThetaProps} (10). So $3qz(q^6) = c(q^3)-c(q^{12})$. Substituting $q = q^{1/6}$, we have $3q^{1/6}z(q) = c(q^{1/2})-c(q^2)$ which will give the desired result.
\end{proof}

\chapter{Character Calculation}
In this chapter, we gather all of the information necessary to setup and complete our character calculation, proving our main result.

\section{Theta Functions for $F_4$ and $G_2$}
We first determine which theta functions will be needed for our characters. All of the necessary $F_4$ and $G_2$ theta functions turn out to be equivalent to theta series for the $A_2$ and $D_4$ lattices which were calculated by Conway and Sloane. We begin by noting how our theta functions relate to their theta series of a lattice.

\subsection{Two Theta Functions}
In \cite{conway}, Conway and Sloane define the theta series of a lattice, $L$, as follows: $$\Theta_L(q) = \sum_{x\in L}q^{(x,x)}.$$ Similarly, for a translation of a lattice, $L+a$, $$\Theta_{L+a}(q) = \sum_{x\in L+a}q^{(x,x)}.$$ Kac and Peterson, on the other hand, define their theta functions by $$\Theta_{\lambda} = e^{k\Lambda_0}\sum_{\gamma\in M^*+k^{-1}\overline{\lambda}}e^{-\frac{1}{2}k|\gamma|^2\delta+k\gamma}$$ where $k = \lambda(c)$ and $\Lambda_0$ is the basic fundamental weight of the affine algebra in question. By taking the horizontal specialization (and noting that we will only be working with level 1 modules), we see that $$(\Theta_{\lambda})_{\hs} = e^{\Lambda_0}\sum_{\gamma\in M^*+\overline{\lambda}}q^{\frac{1}{2}(\gamma,\gamma)}.$$ This gives us that $$e^{-\Lambda_0}(\Theta_{\lambda})_{\hs} = \Theta_{M^*+\overline{\lambda}}(q^{1/2}).$$ In what follows we use the following notation: $Q_{X_n}$ is the root lattice of finite type $X_n$ and $P_{X_n}$ is the weight lattice of finite type $X_n$. The addition of a hat will denote the root or weight lattice of the corresponding affine algebra of type $X_n^{(1)}$. A superscript $k$ on the affine weight lattice will denote the subset of the affine weight lattice of all weights of level $k$. Finally, we will use the subscripts $F_4$ and $G_2$ on $\delta$ to differentiate between $\widehat{Q}_{F_4}$ and $\widehat{Q}_{G_2}$.

\subsection{$F_4$ Weight Lattice and Theta Functions}
We have that $$\widehat{P}^1_{F_4} = \{\Omega_0+\omega+n\delta_{F_4}\mid\omega\in P_{F_4},n\in\C\}.$$ 
But since $P_{F_4} = Q_{F_4}$, we will have 
$$\widehat{P}^1_{F_4} = \{\Omega_0+a\beta_1+b\beta_2+c\beta_3+d\beta_4+n\delta_{F_4}\mid a,b,c,d\in\Z,n\in\C\}.$$ 
To compute our character, recall that we need to find a set of coset representatives of $\widehat{P}^1_{F_4}/(M^*_{F_4}+\C\delta_{F_4})$ where 
$$M^*_{F_4} = \{a\beta_1+b\beta_2+2c\beta_3+2d\beta_4\mid a,b,c,d\in\Z\}$$ 
which will determine which string and theta functions we must compute. Because $\Omega_0\not\in M^*_{F_4}+\C\delta_{F_4}$, $n\delta_{F_4}\in\C\delta_{F_4}$ for all $n\in\C$ and $a\beta_1+b\beta_2+c\beta_3+d\beta_4\in M^*_{F_4}$ only when $c$ and $d$ are even, we can take our coset representatives to be $\Omega_0, \Omega_0+\beta_3$, $\Omega_0+\beta_4$ and $\Omega_0+\beta_3+\beta_4$.
Let $\beta_1'\coloneqq\beta_1$, $\beta_2'\coloneqq\beta_2$, $\beta_3'\coloneqq-\beta_1-2\beta_2-2\beta_3-2\beta_4$ and $\beta_4'\coloneqq \beta_1+2\beta_2+4\beta_3+2\beta_4$. It follows that $\beta_1',\beta_2',\beta_3'$ and $\beta_4'$ form a root system of type $D_4$. Further, $\Z\{\beta_1',\beta_2',\beta_3',\beta_4'\} = M^*_{F_4}$ and so 
$$e^{-\Omega_0}(\Theta_{\Omega_0})_{\hs} = \Theta_{M^*_{F_4}}(q^{1/2}) = \Theta_{D_4}(q^{1/2}).$$ 
Conway and Sloane give $\Theta_{D_4}(q) = \frac{1}{2}(\theta_3(q)^4+\theta_4(q)^4)$ and so $$e^{-\Omega_0}(\Theta_{\Omega_0})_{\hs} = \frac{1}{2}(\theta_3(q^{1/2})^4+\theta_4(q^{1/2})^4)$$ 
(see equation (89) in \cite{conway}). They also give the theta series of the $D_4$ lattice when shifted by various ``glue vectors." These vectors are as follows: 
$$[1] = \frac{1}{2}(\beta_1'+2\beta_2'+\beta_3'+2\beta_4') = \beta_3+\beta_4,$$ 
$$[2] = \frac{1}{2}(\beta_4'-\beta_3') = \beta_1+2\beta_2+3\beta_3+2\beta_4$$
and 
$$[3] = \frac{1}{2}(\beta_1'+2\beta_2'+2\beta_3'+\beta_4') = -\beta_4$$
(\cite{conway} pg. 117). Notice that 
$$M_{F_4}^*+[1] = M_{F_4}^*+\beta_3+\beta_4,$$ 
$$M_{F_4}^*+[2] = M_{F_4}^*+\beta_3$$
and 
$$M_{F_4}^*+[3] = M_{F_4}^*+\beta_4.$$ 
Equation (88) on page 118 of \cite{conway} states: 
$$\Theta_{D_4+[1]}(q) = \Theta_{D_4+[3]}(q) = \frac{1}{2}\theta_2(q)^4$$ 
and equation (89) states 
$$\Theta_{D_4+[2]}(q) = \frac{1}{2}(\theta_3(q)^4-\theta_4(q)^4).$$ 
But by Proposition \ref{jacobiprops} (1), $\theta_2(q)^4 = \theta_3(q)^4-\theta_4(q)^4$. So 
$$e^{-\Omega_0}(\Theta_{\Omega_0+\beta_3})_{\hs} = e^{-\Omega_0}(\Theta_{\Omega_0+\beta_3})_{\hs} = e^{-\Omega_0}(\Theta_{\Omega_0+\beta_3+\beta_4})_{\hs} = \frac{1}{2}\theta_2(q^{1/2})^4.$$ Finally, note that $M_{F_4}^*+\omega_4 = M_{F_4}^*+\beta_3$ and so these last 3 theta functions also equal $e^{-\Omega_0}(\Theta_{\Omega_4})_{\hs}$.

\subsection{$G_2$ Weight Lattice and Theta Functions}
We have 
$$\widehat{P}_{G_2}^1 = \{\Xi_0+\xi+n\delta_{G_2}\mid\xi\in P_{G_2}, n\in\C\}.$$ 
Now we are looking for coset representatives of $\widehat{P}_{G_2}^1/(M^*_{G_2}+\C\delta_{G_2})$ where 
$$M^*_{G_2} = \{a\zeta_1+3b\zeta_2\mid a,b\in\Z\}.$$ 
As in the previous case, $Q_{G_2} = P_{G_2}$ and so 
$$\widehat{P}_{G_2}^1 = \{\Xi_0+a\zeta_1+b\zeta_2+n\delta_{G_2}\mid a,b\in\Z, n\in\C\}.$$ Since $\Xi_0\not\in M^*_{G_2}$, $n\delta_{G_2}\in\C\delta_{G_2}$ for all $n\in\C$ and $a\zeta_1+b\zeta_2\in M^*_{G_2}$ only when $b\in 3\Z$, we can take the following coset representatives: $\Xi_0$, $\Xi_0+\zeta_2$, $\Xi_0+2\zeta_2$.

Let $\zeta_1'\coloneqq -2\zeta_1-3\zeta_2$ and $\zeta_2'\coloneqq\zeta_1$. Then it will follow that $\zeta_1',\zeta_2'$ form a root system of type $A_2$ and that $\Z\{\zeta_1',\zeta_2'\} = M^*_{G_2}$. So, 
$$e^{-\Xi_0}(\Theta_{\Xi_0})_{\hs} = \Theta_{M_{G_2}^*}(q) = \Theta_{A_2}(q^{1/2}).$$ 
However, Conway and Sloane renormalize the $A_2$ lattice so that long roots have squared length 1. So, in what follows, we will not need to substitute $q^{1/2}$ for $q$ in the theta series they obtain. So, to begin, 
$$\Theta_{A_2}(q) = \theta_3(q)\theta_3(q^3)+\theta_2(q)\theta_2(q^3)$$ 
(equation (60) on page 111 of \cite{conway}) which, by Proposition \ref{CubicThetaProps} (1) equals $a(q)$. We will define two glue vectors for $A_2$ as follows: 
$$[1]\coloneqq \frac{1}{3}(\zeta_1'+2\zeta_2') = -\zeta_2$$ 
and 
$$[2]\coloneqq\frac{1}{3}(2\zeta_1'+\zeta_2') =  -\zeta_1-2\zeta_2.$$ 
Notice that $M_{G_2}^*+[1] = M_{G_2}^*+2\zeta_2$ and $M_{G_2}^*+[2] = M_{G_2}^*+\zeta_2$. By equation (57) on page 110 of \cite{conway}, we have $\Theta_{A_2+[1]}(q) = \Theta_{A_2+[2]}(q)$ and so (by (60) on pg. 111 of \cite{conway}) we have 
$$e^{-\Xi_0}(\Theta_{\Xi_0+\zeta_1})_{\hs} = e^{-\Xi_0}(\Theta_{\Xi_0+\zeta_2})_{\hs} = \theta_2(q)\psi_6(q^3)+\theta_3(q)\psi_3(q^3).$$ 
By proposition \ref{jacobiprops} (4) and (5), we have 
\begin{align}
&\theta_2(q)\psi_6(q^3)+\theta_3(q)\psi_3(q^3)\\
&= \theta_2(q)\frac{1}{2}(\theta_2(q^{1/3})-\theta_2(q^3))+\theta_3(q)\frac{1}{2}(\theta_3(q^{1/3})-\theta_3(q^3))\\ 
&= \frac{1}{2}(\theta_2(q)\theta_2(q^{1/3})-\theta_2(q)\theta_2(q^3)+\theta_3(q)\theta_3(q^{1/3})-\theta_3(q)\theta_3(q^3))\\
&= \frac{1}{2}((\theta_2(q)\theta_2(q^{1/3})+\theta_3(q)\theta_3(q^{1/3}))-(\theta_2(q)\theta_2(q^3)+\theta_3(q)\theta_3(q^3)))\\
&= \frac{1}{2}(a(q^{1/3})-a(q))\\
&= c(q)
\end{align}
where line 5 follows from line 4 by proposition \ref{CubicThetaProps} (1) and line 6 follows from line 5 by proposition \ref{CubicThetaProps} (7). We note that $M_{G_2}^*+\xi_2 = M_{G_2}^*+2\zeta_2$ and so these last two theta functions also equal $e^{-\Xi_0}(\Theta_{\Xi_2})_{\hs}$.

\subsection{String Functions}
We will first consider the $F_4$ Weyl group. Let $s_i\coloneqq s_{\beta_i}$ and $s\coloneqq s_4\circ s_3\circ s_2\circ s_1\circ s_3\circ s_2$. Then $s(\beta_3+\beta_4) = \omega_4$. Further, $s_3(\beta_4) = \beta_3+\beta_4$ so $(s\circ s_3)(\beta_4) = s(\beta_3+\beta_4) = \omega_4$. So $c^{\Omega}_{\Omega_0+\beta_4} = c^{\Omega}_{\Omega_0+\beta_3+\beta_4} = c^{\Omega}_{\Omega_4}$ for any $\Omega\in\widehat{P}_{F_4}$. Also, since $M_{F_4}^*+\beta_3 = M_{F_4}^*+\omega_4$, we also have $c^{\Omega}_{\Omega_0+\beta_3} = c^{\Omega}_{\Omega_4}$. Similarly, working with the $G_2$ Weyl group, we notice that $s_{\zeta_1+\zeta_2}(\xi_2) = \zeta_2$. This, together with $M_{G_2}^*+\zeta_2 = M_{G_2}^*+\xi_2$ gives $c^{\Xi}_{\Xi_0+\zeta_2} = c^{\Xi}_{\Xi_0+2\zeta_2} = c^{\Xi}_{\Xi_2}$ for all $\Xi\in\widehat{P}_{G_2}$.

We now list the string functions (or differences of string functions) that will be necessary in our character calculation. These all are stated on page 222 of \cite{KacPeterson} without proof. Bernard and Thierry-Mieg on page 234 of \cite{BTM} verify the $G_2^{(1)}$ string functions stated by Kac and Peterson. We will verify the $F_4^{(1)}$ string functions later in this work.

\begin{align}
c_{\Xi_2}^{\Xi_2} &= \eta(\tau)^{-3}q^{3/40}\prod_{n\not\equiv\pm 1\mod 5}(1-q^{3n})\\
c_{\Xi_2}^{\Xi_0} &= \eta(\tau)^{-3}q^{27/40}\prod_{n\not\equiv\pm 2\mod 5}(1-q^{3n})\\
c&\coloneqq c_{\Xi_0}^{\Xi_0}-c_{\Xi_2}^{\Xi_0} = \eta(\tau)^{-3}q^{1/120}\prod_{n\not\equiv\pm 1\mod 5}(1-q^{n/3})\\
c^\prime &\coloneqq c_{\Xi_2}^{\Xi_2}-c_{\Xi_0}^{\Xi_2} = \eta(\tau)^{-3}q^{3/40}\prod_{n\not\equiv\pm 2\mod 5}(1-q^{n/3})\\
c_{\Omega_4}^{\Omega_4} &= \eta(\tau)^{-6}\eta(2\tau)q^{1/20}\prod_{n\not\equiv\pm 1\mod 5}(1-q^{2n})\\
c_{\Omega_4}^{\Omega_0} &= \eta(\tau)^{-6}\eta(2\tau)q^{9/20}\prod_{n\not\equiv\pm 2\mod 5}(1-q^{2n})\\
d&\coloneqq c_{\Omega_0}^{\Omega_0}-c_{\Omega_4}^{\Omega_0} = \eta(\tau)^{-6}\eta(\tau/2)q^{1/80}\prod_{n\not\equiv\pm 1\mod 5}(1-q^{n/2})\\
d^\prime &\coloneqq c_{\Omega_4}^{\Omega_4}-c_{\Omega_0}^{\Omega_4} = \eta(\tau)^{-6}\eta(\tau/2)q^{9/80}\prod_{n\not\equiv\pm 2\mod 5}(1-q^{n/2}).
\end{align}
Note that 
$$\prod_{n\not\equiv\pm 1\mod 5}(1-q^{n}) = \frac{(q;q)_\infty}{(q;q^5)_\infty(q^4;q^5)_\infty} = \varphi_1G(q)$$ 
by proposition \ref{GH} (1). Similarly, 
$$\prod_{n\not\equiv\pm 2\mod 5}(1-q^{n}) = \varphi_1H(q)$$ 
by proposition \ref{GH} (2). Using the same notation convention for $G$ and $H$ that we have established for $\eta$ and $\varphi$, we rewrite our string functions as follows:

\begin{align}
c_{\Xi_2}^{\Xi_2} &= q^{-1/20}\varphi_1^{-3}\varphi_3G_3\\
c_{\Xi_2}^{\Xi_0} &= q^{11/20}\varphi_1^{-3}\varphi_3H_3\\
c &= q^{-7/60}\varphi_1^{-3}\varphi_{1/3}G_{1/3}\\
c^\prime &= q^{-1/20}\varphi_1^{-3}\varphi_{1/3}H_{1/3}\\
c_{\Omega_4}^{\Omega_4} &= q^{-7/60}\varphi_1^{-6}\varphi_2^2G_2\\
c_{\Omega_4}^{\Omega_0} &= q^{17/60}\varphi_1^{-6}\varphi_2^2H_2\\
d &= q^{-13/60}\varphi_1^{-6}\varphi_{1/2}^2G_{1/2}\\
d^\prime &= q^{-7/60}\varphi_1^{-6}\varphi_{1/2}^2H_{1/2}.
\end{align}
We end this section by noting that:
\begin{align}
c_{\Xi_0}^{\Xi_0} &= c+c_{\Xi_2}^{\Xi_0}\\
c_{\Xi_0}^{\Xi_2} &= -c^\prime+ c_{\Xi_2}^{\Xi_2}\\
c_{\Omega_0}^{\Omega_0} &= d+c_{\Omega_4}^{\Omega_0}\\
c_{\Omega_0}^{\Omega_4} &= -d^\prime+c_{\Omega_4}^{\Omega_4}.
\end{align}

\section{Setting up Characters}
The characters of the level 1 modules of the so-called ``simply-laced" algebras (i.e. those with only one root length) are quite simple in terms of the Kac-Peterson formula (see example 1 on page 217 of \cite{KacPeterson}). In the case of $V^{\Lambda_0}$, we have $\ch(V^{\Lambda_0})_{\hs} = e^{\Lambda_0}\varphi_1^{-8}\Theta_{Q_{E_8}}(q^{1/2})$. Since $\Theta_{Q_{E_8}} = \frac{1}{2}(\theta_2(q)^8+\theta_3(q)^8+\theta_4(q)^8)$ (see page 122 in \cite{conway}), we have:
$$\gr(V^{\Lambda_0})_{\hs} = \varphi_1^{-8}\frac{1}{2}(\theta_2(q^{1/2})^8+\theta_3(q^{1/2})^8+\theta_4(q^{1/2})^8).$$

A quick computation gives that $s_{\Omega_0} = -13/60$, $s_{\Omega_4} = 23/60$, $s_{\Xi_0} = -7/60$, $s_{\Xi_2} = 17/60$. From the previous three sections, we have the following characters for our level 1 $F_4^{(1)}$ and $G_2^{(1)}$ modules:
\begin{align}
\ch(V^{\Omega_0}) &= q^{13/60}(c^{\Omega_0}_{\Omega_0}\Theta_{\Omega_0}+3c^{\Omega_0}_{\Omega_4}\Theta_{\Omega_4})\\
\ch(V^{\Omega_4}) &= q^{-23/60}(c^{\Omega_4}_{\Omega_0}\Theta_{\Omega_0}+3c^{\Omega_4}_{\Omega_4}\Theta_{\Omega_4})\\
\ch(V^{\Xi_0}) &= q^{7/60}(c^{\Xi_0}_{\Xi_0}\Theta_{\Xi_0}+2c^{\Xi_0}_{\Xi_2}\Theta_{\Xi_2})\\
\ch(V^{\Xi_2}) &= q^{-17/60}(c^{\Xi_2}_{\Xi_0}\Theta_{\Xi_0}+2c^{\Xi_2}_{\Xi_2}\Theta_{\Xi_2}).
\end{align}

\section{The Branching Rule Decomposition}
We now state our main result, which is the branching rule decomposition of the $E_8^{(1)}$ basic module $V^{\Lambda_0}$ with respect to the subalgebra of type $F_4^{(1)}\oplus G_2^{(1)}$. This result is stated by Kac and Sanielevici in \cite{KacSan}, Kac and Wakimoto in \cite{KacWak}, and by Bernard and Thierry-Mieg in \cite{BTM}, but we will now prove it rigorously.

\begin{theorem}
As $F_4^{(1)}\oplus G_2^{(1)}$-modules, we have $V^{\Lambda_0} = (V^{\Omega_0}\otimes V^{\Xi_0})\oplus(V^{\Omega_4}\otimes V^{\Xi_2})$.
\end{theorem}

\begin{proof}
We will begin by setting up our graded dimension for $(V^{\Omega_0}\otimes V^{\Xi_0})\oplus(V^{\Omega_4}\otimes V^{\Xi_2})$. Note that shifting $F_4^{(1)}$-modules by $e^{-\Lambda_0}$ will have the same effect as shifting by $e^{-\Omega_0}$ and likewise on $G_2^{(1)}$-modules will have the effect of shifting by $e^{-\Xi_0}$. So we have:
\begin{align}
(e^{-\Omega_0}&\ch(V^{\Omega_0})_{\hs}e^{-\Xi_0}\ch(V^{\Xi_0})_{\hs})+q(e^{-\Omega_0}\ch(V^{\Omega_4})_{\hs}e^{-\Xi_0}\ch(V^{\Xi_2})_{\hs})\\
=& q^{13/60}(\doo\THFO+3\dof\THFF)\nonumber\\
&\times q^{7/60}(\coo\THGO+2\cot\THGT)\nonumber\\
&+q[q^{-23/60}(\dfo\THFO+3\dff\THFF)\nonumber\\
&\times q^{-17/60}(\cto\THGO+2\ctt\THGT)]\\
=& q^{1/3}[(\doo\THFO+3\dof\THFF)\nonumber\\
&\times(\coo\THGO+2\cot\THGT)\nonumber\\
&+(\dfo\THFO+3\dff\THFF)\nonumber\\
&\times(\cto\THGO+2\ctt\THGT)]\\
=& q^{1/3}[\coo\doo\THFO\THGO\nonumber\\
&+3\coo\dof\THFF\THGO\nonumber\\
&+2\cot\doo\THFO\THGT\nonumber\\
&+6\cot\dof\THFF\THGT\nonumber\\
&+\cto\dfo\THFO\THGO\nonumber\\
&+3\cto\dff\THFF\THGO\nonumber\\
&+2\ctt\dfo\THFO\THGT\nonumber\\
&+6\ctt\dff\THFF\THGT].
\end{align}

Organizing our terms by products of theta functions, line (34) becomes:
\begin{align*}
q^{1/3}&[\Theta(0,0)(\coo\doo+\cto\dfo)+2\Theta(0,2)(\cot\doo+\ctt\dfo)\\
&+3\Theta(4,0)(\coo\dof+\cto\dff)+6\Theta(4,2)(\cot\dof+\ctt\dff)].
\end{align*}

Where
$$\Theta(0,0)\coloneqq e^{-\Omega_0}(\Theta_{\Omega_0})\hs e^{-\Xi_0}(\Theta_{\Xi_0})\hs\text{, }\Theta(4,0)\coloneqq e^{-\Omega_0}(\Theta_{\Omega_4})\hs e^{-\Xi_0}(\Theta_{\Xi_0})\hs,
$$ $$\Theta(0,2)\coloneqq e^{-\Omega_0}(\Theta_{\Omega_0})\hs (e^{-\Xi_0}\Theta_{\Xi_2})\hs\text{ and }\Theta(4,2)\coloneqq e^{-\Omega_0}(\Theta_{\Omega_4})\hs e^{-\Xi_0}(\Theta_{\Xi_2})\hs.$$
We will expand and simplify each of the four parenthetical string function expressions separately:

\begin{align}
q^{1/3}&(\coo\doo+\cto\dfo)\\
=& q^{1/3}((c+\cot)(d+\dof)+(-c^\prime+\ctt)(-d^\prime+\dff))\\
=& q^{1/3}(cd+c_{\Xi_2}^{\Xi_0}d+c_{\Xi_2}^{\Xi_0}c_{\Omega_4}^{\Omega_0}+cc_{\Omega_4}^{\Omega_0}+c^\prime d^\prime-c^\prime c_{\Omega_4}^{\Omega_4}-c_{\Xi_2}^{\Xi_2}d^\prime+c_{\Xi_2}^{\Xi_2}c_{\Omega_4}^{\Omega_4})\\
=& q^{1/3}(cd+c_{\Xi_2}^{\Xi_0}d+cc_{\Omega_4}^{\Omega_0}+c_{\Xi_2}^{\Xi_0}c_{\Omega_4}^{\Omega_0}+c_{\Xi_2}^{\Xi_2}c_{\Omega_4}^{\Omega_4}-c^\prime c_{\Omega_4}^{\Omega_4}-c_{\Xi_2}^{\Xi_2}d^\prime+c^\prime d^\prime)\\
=& q^{1/3}\varphi_1^{-9}(q^{-\frac{1}{3}}\varphi_\frac{1}{3}\varphi^2_\frac{1}{2}G_\frac{1}{3}G_\frac{1}{2}+q^\frac{1}{3}\varphi_3\varphi^2_\frac{1}{2}H_3G_\frac{1}{2}\nonumber\\
&+q^\frac{1}{6}\varphi_\frac{1}{3}\varphi^2_2G_\frac{1}{3}H_2+q^\frac{5}{6}\varphi_3\varphi^2_2H_3H_2\nonumber\\
&+q^{-\frac{1}{6}}\varphi_3\varphi^2_2G_3G_2-q^{-\frac{1}{6}}\varphi_\frac{1}{3}\varphi^2_2H_\frac{1}{3}G_2\nonumber\\
&-q^{-\frac{1}{6}}\varphi_3\varphi^2_\frac{1}{2}G_3H_\frac{1}{2}+q^{-\frac{1}{6}}\varphi_\frac{1}{3}\varphi^2_\frac{1}{2}H_\frac{1}{3}H_\frac{1}{2})\\
=& q^{1/3}\varphi_1^{-9}(q^{-\frac{1}{3}}\varphi_\frac{1}{3}\varphi^2_\frac{1}{2}G_\frac{1}{3}G_\frac{1}{2}+q^{-\frac{1}{6}}\varphi_\frac{1}{3}\varphi^2_\frac{1}{2}H_\frac{1}{3}H_\frac{1}{2}\nonumber\\
&+q^\frac{1}{3}\varphi_3\varphi^2_\frac{1}{2}H_3G_\frac{1}{2}-q^{-\frac{1}{6}}\varphi_3\varphi^2_\frac{1}{2}G_3H_\frac{1}{2}\nonumber\\
&+q^\frac{1}{6}\varphi_\frac{1}{3}\varphi^2_2G_\frac{1}{3}H_2-q^{-\frac{1}{6}}\varphi_\frac{1}{3}\varphi^2_2H_\frac{1}{3}G_2\nonumber\\
&+q^\frac{5}{6}\varphi_3\varphi^2_2H_3H_2+q^{-\frac{1}{6}}\varphi_3\varphi^2_2G_3G_2)\\
=& q^{1/3}\varphi_1^{-9}(\varphi_\frac{1}{3}\varphi^2_\frac{1}{2}q^{-\frac{2}{6}}(G_\frac{1}{3}G_\frac{1}{2}+q^\frac{1}{6}H_\frac{1}{3}H_\frac{1}{2})\nonumber\\
&-q^{-\frac{1}{6}}\varphi_3\varphi^2_\frac{1}{2}(G_3H_\frac{1}{2}-q^\frac{1}{2}G_\frac{1}{2}H_3)\nonumber\\
&-q^{-\frac{1}{6}}\varphi_\frac{1}{3}\varphi^2_2(G_2H_\frac{1}{3}-q^\frac{1}{3}G_\frac{1}{3}H_2)\nonumber\\
&+q^{-\frac{1}{6}}\varphi_3\varphi^2_2(G_2G_3+qH_2H_3)).
\end{align}
By proposition \ref{RRR}, line (41) becomes:

$$q^{1/3}\varphi_1^{-9}\left(\frac{q^{-\frac{1}{3}}\varphi_\frac{1}{3}\varphi^2_\frac{1}{2}\chi(-q^\frac{1}{2})}{\chi(-q^\frac{1}{6})}-\frac{q^{-\frac{1}{6}}\varphi_3\varphi^2_\frac{1}{2}\chi(-q^\frac{1}{2})}{\chi(-q^\frac{3}{2})}-\frac{q^{-\frac{1}{6}}\varphi_\frac{1}{3}\varphi^2_2\chi(-q^\frac{1}{3})}{\chi(-q)}+\frac{q^{-\frac{1}{6}}\varphi_3\varphi^2_2\chi(-q^3)}{\chi(-q)}\right).$$
Noting that $\chi(-q) = (q;q^2)_\infty = \varphi_1/\varphi_2$, this becomes:
\begin{align}
\varphi^{-8}_1&\left(\frac{\varphi^2_\frac{1}{3}\varphi^3_\frac{1}{2}}{\varphi^2_1\varphi_\frac{1}{6}}-\frac{q^{\frac{1}{6}}\varphi^2_3\varphi^3_\frac{1}{2}}{\varphi^2_1\varphi_\frac{3}{2}}-\frac{q^{\frac{1}{6}}\varphi^2_\frac{1}{3}\varphi^3_2}{\varphi^2_1\varphi_\frac{2}{3}}+\frac{q^{\frac{1}{6}}\varphi^2_3\varphi^3_2}{\varphi^2_1\varphi_6}\right)\\
=& \varphi_1^{-8}(w(q^{1/6})-q^{1/6}x(q^{1/2})-q^{1/6}y(q^{1/3})+q^{1/6}z(q))\\
=& \varphi_1^{-8}(\phi(-q^{1/2})\phi(-q^{3/2})+q^{1/6}x(q^{1/2})-q^{1/6}x(q^{1/2})\nonumber\\
&-q^{1/6}(z(q)-2q^{1/3}\psi(q)\psi(q^3))+q^{1/6}z(q))\\
=& \varphi_1^{-8}(\phi(-q^{1/2})\phi(-q^{3/2})+2q^{1/2}\psi(q)\psi(q^3))\\
=& \varphi_1^{-8}a(q^2)
\end{align}
Where line (44) follows from (43) by theorem \ref{wDissection} and line (46) follows from (45) by corollary \ref{oops}. 
\newline
Next,
\begin{align}
q^{1/3}&(\cot(d+\dof)+\ctt(-d^\prime+\dff))\\
=& q^{1/3}(c_{\Xi_2}^{\Xi_0}d+c_{\Xi_2}^{\Xi_0}c_{\Omega_4}^{\Omega_0}-c_{\Xi_2}^{\Xi_2}d^\prime+c_{\Xi_2}^{\Xi_2}c_{\Omega_4}^{\Omega_4})\\
=& q^{1/3}(c_{\Xi_2}^{\Xi_0}d+c_{\Xi_2}^{\Xi_0}c_{\Omega_4}^{\Omega_0}+c_{\Xi_2}^{\Xi_2}c_{\Omega_4}^{\Omega_4}-c_{\Xi_2}^{\Xi_2}d^\prime)\\
=& q^{1/3}\varphi_1^{-9}(q^\frac{1}{3}\varphi_3\varphi^2_\frac{1}{2}H_3G_\frac{1}{2}+q^\frac{5}{6}\varphi_3\varphi^2_2H_3H_2\nonumber\\
&+q^{-\frac{1}{6}}\varphi_3\varphi^2_2G_3G_2-q^{-\frac{1}{6}}\varphi_3\varphi^2_\frac{1}{2}G_3H_\frac{1}{2})\\
=& q^{1/3}\varphi_1^{-9}(-q^{-\frac{1}{6}}\varphi_3\varphi^2_\frac{1}{2}(G_3H_\frac{1}{2}-q^\frac{1}{2}G_\frac{1}{2}H_3)\nonumber\\
&+q^{-\frac{1}{6}}\varphi_3\varphi^2_2(G_2G_3+qH_2H_3))\\
=& q^{1/3}\varphi_1^{-9}\left(\frac{-q^{-\frac{1}{6}}\varphi_3\varphi^2_\frac{1}{2}\chi(-q^\frac{1}{2})}{\chi(-q^\frac{3}{2})}+\frac{q^{-\frac{1}{6}}\varphi_3\varphi^2_2\chi(-q^3)}{\chi(-q)}\right)\\
=& \varphi_1^{-8}\left(\frac{-q^{\frac{1}{6}}\varphi^2_3\varphi^3_\frac{1}{2}}{\varphi_1^2\varphi_\frac{3}{2}}+\frac{q^{\frac{1}{6}}\varphi^2_3\varphi^3_2}{\varphi^2_1\varphi_6}\right)\\
=& \varphi_1^{-8}(-q^\frac{1}{6}x(q^\frac{1}{2})+q^\frac{1}{6}z(q))\\
=& \varphi_1^{-8}(-q^\frac{1}{6}(z(q)-q^{-\frac{1}{6}}c(q^2))+q^\frac{1}{6}z(q))\\
=& \varphi_1^{-8}c(q^2)
\end{align}
where line (52) follows from (51) by proposition \ref{RRR} and (55) follows from (54) by theorem \ref{xDissection}. 
\newline
We continue with
\begin{align}
q^{1/3}&(\coo\dof+\cto\dff)\\
=& q^{1/3}((c+\cot)\dof+(-c^\prime+\ctt)\dff)\\
=& q^{1/3}(cc_{\Omega_4}^{\Omega_0}+c_{\Xi_2}^{\Xi_0}c_{\Omega_4}^{\Omega_0}-c'c_{\Omega_4}^{\Omega_4}+c_{\Xi_2}^{\Xi_2}c_{\Omega_4}^{\Omega_4})\\
=& q^{1/3}(cc_{\Omega_4}^{\Omega_0}+c_{\Xi_2}^{\Xi_0}c_{\Omega_4}^{\Omega_0}+c_{\Xi_2}^{\Xi_2}c_{\Omega_4}^{\Omega_4}-c'c_{\Omega_4}^{\Omega_4})\\
=& \varphi_1^{-9}(q^\frac{1}{6}\varphi_\frac{1}{3}\varphi^2_2G_\frac{1}{3}H_2q^\frac{5}{6}\varphi_3\varphi^2_2H_3H_2\nonumber\\
&+q^{-\frac{1}{6}}\varphi_3\varphi^2_2G_3G_2-q^{-\frac{1}{6}}\varphi_\frac{1}{3}\varphi^2_2H_\frac{1}{3}G_2)\\
=& q^{1/3}\varphi_1^{-9}(-q^{-\frac{1}{6}}\varphi_\frac{1}{3}\varphi^2_2(G_2H_\frac{1}{3}-q^\frac{1}{3}G_\frac{1}{3}H_2)\nonumber\\
&+q^{-\frac{1}{6}}\varphi_3\varphi^2_2(G_2G_3+qH_2H_3))\\
=& \varphi_1^{-9}\left(\frac{-q^{-\frac{1}{6}}\varphi_\frac{1}{3}\varphi^2_2\chi(-q^\frac{1}{3})}{\chi(-q)}+\frac{q^{-\frac{1}{6}}\varphi_3\varphi^2_2\chi(-q^3)}{\chi(-q)}\right)\\
=& \varphi_1^{-8}\left(\frac{-q^{\frac{1}{6}}\varphi^2_\frac{1}{3}\varphi^3_2}{\varphi_1^2\varphi_\frac{2}{3}}+\frac{q^{\frac{1}{6}}\varphi^2_3\varphi^3_2}{\varphi_1^2\varphi_6}\right)\\
=& \varphi_1^{-8}(-q^{\frac{1}{6}}y(q^{\frac{1}{3}})+q^{\frac{1}{6}}z(q))\\
=& \varphi_1^{-8}(-q^\frac{1}{6}(z(q)-2q^\frac{1}{3}\psi(q)\psi(q^3))+q^\frac{1}{6}z(q))\\
=& \varphi_1^{-8}(2q^\frac{1}{2}\psi(q)\psi(q^3))
\end{align}
where line (63) follows from (62) by proposition \ref{RRR} and (66) follows from (65) by theorem \ref{yDissection}. 
\newline
Finally,
\begin{align}
q^{1/3}&(c_{\Xi_2}^{\Xi_0}c_{\Omega_4}^{\Omega_0}+c_{\Xi_2}^{\Xi_2}c_{\Omega_4}^{\Omega_4})\\
=& q^{1/3}\varphi_1^{-9}(q^\frac{5}{6}\varphi_3\varphi^2_2H_3H_2+q^{-\frac{1}{6}}\varphi_3\varphi^2_2G_3G_2)\\
=& q^{1/3}\varphi_1^{-9}q^{-\frac{1}{6}}\varphi_3\varphi^2_2(G_2G_3+qH_2H_3)\\
=& q^{1/3}\varphi_1^{-9}\left(\frac{q^{-\frac{1}{6}}\varphi_3\varphi^2_2\chi(-q^3)}{\chi(-q)}\right)\\
=& \varphi_1^{-8}\left(\frac{q^{\frac{1}{6}}\varphi^2_3\varphi^3_2}{\varphi_1^2\varphi_6}\right)\\
=& \varphi_1^{-8}q^{\frac{1}{6}}z(q)\\
=& \varphi_1^{-8}\frac{1}{3}(c(q^{\frac{1}{2}})-c(q^2))
\end{align}
where line (71) follows from (70) by \ref{RRR} (1) and (74) follows from (73) by corollary \ref{zDissection}.

Our calculation now becomes:
\begin{align}
\varphi_1^{-8}&\bigg[\Theta(0,0)(a(q^2))+2\Theta(0,2)(c(q^2))\nonumber\\
&+3\Theta(4,0)(2q^{1/2}\psi(q)\psi(q^3))+6\Theta(4,2)\left(\frac{1}{3}(c(q^{1/2})-c(q^2))\right)\bigg].
\end{align}
Since our goal is to show that this equals the horizontal specialization of the graded dimension of $V^{\Lambda_0}$ we notice that this goal will be accomplished once we show that $$\Theta(0,0)(a(q^2))+2\Theta(0,2)(c(q^2))+3\Theta(4,0)(2q^{1/2}\psi(q)\psi(q^3))+2\Theta(4,2)(c(q^{1/2})-c(q^2))$$ equals $(e^{-\Lambda_0}\Theta_{\Lambda_0})_{\hs} = \dfrac{1}{2}(\theta_2(q^{1/2})^8+\theta_3(q^{1/2})^8+\theta_4(q^{1/2})^8).$ Plugging in the theta functions discussed in the previous chapter gives:
\begin{align}
\Theta&(0,0)(a(q^2))+2\Theta(0,2)(c(q^2))+3\Theta(4,0)(2q^{1/2}\psi(q)\psi(q^3))+2\Theta(4,2)(c(q^{1/2})-c(q^2))\\
=& \left(\frac{1}{2}(\theta_3(q^{1/2})^4+\theta_4(q^{1/2})^4)a(q)\right)(a(q^2))+2\left(\frac{1}{2}(\theta_3(q^{1/2})^4+\theta_4(q^{1/2})^4)c(q)\right)(c(q^2))\nonumber\\
&+3\left(\frac{1}{2}\theta_2(q^{1/2})^4a(q)\right)(2q^{1/2}\psi(q)\psi(q^3))+2\left(\frac{1}{2}\theta_2(q^{1/2})^4c(q)\right)(c(q^{1/2})-c(q^2))\\
=& \frac{1}{2}(\theta_3(q^{1/2})^4+\theta_4(q^{1/2})^4)(a(q)a(q^2)+2c(q)c(q^2))\nonumber\\
&+\frac{1}{2}\theta_2(q^{1/2})^4(6q^{1/2}\psi(q)\psi(q^3)a(q)+2c(q)(c(q^{1/2})-c(q^2))).
\end{align}
Noting that $6q^{1/2}\psi(q)\psi(q^3) = a(q^{1/2})-a(q^2)$ by proposition \ref{CubicThetaProps} (12), line (78) becomes
\begin{align}
\frac{1}{2}&(\theta_3(q^{1/2})^4+\theta_4(q^{1/2})^4)(a(q)a(q^2)+2c(q)c(q^2))\nonumber\\
&+\frac{1}{2}\theta_2(q^{1/2})^4((a(q^{1/2})-a(q^2))a(q)+2c(q)(c(q^{1/2})-c(q^2)))\\
=& \frac{1}{2}(\theta_3(q^{1/2})^4+\theta_4(q^{1/2})^4)(a(q)a(q^2)+2c(q)c(q^2))\nonumber\\
&+\frac{1}{2}\theta_2(q^{1/2})^4(a(q^{1/2})a(q)-a(q)a(q^2)+2c(q^{1/2})c(q)-2c(q)c(q^2))\\
=& \frac{1}{2}(\theta_3(q^{1/2})^4+\theta_4(q^{1/2})^4)(a(q)a(q^2)+2c(q)c(q^2))\nonumber\\
&+\frac{1}{2}\theta_2(q^{1/2})^4((a(q^{1/2})a(q)+2c(q^{1/2})c(q))-(a(q)a(q^2)+2c(q)c(q^2))).
\end{align}
By proposition \ref{a12c12} $$a(q)a(q^2)+2c(q)c(q^2) = \frac{1}{2}(\theta_3(q^{1/2})^4+\theta_4(q^{1/2})^4)$$ and, consequently, $$a(q^{1/2})a(q)+2c(q^{1/2})c(q) = \frac{1}{2}(\theta_3(q^{1/4})^4+\theta_4(q^{1/4})^4).$$ But by proposition \ref{jacobiprops} (2) and (3), $$\theta_3(q^{1/4})^4 = (\theta_3(q^{1/4})^2)^2 = \theta_3(q^{1/2})^4+2\theta_3(q^{1/2})^2\theta_2(q^{1/2})^2+\theta_2(q^{1/2})^4$$ and $$\theta_4(q^{1/4})^4 = (\theta_4(q^{1/4})^2)^2 = \theta_3(q^{1/2})^4-2\theta_3(q^{1/2})^2\theta_2(q^{1/2})^2+\theta_2(q^{1/2})^4$$ thus $$\frac{1}{2}(\theta_3(q^{1/4})^4+\theta_4(q^{1/4})^4) = \theta_3(q^{1/2})^4+\theta_2(q^{1/2})^4.$$ So then line (81) becomes:
\begin{align}
\frac{1}{2}&(\theta_3(q^{1/2})^4+\theta_4(q^{1/2})^4)\frac{1}{2}(\theta_3(q^{1/2})^4+\theta_4(q^{1/2})^4)\nonumber\\
&+\frac{1}{2}\theta_2(q^{1/2})^4(\theta_3(q^{1/2})^4+\theta_2(q^{1/2})^4-\frac{1}{2}(\theta_3(q^{1/2})^4+\theta_4(q^{1/2})^4))\\
=& \frac{1}{4}(\theta_3(q^{1/2})^4+\theta_4(q^{1/2})^4)^2+\frac{1}{2}\theta_2(q^{1/2})^4\left(\frac{1}{2}\theta_3(q^{1/2})^4+\theta_2(q^{1/2})^4-\frac{1}{2}\theta_4(q^{1/2})^4\right)\\
=& \frac{1}{4}(\theta_3(q^{1/2})^4+\theta_4(q^{1/2})^4)^2+\frac{1}{2}\theta_2(q^{1/2})^4(\frac{1}{2}\theta_2(q^{1/2})^4+\theta_2(q^{1/2})^4)\\
=& \frac{1}{4}(\theta_3(q^{1/2})^4+\theta_4(q^{1/2})^4)^2+\frac{3}{4}\theta_2(q^{1/2})^8\\
=& \frac{1}{4}(\theta_3(q^{1/2})^8+2\theta_3(q^{1/2})^4\theta_4(q^{1/2})^4+\theta_4(q^{1/2})^8+3\theta_2(q^{1/2})^8)\\
=& \frac{1}{4}(\theta_3(q^{1/2})^8+2\theta_3(q^{1/2})^4\theta_4(q^{1/2})^4+\theta_4(q^{1/2})^8+3(\theta_3(q^{1/2})^4-\theta_4(q^{1/2})^4)^2)\\
=& \frac{1}{4}(\theta_3(q^{1/2})^8+2\theta_3(q^{1/2})^4\theta_4(q^{1/2})^4+\theta_4(q^{1/2})^8+3\theta_3(q^{1/2})^8\nonumber\\
&-6\theta_3(q^{1/2})^4\theta_4(q^{1/2})^4+3\theta_4(q^{1/2})^4)\\
=& \frac{1}{4}[(2\theta_3(q^{1/2})^8+2\theta_4(q^{1/2})^8)+(2\theta_3(q^{1/2})^8-4\theta_3(q^{1/2})^4\theta_4(q^{1/2})^4+2\theta_4(q^{1/2})^8)]\\
=& \frac{1}{4}[(2\theta_3(q^{1/2})^8+2\theta_4(q^{1/2})^8)+2(\theta_3(q^{1/2})^4-\theta_4(q^{1/2})^4)^2]\\
=& \frac{1}{4}[(2\theta_3(q^{1/2})^8+2\theta_4(q^{1/2})^8)+2\theta_2(q^{1/2})^8]\\
=& \frac{1}{2}(\theta_2(q^{1/2})^8)+\theta_3(q^{1/2})^8)+\theta_4(q^{1/2})^8))
\end{align}
where proposition \ref{jacobiprops} (1) is used to obtain line (84) from (83), (87) from (86) and (91) from (90). We have obtained the desired result and thus have proven our main result.
\end{proof}

\chapter{Verification of the level-one $F_4^{(1)}$ string functions}
In section 4.1 we mentioned that the $F_4^{(1)}$ and $G_2^{(1)}$ string functions that were used in our calculation are stated by Kac and Peterson without proof. And although Bernard and Thierry-Mieg confirm the level-one $G_2^{(1)}$ string functions in \cite{BTM}, their methods fail to confirm those for $F_4^{(1)}$. In order to confirm the string functions, we will make use of the work of Kac and Wakimoto (\cite{KacWak}) which allows us to calculate them via Virasoro characters.

\section{Virasoro Theory}
We will give an exposition of the definitions, theorems and constructions that are important to the representation theory of the Virasoro algebra. For more information, see \cite{raina}.
\begin{definition}
The \textbf{Witt algebra}, $\mathcal{D}$, is the infinite dimensional Lie algebra with basis $\{d_m\mid m\in\Z\}$ and brackets given by $[d_m,d_n] = (n-m)d_{m+n}$ for all $m,n\in\Z$. There is a representation of $\mathcal{D}$ on $\C[t,t^{-1}]$ where $d_m$ acts as $t^{m+1}\dfrac{d}{dt}$. This can be extended to an action on the affine algebra $\widehat{\frak g}$ by $d_m\cdot x(n) = nx(m+n)$. Note that under this action, $d_0 = d$.
\end{definition}

\begin{definition}
The \textbf{Virasoro algebra}, $\vir$, is the central extension of $\mathcal{D}$ with basis $\{L_m,c_{\vir}\mid m\in\Z\}$ and brackets given by:
\begin{enumerate}[(V1)]
    \item $[\vir,c_{\vir}] = 0$,
    \item $[L_m,L_n] = (m-n)L_{m+n}+\dfrac{1}{12}(m^3-m)\delta_{m,-n}c_{\vir}$ for all $m,n\in\Z$.
\end{enumerate}
\end{definition}

\begin{theorem}
For each pair $(c,h)\in\C^2$ there is an irreducible representation of $\vir$, denoted $\vir(c,h)$, such that $c_{\vir}$ acts as scalar multiplication by $c$ (called the \textbf{central charge}) and there is a highest weight vector, $v$, such that $L_m\cdot v = 0$ for $m>0$ and $L_0\cdot v = hv$. Further, $\vir(c,h)$ decomposes into a direct sum of finite dimensional eigenspaces for $L_0$ and the graded dimension is the formal power series $$\gr(c,h) = \sum_{n\geq 0}\dim(\vir(c,h)_{h-n})q^n.$$
\end{theorem}

\begin{theorem}[Sugawara Construction]
Let $\widehat{\frak g}$ be an affine Lie algebra and $V^\Lambda$ an irreducible, highest weight $\widehat{\frak g}$-module. Let $\{u_i\mid 1\leq i\leq\dim(\lag)\}$ be a basis of the underlying finite dimensional, simple Lie algebra $\frak g$ and let $\{u^i\mid 1\leq i\leq\dim(\lag)\}$ be its dual basis with respect to the normalized invariant, symmetric bilinear form given in definition \ref{AffineDef}. The following operators, called \textbf{Sugawara operators}, represent $\vir$ on $V^\Lambda$:
$$L_m = \frac{1}{2(h^\vee+\Lambda(c))}\sum_{k\in\Z}\textbf{:}u_i(-k)u^i(m+k)\textbf{:}$$ for all $m\in\Z$ where $\textbf{: :}$, the \textbf{bosonic normal ordering}, is defined by
$$\textbf{:}u_i(-k)u^i(m+k)\textbf{:} = 
\begin{cases}
u_i(-k)u^i(m+k)\text{ if }-k\leq m+k \\
u^i(m+k)u_i(-k)\text{ if }m+k<-k
\end{cases}
$$
and $u(n) = u\otimes t^n$ as in definition \ref{AffineDef}.
In this representation, $$c = \frac{\dim(\lag)\cdot\Lambda(c)}{h^\vee+\Lambda(c)}$$ and $L_0$ acts as $-d$ on $V^\Lambda$.
\end{theorem}

For the next theorem, let $\Tilde{\laa}$ denote $\laa\otimes_\C\C[t,t^{-1}]\oplus\C c$.

\begin{theorem}
Let $\lap$ be a simple Lie subalgebra of a finite dimensional simple Lie algebra $\frak g$ and let $V^\Lambda$ be an irreducible $\Tilde{\lag}$-module. Further, let $\vir_{\lap}(V^\Lambda)$ and $\vir_{\lag}(V^\Lambda)$ be two representations of $\vir$ on $V^\Lambda$ given by the Sugawara construction for $\Tilde{\lap}$ and $\Tilde{\lag}$ respectively. Denoting the respective operators with superscripts, i.e. $\{L_m^\lap,c^\lap\mid m\in\Z\}$ and $\{L_m^\lag,c^\lag\mid m\in\Z\}$, the differences $L_m^\lag-L_m^\lap$ give a representation of $\vir$, $\vir_{\lag-\lap}(V^\Lambda)$ with central charge $c^{\lag/\lap} = c^\lag-c^\lap$. This representation also commutes with $\Tilde{\lap}$ and $\vir_\lap$, that is $[L_m^\lag-L_m^\lap,\Tilde{\lap}] = 0 = [L_m^\lag-L_m^\lap,L_n^\lap]$ for all $m,n\in\Z$.
\end{theorem}

The ``coset Virasoro" construction in this last theorem is due to Goddard, Kent and Olive (see \cite{GKO}) 
and is useful in many affine branching rule problems, especially those where the central charge, 
$0 < c^\lag-c^\lap < 1$. Having an algebra of operators that commutes with $\Tilde{\lap}$ means that 
$V^\Lambda$ decomposes into a direct sum of tensor products of a module for that coset Virasoro and a 
module for $\Tilde{\lap}$. Because the $E_8^{(1)}$ subalgebra of type $F_4^{(1)}\oplus G_2^{(1)}$ is a 
conformal subalgebra where $0 = c^{\lag/\lap}$, the coset Virasoro modules would be trivial, so 
this construction is not useful for our decomposition. 
When computing graded dimension, it is often more convenient to shift by $q^{h-c/24}$ and so we define 
$\chi(c,h)\coloneqq q^{h-c/24}\gr(c,h)$. 

For $0<c_{\vir}<1$, there is a discrete series of ``minimal models" for which the Virasoro characters have special behavior. For integers $m, n, s$ and $t$ with $s,t\geq 2$, $s,t$ relatively prime, $1\leq m<s$ and $1\leq n<t$, let
\begin{equation}
\label{FFc}    c_{s,t}\coloneqq 1-\frac{6(s-t)^2}{st}
\end{equation}
and
\begin{equation}
\label{FFh}    h_{s,t}^{m,n} = \frac{(mt-ns)^2-(s-t)^2}{4st}.
\end{equation}
Notice that for each $c = c_{s,t}$, there are only finitely many values of $h = h_{s,t}^{m,n}$. In \cite{FeiginFuchs}, Feigin and Fuchs showed that 
$$\chi_{s,t}^{m,n}(q)\coloneqq\chi(c_{s,t},h_{s,t}^{m,n}) = \frac{q^{h_{s,t}^{m,n}-c_{s,t}/24}}{\varphi(q)}\sum_{k\in\Z}q^{stk^2}(q^{k(mt-ns)}-q^{(mt+ns)k+mn}).$$

In his dissertation (\cite{Quincy}), Loney calculated the characters for $(s,t) = (3,4)$ and $(4,5)$ using the above formula. We now state those which will be needed for our purposes and note to the right of each one its equation number in \cite{Quincy}.

\begin{theorem}
\label{QChars}
Let $V(q)\coloneqq\varphi(q)/\varphi(q^2)$ (so that $V(-q) = \varphi(-q)/\varphi(q^2)$). Then we have the following:
\begin{enumerate}[(1)]
    \item $\chi_{3,4}^{1,1}(q) = \dfrac{1}{2}q^{-1/48}(V(-q)+V(q))$ [3.39]
    \item $\chi_{3,4}^{1,3}(q) = \dfrac{1}{2}q^{-1/48}(V(-q)-V(q))$ [3.40]
    \item $\chi_{3,4}^{1,2} = q^{1/24}V(q)^{-1}$ [3.41]
    \item $\chi_{4,5}^{1,1}(q) = \dfrac{1}{2}q^{-7/240}(V(-q)G(-q)+V(q)G(q))$ [3.50]
    \item $\chi_{4,5}^{1,4}(q) = \dfrac{1}{2}q^{-7/240}(V(-q)G(-q)-V(q)G(q)).$ [3.51]
        \item $\chi_{4,5}^{1,2}(q) = q^{17/240}\dfrac{1}{2}(V(-q)G(-q)+V(q)G(q))$ [3.54]
    \item $\chi_{4,5}^{1,3}(q) = q^{17/240}\dfrac{1}{2}(V(-q)H(-q))-V(q)H(q)$ [3.55]
    \item $\chi_{4,5}^{2,1}(q) = q^{49/120}H(q^2)V(q)^{-1}$ [3.56]
    \item $\chi_{4,5}^{2,2}(q) = q^{1/120}G(q^2)V(q)^{-1}$ [3.57].
\end{enumerate}
\end{theorem}

\section{Branching Functions}
In \cite{KacWak}, Kac and Wakimoto present in-depth information on the so-called ``branching functions." We begin with a preliminary definition.
\begin{definition}
For a finite dimensional simple Lie algebra, $\lag$, let $\lan_+\coloneqq\bigoplus_{\alpha\in\Phi^+}\lag_\alpha$ and $\lan_-\coloneqq\bigoplus_{\alpha\in\Phi^-}\lag_\alpha$. Then $\lag = \lan_-\oplus\lah\oplus\lan_+$ is called the \textbf{triangular decomposition} of $\lag$. Analogously, define $\widehat{\lan}_\pm\coloneqq\lan_\pm\oplus\sum_{k>0}(\lag\otimes t^{\pm k})$.
\end{definition}
In what follows, let $\lag$ be a finite dimensional rank $\ell$ simple Lie algebra with subalgebra $\lap$ and let $\widehat{\lag}$ and $\widehat{\lap}$ be the corresponding affine algebras. Any notation ornamented with a $\mathring{}$ will correspond to $\lap$ and $\widehat{\lap}$ and any notation without will correspond to $\lag$ and $\widehat{\lag}$. Let $\Lambda\in\widehat{P}^+$ and $\lambda\in\mathring{\widehat{P}^+}+\C\mathring{\delta}$. Viewing $V^\Lambda$ as a $\widehat{\lap}$-module, we can express it as a direct sum of $\widehat{\lap}$-modules $V^\lambda$. We wish to count the occurrence of each $V^\lambda$ in this decomposition. Let
$$(V^\Lambda)_\lambda^{\mathring{\widehat{\lan}}_+}\coloneqq\{v\in V^\Lambda\mid hv = \lambda(h)v\text{ for }h\in\mathring{H},\mathring{\widehat{\lan}}_+(v) = 0\}$$ and
$$\mult_\Lambda(\lambda,\widehat{\lap}) = \dim((V^\Lambda)_\lambda^{\mathring{\widehat{\lan}}_+}).$$
We now define the branching functions $$b^\Lambda_{\lambda}(\tau;\hat{\lag},\hat{\lap}) = q^{s_\Lambda-\mathring{s}_\lambda}\sum_{n\in\Z}\mult_\Lambda(\lambda-n\delta,\hat{\lap})q^n.$$
Note that we will often just write $b^\Lambda_\lambda(\tau)$ or $b^\Lambda_\lambda$. From here, it follows that the string functions are special cases of the branching functions for $\lap = \lah$. Since $\lah$ is abelian, we get for $\Lambda,\lambda\in\widehat{P}^+$ that 
\begin{equation}
\label{StrBranch}    b^\Lambda_\lambda = \eta(\tau)^\ell c^\Lambda_\lambda
\end{equation}
(see remark 12.8 and equation 12.8.13 on pages 232-233 of \cite{kac}).

Next, we will consider the special case where $\lak$ is a finite dimensional simple Lie algebra, 
$\lag = \lak\oplus\lak$ and $\lap = \lak$ considered as diagonally embedded in $\lag$. 
Given two finite dimensional irreducible $\lak$-modules, $V(\lambda)$ and $V(\mu)$, with highest weights
$\lambda$ and $\mu$, the Weyl complete reducibility theorem says there is a direct sum decomposition
of the tensor product
$$V(\lambda)\otimes V(\mu) = \sum_{\nu}\mult_{\lambda,\mu}^\nu V(\nu)$$
where the sum is over highest weights $\nu$ which must be less than or equal to $\lambda+\mu$ in the
partial ordering of integral weights of $\lak$. The coefficients are called the outer multiplicities in 
the tensor product, and they just count the number of times a summand, $V(\nu)$, occurs in the 
decomposition. That multiplicity is the dimension of the subspace of vectors of weight $\nu$ annihilated
by the positive root vectors in $\lak$, that is, the positive nilpotent subalgebra of $\lak$. 

For the category of highest weight modules of an affine Kac-Moody Lie algebra, there is a similar 
complete reducibility theorem giving the direct sum decomposition of a tensor product of two irreducible
highest weight modules. These decompositions can both be interpreted as branching rule decompositions
because the tensor product can be considered as a module for the direct sum of two copies of the Lie
algebra, and the branching is with respect to the diagonally embedded copy of the Lie algebra in that
sum of two copies. 

Let $\widehat{\lak}$ be the affine Kac-Moody Lie algebra with underlying finite dimensional simple Lie algebra $\lak$. 
Let $\Lambda'$ and $\Lambda''$ be two dominant integral weights of $\widehat{\lak}$ with levels $m'$ 
and $m''$ respectively. Then the tensor product 
$V^{\Lambda'}\otimes V^{\Lambda''}$ is a $\widehat{\lak}$-module of level $m'+m''$ under the action 
$$x\cdot (v\otimes w) = (x\cdot v)\otimes w + v\otimes(x\cdot w)$$
for any $x\in \widehat{\lak}$, $v\in V^{\Lambda'}$ and $w\in V^{\Lambda''}$. 
From the point of view of branching rules, this tensor product is an irreducible module for 
$\widehat{\lag} = \widehat{\lak} \oplus \widehat{\lak}$ where an ordered pair $(x,y)$ in the direct
sum acts on a basic tensor $v\otimes w$ by 
$(x,y)\cdot (v\otimes w) = (x\cdot v)\otimes w + v\otimes (y\cdot w)$. The diagonally embedded
copy of $\widehat{\lak}$ in $\widehat{\lag}$ acts as shown above on a tensor product. 
We have the direct sum decomposition
$$V^{\Lambda'}\otimes V^{\Lambda''} = \sum_{\Lambda} \mult_{\Lambda',\Lambda''}(\Lambda) V^{\Lambda}$$
where the sum is over weights $\Lambda$ of level $m = m'+m''$ which are less than or equal to 
$\Lambda' + \Lambda''$ in the partial ordering of weights of $\widehat{\lak}$. 
As in the finite dimensional case,
the outer multiplicity of the summand $V^{\Lambda}$ is the dimension of the subspace of tensors of weight
$\Lambda$ which are annihilated by the positive nilpotent subalgebra of $\widehat{\lak}$ acting as above. 
If $\Lambda$ is a weight such that $V^{\Lambda}$ occurs in the tensor product decomposition above, then 
the same is true for all weights $\Lambda-n\delta$ for non-negative integers $n$. 
This means that the summands in the decomposition can be organized into a finite number of ``strings"
since there are only a finite number of dominant integral weights on level $m$ at the top of each string.
We define: $$b_\Lambda^{\Lambda',\Lambda''}(\tau)\coloneqq 
q^{s_{\Lambda'}+s_{\Lambda''}-s_{\Lambda}}\sum_{n\in\Z}\mult_{\Lambda',\Lambda''}(\Lambda-n\delta)q^n.$$

We now look at the specific case where $\lak$ is of type $A_1$. Denote the two level-one fundamental weights of $\widehat{\lak}$ by $\Gamma_0$ and $\Gamma_1$ and define $\Gamma_{m;j}\coloneqq (m-j)\Gamma_0+j\Gamma_1$. Further, let $$\chi_{j+1,k+1}^{(m)}(q)\coloneqq b_{\Gamma_{m+1;k}}^{\Gamma_i,\Gamma_{m;j}}$$ where $i = j-k\text{ (mod 2)}$. Because of the symmetry of the $A_1^{(1)}$ Dynkin diagram, we will get that $\chi_{r,s}^{(m)} = \chi_{m+2-r,m+3-s}^{(m)}$. Thus, all of the $\chi_{r,s}^{(m)}$ are contained in the set $\{\chi_{r,s}^{(m)}\mid 1\leq s\leq r\leq m+1\}$.

In \cite{KacWak2}, Kac and Wakimoto showed that $$\chi_{r,s}^{(m)} = \frac{1}{\eta(\tau)}(f_{(m+2)(m+3),r(m+3)-s(m+2)}-f_{(m+2)(m+3),r(m+3)+s(m+2)})$$ where $$f_{a,b} = \sum_{n\in\Z}q^{a(n+b/(2a))^2}$$ (see formula 4.3). Further, they show that $\chi_{r,s}^{(m)} = \chi(c^{(m)},h_{r,s}^{(m)})$ where 
\begin{equation}
\label{KWc}    c^{(m)} = 1-\frac{6}{(m+2)(m+3)}
\end{equation}
and 
\begin{equation}
\label{KWh}    h_{r,s}^{(m)} = \frac{[r(m+3)-s(m+2)]^2-1}{4(m+2)(m+3)}.
\end{equation}

It should be noted that Kac and Wakimoto actually claim that $$\chi_{r,s}^{(m)} = q^{-c^{(m)}/24}\gr(c^{(m)},h_{r,s}^{(m)}),$$ but because they define $\gr(c,h) = \sum_{n\geq 0}\dim(\vir(c,h)_{h+n})q^{h+n}$ and have $h$ as the eigenvalue of $d_0$ rather than $-d_0$, this is equivalent to what was stated in the first half of this paragraph.

Finally, we state the results of Kac and Wakimoto which are most important for our purposes:

\begin{theorem}[Propostion 4.4.1(d) in \cite{KacWak}]
\label{branch}
\leavevmode
\newline
We have the following branching functions expressed in terms of Virasoro characters:
\begin{enumerate}[(1)]
    \item $b_{\Omega_0}^{\Omega_0} = \chi_{1,1}^{(1)}\chi_{1,1}^{(2)}+\chi_{2,1}^{(1)}\chi_{3,1}^{(2)}$,
    \item $b_{\Omega_0}^{\Omega_4} = \chi_{1,1}^{(1)}\chi_{3,2}^{(2)}+\chi_{2,1}^{(1)}\chi_{3,3}^{(2)}$,
    \item $b_{\Omega_4}^{\Omega_0} = \chi_{2,2}^{(1)}\chi_{2,1}^{(2)}$,
    \item $b_{\Omega_4}^{\Omega_4} = \chi_{2,2}^{(1)}\chi_{2,2}^{(2)}$.
\end{enumerate}
\end{theorem}

\section{Verification of String Functions}
We will begin by stating 4 more identities involving the Rogers-Ramanujan series; entries 3.2, 3.3, 3.20 and 3.21 in \cite{berndt}. These were first proved by Watson.
\begin{proposition}[\cite{watson}]
\label{MoreRR}
\leavevmode
\newline
We have the following idientites:
\begin{enumerate}[(1)]
    \item $G(q)G(q^4)+qH(q)H(q^4) = \dfrac{\phi(q)}{\varphi(q^2)}$,
    \item $G(q)G(q^4)-qH(q)H(q^4) = \dfrac{\phi(q^5)}{\varphi(q^2)}$,
    \item $G(q)H(-q)+G(-q)H(q) = \dfrac{2\psi(q^2)}{\varphi(q^2)}$,
    \item $G(q)H(-q)-G(-q)H(q) = \dfrac{2q\psi(q^{10})}{\varphi(q^2)}.$
\end{enumerate}
\end{proposition}

Calculating $c$ and $h$ for the characters in Theorem \ref{branch} using equations \ref{KWc} and \ref{KWh} and then doing the same with the characters in \ref{QChars} using equations \ref{FFc} and \ref{FFh} allows us to rewrite the branching functions in terms of the right hand sides of the Virasoro characters calculated by Loney. We obtain the following:

\begin{table}[h!]
\centering
\begin{tabular}{|c|c|c|c|}
\hline
$\chi_{p,q}^{(r)}$ & $c$    & $h$    & $\chi_{s,t}^{m,n}$ \\
$p,q,r$            &      &      & $s,t,m,n$          \\ \hline
1,1,1            & 1/2  & 0    & 3,4,1,1          \\
2,1,1            &      & 1/2  & 3,4,1,3          \\
2,2,1            &      & 1/16 & 3,4,1,2          \\ \hline
1,1,2            & 7/10 & 0    & 4,5,1,1          \\
3,1,2            &      & 3/2  & 4,5,1,4          \\
3,3,2            &      & 1/10 & 4,5,1,2          \\
3,2,2            &      & 3/5  & 4,5,1,3          \\
2,1,2            &      & 7/16 & 4,5,2,1          \\
2,2,2            &      & 3/80 & 4,5,2,2          \\ \hline
\end{tabular}
\caption{Virasoro Characters}
\label{table}
\end{table}

\begin{theorem}
Using our shorthand notation from previous chapters, we have:
\begin{enumerate}[(1)]
    \item $c_{\Omega_4}^{\Omega_4} = q^{-7/60}\varphi_1^{-6}\varphi_2^2G_2$,
    \item $c_{\Omega_4}^{\Omega_0} = q^{17/60}\varphi_1^{-6}\varphi_2^2H_2$
\end{enumerate}
\end{theorem}

\begin{proof}
Combining equation \ref{StrBranch}, theorem \ref{branch}(4), table \ref{table} and theorem \ref{QChars} (3) and (9) we have:
\begin{align}
    c_{\Omega_4}^{\Omega_4} &= q^{-1/6}\varphi_1^{-4}(\chi_{3,4}^{1,2}\chi_{4,5}^{2,2})\\
    &= q^{-1/6}\varphi_1^{-4}(q^{1/24}V(q)^{-1})(q^{1/120}G_2V(q)^{-1})\\
    &= q^{-7/60}\varphi_1^{-4}(\varphi_1^{-2}\varphi_2^2)G_2\\
    &= q^{-7/60}\varphi_1^{-6}\varphi_2^2G_2.
\end{align}
Similalrly, combining equation \ref{StrBranch}, theorem \ref{branch}(3), table \ref{table} and theorem \ref{QChars} (3) and (8) we have:
\begin{align}
    c_{\Omega_4}^{\Omega_4} &= q^{-1/6}\varphi_1^{-4}(\chi_{3,4}^{1,2}\chi_{4,5}^{2,1})\\
    &= q^{-1/6}\varphi_1^{-4}(q^{1/24}V(q)^{-1})(q^{49/120}H_2V(q)^{-1})\\
    &= q^{17/60}\varphi_1^{-4}(\varphi_1^{-2}\varphi_2^2)H_2\\
    &= q^{17/60}\varphi_1^{-6}\varphi_2^2H_2.
\end{align}
\end{proof}

\begin{lemma}
\label{eureka}
$\psi(q^2)\phi(q^5)-q\psi(q^{10})\phi(q) = \varphi(q)\varphi(q^5)$.
\end{lemma}

\begin{proof}
We begin by noting that $$\eta_4\eta_{20} = \dfrac{1}{2}\sum_{\substack{x,y\in\Z \\ x\not\equiv y\mod 2}}(-1)^yq^{x^2+5y^2}.$$ This was shown by Hiramatsu, Ishii and Mimura in \cite{japan} (see pages 78-79). We will split this sum up based on whether $x$ is odd or even (and hence $y$ must be the opposite). This gives us:
\begin{align}
    \dfrac{1}{2}\sum_{\substack{x,y\in\Z \\ x\not\equiv y\mod 2}}(-1)^yq^{x^2+5y^2} =& \dfrac{1}{2}\sum_{m,n\in\Z}(-q^{4m^2+20n^2+20n+5}+q^{4m^2+4m+20n^2+1})\\
    =& \frac{1}{2}q\left(\sum_{m,n\in\Z}q^{4m^2+4m+20n^2}-q^4\sum_{m,n\in\Z}q^{4m^2+20n^2+20n}\right)\\
    =& q\bigg[\left(\frac{1}{2}\sum_{m\in\Z}q^{4m^2+4m}\right)\left(\sum_{n\in\Z}q^{20n^2}\right)\nonumber\\
    &-q^4\left(\sum_{m\in\Z}q^{4m^2}\right)\left(\frac{1}{2}\sum_{n\in\Z}q^{20n^2+20n}\right)\bigg]\\
    =& q(\psi(q^8)\phi(q^{20})-q^4\psi(q^{40})\phi(q^4))
\end{align}
where the first summand in line (14) comes from replacing $x$ by $2m$ and $y$ by $2n+1$ and the second summand comes from replacing $x$ by $2m+1$ and $y$ by $2n$. So we have shown that $\eta_4\eta_{20} = q(\psi(q^8)\phi(q^{20})-q^4\psi(q^{40})\phi(q^4))$ which is equivalent to $\eta_1\eta_{5} = q^{1/4}\varphi_1\varphi_5 = q^{1/4}(\psi(q^2)\phi(q^{5})-q\psi(q^{10})\phi(q))$, which simplifies to give the desired result.
\end{proof}

\begin{theorem}
With our notation as before, we have:
\begin{enumerate}[(1)]
    \item $c_{\Omega_0}^{\Omega_0} = q^{-13/60}\varphi_1^{-6}\varphi_\frac{1}{2}^2G_\frac{1}{2}+q^{17/60}\varphi_1^{-6}\varphi_2^2H_2$,
    \item $c_{\Omega_0}^{\Omega_4} = q^{-7/60}\varphi_1^{-6}\varphi_2^2G_2-q^{-7/60}\varphi_1^{-6}\varphi_\frac{1}{2}^2H_\frac{1}{2}$.
\end{enumerate}
\end{theorem}

\begin{proof}
By theorem \ref{branch}(1), table \ref{table} and theorem \ref{QChars} (1), (2), (4) and (5), we have:
\begin{align}
    &b_{\Omega_0}^{\Omega_0} = \chi_{3,4}^{1,1}\chi_{4,5}^{1,1}+\chi_{3,4}^{1,3}\chi_{4,5}^{1,4}\\
    =& q^{-1/48}\dfrac{1}{2}(V(-q^{1/2})+V(q^{1/2}))q^{-7/240}\dfrac{1}{2}(V(-q^{1/2})G(-q^{1/2})+V(q^{1/2})G(q^{1/2}))\nonumber\\
    +&q^{-1/48}\dfrac{1}{2}(V(-q^{1/2})-V(q^{1/2}))q^{-7/240}\dfrac{1}{2}(V(-q^{1/2})G(-q^{1/2})-V(q^{1/2})G(q^{1/2}))\\
    =& \dfrac{1}{2}q^{-1/20}(V(-q^{1/2})^2G(-q^{1/2})+V(q^{1/2})^2G(q^{1/2}))\\
    =& \dfrac{1}{2}q^{-1/20}(\varphi(-q^{1/2})^2\varphi(q)^{-2}G(-q^{1/2})+\varphi(q^{1/2})^2\varphi(q)^{-2}G(q^{1/2}))
\end{align}
and so, by formula \ref{StrBranch}, we have:
\begin{align}
    c_{\Omega_0}^{\Omega_0} &= q^{-1/6}\varphi(q)^{-4}\dfrac{1}{2}q^{-1/20}(\varphi(-q^{1/2})^2\varphi(q)^{-2}G(-q^{1/2})+\varphi(q^{1/2})^2\varphi(q)^{-2}G(q^{1/2}))\\
    &= q^{-13/60}\dfrac{1}{2}\varphi(q)^{-6}(\varphi(-q^{1/2})^2G(-q^{1/2})+\varphi(q^{1/2})^2G(q^{1/2})).
\end{align}
We wish to show that this equals
\begin{align}
    q^{-13/60}&\varphi_1^{-6}\varphi_\frac{1}{2}^2G_\frac{1}{2}+q^{17/60}\varphi_1^{-6}\varphi_2^2H_2\\
    &= q^{-13/60}\varphi_1^{-6}(\varphi_\frac{1}{2}^2G_\frac{1}{2}+q^{1/2}\varphi_2^2H_2).
\end{align}
We will be done if we can show that
\begin{equation}
    \varphi(q^{1/2})^2G(q^{1/2})+q^{1/2}\varphi(q^2)^2H(q^2) = \dfrac{1}{2}(\varphi(-q^{1/2})^2G(-q^{1/2})+\varphi(q^{1/2})^2G(q^{1/2}))
\end{equation}
which is equivalent to
\begin{equation}
\label{FirstID}    \varphi(-q^{1/2})^2G(-q^{1/2})-\varphi(q^{1/2})^2G(q^{1/2}) = 2q^{1/2}\varphi(q^2)^2H(q^2).
\end{equation}
We will next reduce the proof of the second string function in our theorem to a similar identity. By theorem \ref{branch}(2), table \ref{table} and theorem \ref{QChars} (1), (2), (6) and (7), we have:
\begin{align}
    b_{\Omega_4}^{\Omega_0} =& \chi_{3,4}^{1,1}\chi_{4,5}^{1,3}+\chi_{3,4}^{1,3}\chi_{4,5}^{1,2}\\
    =& q^{-1/48}\dfrac{1}{2}(V(-q^{1/2})+V(q^{1/2}))q^{17/240}\dfrac{1}{2}(V(-q^{1/2})H(-q^{1/2})-V(q^{1/2})H(q^{1/2}))\nonumber\\
    &+q^{-1/48}\dfrac{1}{2}(V(-q^{1/2})-V(q^{1/2}))q^{17/240}\dfrac{1}{2}(V(-q^{1/2})H(-q^{1/2})+V(q^{1/2})H(q^{1/2}))\\
    =& \dfrac{1}{2}q^{1/20}(V(-q^{1/2})^2H(-q^{1/2})-V(q^{1/2})^2H(q^{1/2}))\\
    =& \dfrac{1}{2}q^{1/20}(\varphi(-q^{1/2})^2\varphi(q)^{-2}H(-q^{1/2})-\varphi(q^{1/2})^2\varphi(q)^{-2}H(q^{1/2}))
\end{align}
and so, by formula \ref{StrBranch}, we have:
\begin{align}
    c_{\Omega_4}^{\Omega_0} &= q^{-1/6}\varphi(q)^{-4}\dfrac{1}{2}q^{1/20}(\varphi(-q^{1/2})^2\varphi(q)^{-2}H(-q^{1/2})-\varphi(q^{1/2})^2\varphi(q)^{-2}H(q^{1/2}))\\
    &= q^{-7/60}\dfrac{1}{2}\varphi(q)^{-6}(\varphi(-q^{1/2})^2H(-q^{1/2})-\varphi(q^{1/2})^2H(q^{1/2})).
\end{align}
We wish to show that this equals
\begin{align}
    q^{-7/60}&\varphi_1^{-6}(\varphi_2^2G_2-q^{-7/60}\varphi_1^{-6}\varphi_\frac{1}{2}^2H_\frac{1}{2}\\
    &= q^{-7/60}\varphi_1^{-6}(\varphi_2^2G_2-\varphi_\frac{1}{2}^2H_\frac{1}{2})
\end{align}
And so we will be done if we can show that
\begin{equation}
    \varphi_2^2G_2-\varphi_\frac{1}{2}^2H_\frac{1}{2} = \dfrac{1}{2}\varphi(q)^{-6}(\varphi(-q^{1/2})^2H(-q^{1/2})-\varphi(q^{1/2})^2H(q^{1/2}))
\end{equation}
which is equivalent to
\begin{equation}
\label{SecID}    \varphi(-q^{1/2})^2H(-q^{1/2})+\varphi(q^{1/2})^2H(q^{1/2}) = 2\varphi(q^2)^2G(q^2).
\end{equation}
Making the substitution $t = q^{1/2}$, equations \ref{FirstID} and \ref{SecID} become:
\begin{equation}
\label{1ID}    \varphi(-t)^2G(-t)-\varphi(t)^2G(t) = 2t\varphi(t^4)^2H(t^4)
\end{equation} and
\begin{equation}
\label{2ID}    \varphi(-t)^2H(-t)+\varphi(t)^2H(t) = 2\varphi(t^4)^2G(t^4).
\end{equation}

Then, multiplying both sides of \ref{1ID} by $H(t)$ and both sides of \ref{2ID} by $G(t)$, we get:
\begin{equation}
\label{3ID}    \varphi(-t)^2G(-t)H(t)-\varphi(t)^2G(t)H(t) = 2t\varphi(t^4)^2H(t^4)H(t)
\end{equation} and
\begin{equation}
\label{4ID}    \varphi(-t)^2G(t)H(-t)+\varphi(t)^2G(t)H(t) = 2\varphi(t^4)^2G(t^4)G(t).
\end{equation}
Let $A(t)$ equal the left hand side of equation \ref{3ID} and $B(t)$ equal the left hand side of \ref{4ID}. We then have the following:
\begin{align}
    A(t)+B(t) &= \varphi(-t)^2(G(-t)H(t)+G(t)H(-t))\\
    &= \varphi(-t)^2\frac{2\psi(t^2)}{\varphi(t^2)}\\
    &= \varphi(t^2)\phi(t)\frac{2\varphi(t^4)^2/\varphi(t^2)}{\varphi(t^2)}\\
    &= \frac{\phi(t)}{\varphi(t^2)}2\varphi(t^4)^2\\
    &= 2\varphi(t^4)^2(G(t)G(t^4)+tH(t)H(t^4))
\end{align}
where line (43) follows from (42) by proposition \ref{MoreRR} (3), (44) follows from line (43) since $\psi(q) = \varphi_2^2/\varphi_1$ and $$\frac{\varphi(-q)}{\varphi(q^2)} = \frac{\phi(q)}{\varphi(-q)}$$ (this is equation 2.14 in \cite{berndt}; for a proof, see Entry 24(iii) on page 39 of \cite{RamNB}), which implies that $\varphi(-q)^2 = \varphi(q^2)\phi(q)$. Then, line (46) follows from (45) by proposition \ref{MoreRR} (1). Similarly, we will get
\begin{align}
    B(t)-A(t) &= \varphi(-t)^2(G(t)H(-t)-G(-t)H(t))+2\varphi(t)^2G(t)H(t)\\
    &= \varphi(-t^2)\dfrac{2t\psi(t^{10})}{\varphi(t^2)}+2\varphi(t)\varphi(t^5)\\
    &= 2t\phi(t)\psi(t^{10})+2\varphi(t)\varphi(t^5)\\
    &= 2\psi(t^2)\phi(t^5)\\
    &= 2\varphi(t^4)^2\frac{\phi(t^5)}{\varphi(t^2)}\\
    &= 2\varphi(t^4)^2(G(t)G(t^4)-tH(t)H(t^4))
\end{align}
where line (48) follows from line (47) by proposition \ref{MoreRR} (4), (49) follows from (48) again because $\varphi(-q)^2 = \varphi(q^2)\phi(q)$, (50) follows from (49) by lemma \ref{eureka}, (51) follows from (50) because $\psi(q) = \varphi_2^2/\varphi_1$ and (52) follows from (51) by proposition \ref{MoreRR} (2).

Now, using lines (46) and (52), we see that:
\begin{align}
    A(t) &= \frac{1}{2}((A(t)+B(t))-(B(t)-A(t)))\\
    &= \varphi(t^4)^2(2tH(t)H(t^4))\\
    &= 2t\varphi(t^4)^2H(t^4)H(t)
\end{align}
as desired. Similarly,
\begin{align}
    B(t) &= \frac{1}{2}((A(t)+B(t))+(B(t)-A(t)))\\
    &= \varphi(t^4)^2(2tG(t)G(t^4))\\
    &= 2t\varphi(t^4)^2G(t^4)G(t)
\end{align}
which was what we wanted. This completes our proof.
\end{proof}

\addcontentsline{toc}{chapter}{Conclusion}
\chapter*{Conclusion}

In this dissertation, we were able to prove the branching rule decomposition of the level-1, irreducible, highest weight $E_8^{(1)}$-module with respect to the subalgebra $F_4^{(1)}\oplus G_2^{(1)}$. As mentioned previously, this result is stated in papers by Kac and Sanielivici \cite{KacSan}, Kac and Wakimoto \cite{KacWak} and Bernard and Theirry-Mieg \cite{BTM}. In both papers by Kac and coauthors, the decomposition is stated without any explicit proof, but the implication is that the theory of branching functions is used (and hence Virasoro theory is at work at least in the background). In our calculation, we exclusively used string functions without need for any (direct) Virasoro theory. Although we used some Virasoro theory to verify the $F_4^{(1)}$ string functions, these can also be verified using the theory of modular forms. Bernard and Thierry-Mieg verified the $G_2^{(1)}$ string functions via explicit module constructions. In their paper, Bernard and Thierry-Mieg only briefly mention our result at the very end and state that this decomposition implies many theta function identities. Although it is hard to know what types of identities they had in mind, many beautiful identities appeared in this dissertation in support of the desired decomposition.

As we just mentioned, we proved our result exclusively using character theory via the Kac-Peterson theta function formula and this calculation could conceivably be shown using the theory of branching functions. Another possibility is via explicit module constructions using the theory of vertex operators. This future direction would not only serve to give a more concrete proof of the result, but also give representation-theoretic meaning to many of the identities proven here.

\end{doublespacing}

\addcontentsline{toc}{chapter}{Bibliography} 

\printbibliography

\end{document}